\documentclass[11pt,reqno]{amsart}
\usepackage{mathrsfs}
\usepackage{url}
\usepackage{mathtools}
\usepackage{latexsym,epsfig,amssymb,amsmath,amsthm,color,url,bm}
\usepackage[inline,shortlabels]{enumitem}
\usepackage{hyperref}
\usepackage[foot]{amsaddr}
\usepackage{amsmath,amsbsy}
\usepackage{cleveref}
\usepackage{mwe}
\RequirePackage[numbers]{natbib}
\usepackage{mathptmx}
\usepackage[text={16cm,24cm}]{geometry}
\usepackage{xcolor}

\allowdisplaybreaks 
 \numberwithin{equation}{section}
\newtheorem{theorem}{Theorem}[section]
\newtheorem{prop}[theorem]{Proposition}
\newtheorem{lemma}[theorem]{Lemma}

\newcommand\Item[1][]{%
  \ifx\relax#1\relax  \item \else \item[#1] \fi
  \abovedisplayskip=0pt\abovedisplayshortskip=0pt~\vspace*{-\baselineskip}}

\theoremstyle{definition}

\theoremstyle{definition}
\newtheorem{remark}[theorem]{Remark}

\newcommand{\convas}{\stackrel{\text{a.s.}}{\longrightarrow}}
\newcommand{\convd}{\stackrel{\text{D}}{\longrightarrow}}

\DeclareMathOperator{\Prob}{\mathbf{P}}
\DeclareMathOperator{\E}{\mathbf{E}}

\DeclareMathOperator{\Ber}{Bernoulli}

\title[]{Elephant random walk with attributed steps and extractions of random sizes}
\date{}
\author{Sooraj M,\ Moumanti Podder and Archi Roy}
\address{Sooraj M,\ Indian Institute of Science Education and Research (IISER) Pune, Dr.\ Homi Bhabha Road, Pashan, Pune 411008, Maharashtra, India.}
\address{Moumanti Podder, Indian Institute of Science Education and Research (IISER) Pune, Dr.\ Homi Bhabha Road, Pashan, Pune 411008, Maharashtra, India.}
\address{Archi Roy, Indian Institute of Management, Kozhikode 673570, Kerala, INDIA.}
\email{sooraj.m@students.iiserpune.ac.in}
\email{moumanti@iiserpune.ac.in}
\email{archiroy@iimk.ac.in}

\begin{document}
\bibliographystyle{plainnat}

\begin{abstract}
We study a model of market economics wherein the $(n+1)$-st customer, for each $n\geqslant N$, with $N$ being a prespecified positive integer, draws a sample of (random) size $K_{n}$, either with replacement or without, from the customers of the past. Each sampled customer is queried as to which of the two products, A and B, available in the oligopolistic market, they chose, and whether they are satisfied or not with their choice. The $(n+1)$-st customer now employs a stochastic rule, based on the information collected from the sampled customers, to decide which of the two products to buy. The probability that a customer is satisfied with the product they have purchased equals $q_{1}$ when the product is A, and $q_{2}$ when it is B, independent of all else. The resulting stochastic process may be represented as a variant of the celebrated elephant random walk, with the relative performance (in terms of sale) of A with respect to B, up to and including the $n$-th sale, captured by the position $S_{n}$ of the walker at time $n$. We study the almost sure convergence of $S_{n}/n$, as well as the convergence in distribution of suitably scaled versions of $S_{n}$ (where the scaling depends on the regime we are in).   
\end{abstract}

\subjclass[2020]{}


\maketitle
\section{Introduction}\label{sec:intro}
Consider two competing products, A and B, which have been released to the market around the same time. Motivated by the theory that persistent differences in their performances can emerge purely from the (often non-deterministic) accumulation of feedback, even when all agents are initially identical \citep{denrell2004random}, we study the long-term behaviour of the sale of one of the products relative to the other (including the possibility of the eventual dominance of one product over the other). We assume that at each discrete time-point, henceforth indexed by the set $\mathbb{N}$ of positive integers, a single customer arrives and chooses exactly one of the two products. We denote by $t=0$ the time of release of both the products into the market, and we let $t=n$ represent the time at which the $n$-th sale occurs. We fix $N\in\mathbb{N}$, and we let each of the first $N$ customers, independent of all else, select product A with probability $q$ and product B with probability $(1-q)$ (we remark here that, even if we let the first $N$ customers make their choices in an arbitrary manner, without assumptions of independence or of identical distributions for the choices made, the conclusions drawn in this paper remain unaffected, as all of our results concern themselves with the \emph{asymptotic} behaviour of the stochastic process our model gives rise to). As the market evolves, the $(n+1)$-st customer, for each $n\geqslant N$, bases their decision on a randomly selected sample of past customers, where the size of the sample itself is a random variable $K_{n}$, taking into account both the choices that these sampled customers made and the satisfaction levels that they report. The samples may be drawn either with replacement or without. Each past customer included in the sample drawn is asked to provide a qualitative assessment of their purchase, such as a rating on a fixed scale, which is interpreted by the current, i.e.\ the $(n+1)$-st customer as either positive or negative. For instance, one may consider the quite commonly practised set-up where products receive ratings on a five-point Likert scale -- in such cases, a customer may be classified as `satisfied' if their rating is three stars or higher (out of five), and `dissatisfied' otherwise. Consequently, every historical observation consists of two components: the product chosen and the associated opinion. These observations are aggregated through a smooth reinforcement mechanism that governs the probability with which the current, i.e.\ the $(n+1)$-st customer selects a product. We further consider scenarios where the size $K_{n}$ of the sample drawn by the $(n+1)$-st customer does not conform to a fixed probability distribution, thereby requiring us to work with a sequence of probability distributions $\{\mu_{n}\}$, with $\mu_{n}$ supported on the set $\{1,2,\ldots,n\}$ or some subset thereof. A more detailed formulation of the mathematical model is presented in \Cref{sec:model}. 

We study the relative market dominance process by maintaining a chart which registers a $+1$ whenever product A is purchased, and a $-1$ whenever product B is purchased, so that the sum of all the (signed) tally marks up to and including epoch $n$ indicates the relative sale of product A with respect to product B. This construction naturally leads to a random walk representation, wherein the choice made by the $n$-th customer can be interpreted as the $n$-th step size, denoted by $X_n$ (so that $X_{n}=+1$ if the $n$-th customer selects product A, and $X_{n}=-1$ otherwise), of the walk, and the position of the walker at time $n$, denoted by $S_n=\sum_{i=1}^{n}{X_i}$, encodes the relative market performance of A with respect to B. In this paper, we specifically focus on the strong and weak convergence properties of (suitably scaled versions of) $S_n$. Notably, in contrast to simple random walk models in which each step of the walker is independent of all of its past steps, each step of the walker in our model is heavily influenced by the entirety of its past so far. More specifically, the current step of the walker in our model depends on a finite collection (of random cardinality) of randomly chosen steps from the past, and a certain attribute associated with each step (in our market competition example, this is the opinion of approval / disapproval formed after each purchase). This induces a path-dependence in the random walk, which is reminiscent of the memory-reinforced \emph{elephant random walk model}. Following the terminology of \cite{franchini2025elephant}, we refer to this structure as an elephant random walk with random extractions of random cardinalities, where each extraction further carries with it an associated attribute. At this stage, it is worth noting that a simpler version of this model was studied in \cite{podder2026elephant}, where the sizes of the samples extracted were deterministic (whether constant or varying with $n$) and the steps carried no associated attributes. Allowing samples sizes to be random and incorporating customer-specific binary responses introduce greater analytical challenges, which we have addressed in the subsequent sections.

We further emphasize that, although this work is primarily motivated by the market competition example, the proposed structure can be viewed as a \emph{generalized urn process}. Urn models have a wide range of applicability across diverse fields, including stochastic gene expression \citep{choudhary2020urn}, the dynamics of finite populations of interacting genotypes \citep{benaim2004generalized}, financial modeling \citep{hisakado2025urn}, and the modeling of molecular diffusion \citep{casas2015nonlinear}, to name just a few. It is possible to view our elephant random walk as a generalized urn process with balls of four colours, the first of which corresponds to all those customers who have purchased product A and are satisfied with it, the second of which corresponds to all those customers who have purchased product B and are satisfied with it, and so on. At each step, a random number of balls is chosen uniformly at random from the urn, either with or without replacement, and their colours are recorded. All the sampled balls are then returned to the urn, along with a single \emph{new} ball whose colour is decided according to a probability distribution constructed from the \emph{sampled colour counts} (i.e.\ for each of the four colours, the number of balls of that colour which appear in the sample). We are interested in tracking the evolution of the number of balls of each colour in the urn over time, since a suitable linear combination of these numbers (up to and including the $n$-th step of the urn process) equals $S_n=\sum_{i=1}^{n}{X_i}$. Notably, the existing setup can also be extended to a \textit{lazy} version of the process, in which balls of a fifth colour are considered (corresponding to customers who choose not to purchase any product), or to a \textit{fast-growing} version, in which multiple balls of the same color are added in a single step (equivalently, an incoming customer purchases multiple copies of the same product). While such extensions make tracking the process $\{S_n\}$ considerably more challenging, they also emulate real-world situations more closely, and are, therefore, questions worth pursuing for future research.

\subsection{Organization of the paper} The rest of the paper has been laid out as follows. The model has been formally introduced in \S\ref{sec:model}. A detailed discussion of the existing literature, that is pertinent to the work done in this paper, has been provided in \S\ref{sec:lit_review}. Important definitions and notations have been detailed in \S\ref{sec:notations_definitions} for the reader's convenience. The main results of the paper have been stated in \S\ref{sec:main_results}, with \S\ref{subsec:main_results_strong} dedicated to results pertaining to convergence almost surely, and \S\ref{subsec:main_results_weak} dedicated to results pertaining to convergence in distribution. A few important results from the literature, crucially implemented in proving our own results, have been included in \S\ref{sec:lit_results_useful}. The proofs of all of the results stated in \S\ref{sec:main_results} have been detailed in \S\ref{sec:proofs}, and finally, the Appendix, \S\ref{sec:appendix}, presents several technical lemmas that have been utilized in the proofs of our main results. 

\section{A formal description of the model}\label{sec:model}
For a fixed positive integer $N$, let $X_{1}, X_{2}, \ldots, X_{N}$ be i.i.d.\ Rademacher$(q)$ random variables, i.e.\
\begin{equation}
X_{i}=+1 \text{ with probability } q \quad \text{and} \quad X_{i}=-1 \text{ with probability } (1-q)\nonumber
\end{equation}
for all $i\in[N]$ (henceforth, for any $n\in\mathbb{N}$, we denote by $[n]$ the set $\{1,2,\ldots,n\}$). We also consider parameters $q_{1}, q_{2} \in (0,1)$, and independent random variables $\epsilon_{1}, \epsilon_{2}, \ldots, \epsilon_{N}$, where, for each $i\in[N]$,
\begin{equation}\label{first_few_opinions}
\begin{cases}
{}&\epsilon_{i} \sim \Ber(q_{1}) \text{ conditioned on } X_{i}=+1,\\
{}&\epsilon_{i} \sim \Ber(q_{2}) \text{ conditioned on } X_{i}=-1.
\end{cases}
\end{equation}
So far, we may interpret the set-up as follows: each of the first $N$ customers, due to lack of available data or of awareness regarding the underlying quality of each product, chooses, independent of all else, product A with probability $q$ and product B with probability $(1-q)$. A practical choice for the value of $q$ may be $1/2$, since these initial customers, unaware of the relative performane of one product with respect to another, may make their choices without any bias. The event $\epsilon_{i}=1$ indicates that the $i$-th customer is satisfied with their choice, whereas $\epsilon_{i}=0$ indicates that they are discontented. Provided that a customer selected product A, their probability of feeling satisfied is assumed to equal $q_{1}$, and provided that they chose product B, their probability of feeling satisfied is assumed to equal $q_{2}$. 

We remark here that, if the product qualities are assumed to remain consistent over time, it is not unreasonable to assume that the percentage of customers, out of all those who have chosen to buy product A (respectively, B), who are happy with their purchase, stabilizes over time. Consequently, it is not outlandish to assume (as we do in the model studied in this paper) that, conditioned on purchasing product A (respectively, B), the probability of the customer feeling satisfied equals $q_{1}$ (respectively, $q_{2})$. A somewhat more realistic version of this model would allow the $n$-th customer to be satisfied with probability $q_{1,n}$ if the product purchased is A, and with probability $q_{2,n}$ if the product purchased is B, where each of the sequences $\{q_{1,n}\}$ and $\{q_{2,n}\}$ converges (to, say, $q_{1}$ and $q_{2}$ respectively). The analysis for this slightly more generalized version may be carried out in an analogous manner, with only marginally more technical complications.

We now consider $n\in\mathbb{N}$ with $n\geqslant N$, and we assume that the stochastic process $\{(X_{i},\epsilon_{i}):i\in[n]\}$ has already been defined. At the beginning of the $(n+1)$-st epoch, a sample of size $K_{n}$ is drawn, either with replacement or without, from the set $[n]$, where $K_{n}$ is a random variable whose distribution is independent of the stochastic process so far, and which is supported  
\begin{enumerate*}
\item on some finite subset of $\mathbb{N}$ when the sampling scheme is with replacement,
\item and on the set $[n]$, or a subset thereof, when the sampling scheme is without replacement.
\end{enumerate*}
This is equivalent to drawing a sample of size $K_{n}$ from the set of all past customers. As also mentioned in \S\ref{sec:intro}, we consider both 
\begin{enumerate*}
\item the scenario where $K_{n}$ are independent and identically distributed over all $n\geqslant N$, with their common law denoted by $\mu$,
\item and the scenario where $K_{n}$ are independent over all $n\geqslant N$, but the law $\mu_{n}$ of $K_{n}$ varies with $n$ (with suitable assumptions imposed on the sequence $\{\mu_{n}\}$ of probability distributions).
\end{enumerate*}
Once a realization $k$ of $K_{n}$ has been obtained, let $U_{n,1}, U_{n,2}, \ldots, U_{n,k}$ indicate the indices of the customers included in the sample drawn, so that
\begin{equation}
\Prob\left[U_{n,1}=i_{1},\ldots,U_{n,k}=i_{k}\Big|K_{n}=k\right]=\frac{1}{n^{k}} \text{ for all } \left(i_{1},i_{2},\ldots,i_{k}\right)\in[n]^{k}\label{sample_indices_distribution_with_replacement}
\end{equation}
when the sample is drawn with replacement, and
\begin{equation}
\Prob\left[U_{n,1}=i_{1},\ldots,U_{n,k}=i_{k}\Big|K_{n}=k\right]=\left\{{n\choose k}k!\right\}^{-1} \text{ for all distinct } i_{1},i_{2},\ldots,i_{k}\in[n]\label{sample_indices_distribution_without_replacement}
\end{equation}
when the sample is drawn without replacement. Let $\hat{A}_{n}$ denote the number of satisfied customers who chose product A and who have been included in this sample, and let $\hat{B}_{n}$ denote the number of satisfied customers who chose product B and who have been included in this sample. Formally, 
\begin{equation}
\hat{A}_{n}=\sum_{i=1}^{K_{n}}\chi\left\{X_{U_{n,i}}=+1,\epsilon_{U_{n,i}}=1\right\} \quad \text{and} \quad \hat{B}_{n}=\sum_{i=1}^{K_{n}}\chi\left\{X_{U_{n,i}}=-1,\epsilon_{U_{n,i}}=1\right\},\label{sample_count_defns}
\end{equation}
where $\chi\{A\}$ is the indicator random variable for the event $A$, defined with respect to some sample space. Defining the set 
\begin{equation}
\mathcal{S}=\left\{(x,y)\in[0,1]^{2}:x+y\leqslant 1\right\},\label{domain_defn}
\end{equation}
letting $F:\mathcal{S}\rightarrow[0,1]$ be a function (often referred to as the \emph{reinforcement function} or the \emph{replacement function}, when these models are studied from the perspective of urn processes), and letting $p\in[0,1]$ be a given constant, we now set, conditioned on $K_{n}=k$,
\begin{equation}\label{X_{n+1}_distribution}
X_{n+1}=
\begin{cases}
+1 &\text{with probability } p F\left(k^{-1}\hat{A}_{n},k^{-1}\hat{B}_{n}\right)+(1-p)\left\{1-F\left(k^{-1}\hat{A}_{n},k^{-1}\hat{B}_{n}\right)\right\},\\
-1 &\text{with probability } (1-p)F\left(k^{-1}\hat{A}_{n},k^{-1}\hat{B}_{n}\right)+p\left\{1-F\left(k^{-1}\hat{A}_{n},k^{-1}\hat{B}_{n}\right)\right\}.
\end{cases}
\end{equation} 
Conditioned on $X_{n+1}=+1$, we let $\epsilon_{n+1} \sim \Ber(q_{1})$, and conditioned on $X_{n+1}=-1$, we let $\epsilon_{n+1} \sim \Ber(q_{2})$, where $q_{1}, q_{2}\in(0,1)$ are the same as those mentioned in \eqref{first_few_opinions}. It is worthwhile to note here that $p=1/2$ yields the simple symmetric random walk.

As alluded to in the last paragraph of \S\ref{sec:intro}, the above-mentioned model may also be described as a \emph{balanced} urn process (which indicates that during epoch $n$, for each $n\in\mathbb{N}$, the urn contains precisely $n$ balls in total) involving labeled (or numbered) balls of four different colours. During epoch $n$, there are balls numbered $1,2,\ldots,n$ present in the urn, and each of them possesses one of four possible colours. Let $A_{n}$, $B_{n}$, $C_{n}$ and $D_{n}=n-A_{n}-B_{n}-C_{n}$ indicate, respectively, the counts of balls of colours $1$, $2$, $3$ and $4$, in the urn at the end of epoch $n$ -- in other words, the \emph{configuration} of the urn at the end of epoch $n$ is given by $(A_{n},B_{n},C_{n},D_{n})$. One may interpret 
\begin{enumerate*}
\item balls of colour $1$ to correspond to customers who are satisfied with product A (which is what they chose to purchase), 
\item balls of colour $2$ to correspond to customers who are satisfied with product B, 
\item balls of colour $3$ to correspond to customers who are discontented with product A, and
\item balls of colour $4$ to correspond to customers who are discontented with product B.
\end{enumerate*}
For each $n\geqslant N$, at the beginning of epoch $(n+1)$, a realization $k$ of the random variable $K_{n}$ is generated, and a sample of size $k$ is drawn, uniformly at random, either with replacement or without, from the balls present in the urn. Letting $\hat{A}_{n}$ denote the number of balls of colour $1$ and $\hat{B}_{n}$ the number of balls of colour $2$ present in the sample drawn, we return the sampled balls to the urn, and add to it a new ball, whose colour equals
\begin{align}
{}&1 \text{ with probability } q_{1}\left[p F\left(k^{-1}\hat{A}_{n},k^{-1}\hat{B}_{n}\right)+(1-p)\left\{1-F\left(k^{-1}\hat{A}_{n},k^{-1}\hat{B}_{n}\right)\right\}\right],\label{A_{n+1}_distribution}\\
{}&2 \text{ with probability } q_{2}\left[(1-p)F\left(k^{-1}\hat{A}_{n},k^{-1}\hat{B}_{n}\right)+p\left\{1-F\left(k^{-1}\hat{A}_{n},k^{-1}\hat{B}_{n}\right)\right\}\right],\label{B_{n+1}_distribution}\\
{}&3 \text{ with probability } (1-q_{1})\left[p F\left(k^{-1}\hat{A}_{n},k^{-1}\hat{B}_{n}\right)+(1-p)\left\{1-F\left(k^{-1}\hat{A}_{n},k^{-1}\hat{B}_{n}\right)\right\}\right],\label{C_{n+1}_distribution}\\
{}&4 \text{ with probability } (1-q_{2})\left[(1-p)F\left(k^{-1}\hat{A}_{n},k^{-1}\hat{B}_{n}\right)+p\left\{1-F\left(k^{-1}\hat{A}_{n},k^{-1}\hat{B}_{n}\right)\right\}\right].\label{D_{n+1}_distribution}
\end{align}
Note that, when identifying this urn process with the random walk model described earlier, the position $S_{n}$ of the walker, right after epoch $n$, is given by $S_{n}=(A_{n}+C_{n})-(B_{n}+D_{n})=2(A_{n}+C_{n})-n$, so that understanding the asymptotics of $(A_{n}/n,C_{n}/n)$ (be it the almost sure convergence of $(A_{n}/n,C_{n}/n)$ or the convergence in distribution of a suitably scaled version of this pair) helps shed light on the asymptotics of $S_{n}$ as well. It is also inherently interesting to investigate the asymptotics of $(A_{n}/n,B_{n}/n)$, as this helps us compare the proportion of customers satisfied with product A with the proportion of customers satisfied with product B.

In each of the results deduced in this paper, we shall be in one of the following two scenarios:
\begin{enumerate}[label=(A\arabic*), ref=A\arabic*]
\item \label{Scenario_1} The size $K_{n}$ of the sample drawn by the $(n+1)$-st customer, for each $n\geqslant N$, follows a \emph{common} law $\mu$, assumed to be supported on $[M]$ for some $M\in\mathbb{N}$. Let $\mu(k)=\Prob[K_{n}=k]$ for $k\in[M]$.
\item \label{Scenario_2} The sample size $K_{n}$ follows the law $\mu_{n}$ for each $n\geqslant N$, with $\mu_{n}$ assumed to be supported on some finite interval $I_{n}$ of $\mathbb{N}$ -- when the samples are drawn without replacement, we additionally assume that $I_{n}$ is a subset of $[n]$, for each $n\geqslant N$. Let $\mu_{n}(k)=\Prob[K_{n}=k]$ for each $k\in I_{n}$.
\end{enumerate}
In each of these two scenarios, we assume $\{K_{n}:n\in\mathbb{N},n\geqslant N\}$ to form an independent sequence of random variables. Moreover, $K_{n}$ is assumed to be independent of the stochastic process up to and including epoch $n$ -- formally speaking, $K_{n}$ is independent of $\mathcal{F}_{n}$, where $\mathcal{F}_{n}$ is the $\sigma$-field consisting of all information pertaining to the stochastic process up to and including epoch $n$.  

We mention here two kinds of random experiments that show up in our arguments throughout this paper. One of these is useful when the samples are drawn with replacement, and the other is useful when the samples are drawn without replacement.
\begin{enumerate}[label=(Exp \arabic*), ref=Exp \arabic*]
\item \label{Exp_1} Consider a random experiment, each trial of which may result in one of three possible outcomes: outcome $1$ transpires with probability $x$, outcome $2$ with probability $y$, and outcome $3$ with probability $(1-x-y)$. Fix $n\in\mathbb{N}$ with $n\geqslant N$. If $K_{n}$ i.i.d.\ trials of this experiment are performed, let $V_{i}$ indicate the number of trials that result in outcome $i$, for each $i\in\{1,2,3\}$. Going forward, in various places, we consider $\E[g(V_{1}/K_{n},V_{2}/K_{n})]$, where $g$ is as defined in \eqref{g_defn}.

It is worthwhile to note here that, conditioned on $K_{n}=k$, the random variable $V_{1}$ follows Binomial$(k,x)$, while $V_{2}$ follows Binomial$(k,y)$. Consequently, we have
\begin{equation}\label{exp:1_expectations}
\E\left[\frac{V_{1}}{K_{n}}-x\Big|K_{n}=k\right]=\E\left[\frac{V_{2}}{K_{n}}-y\Big|K_{n}=k\right]=0,
\end{equation}
and
\begin{equation}\label{exp:1_variances}
\E\left[\left(\frac{V_{1}}{K_{n}}-x\right)^{2}\Big|K_{n}=k\right]=\frac{x(1-x)}{k} \quad \text{and} \quad \E\left[\left(\frac{V_{2}}{K_{n}}-y\right)^{2}\Big|K_{n}=k\right]=\frac{y(1-y)}{k}.
\end{equation}

\item \label{Exp_2} Fix $n\in\mathbb{N}$ with $n\geqslant N$. Let an urn contain $r_{1}$ balls of colour $1$, $r_{2}$ balls of colour $2$, and $(n-r_{1}-r_{2})$ balls of colour $3$, so that $(r_{1}/n,r_{2}/n)\in \mathcal{S}_{n}$, where
\begin{equation}
\mathcal{S}_{n}=\left\{(m_{1}/n,m_{2}/n):m_{1},m_{2}\in\mathbb{N}_{0},m_{1}+m_{2}\leqslant n\right\}.\label{S_{n}_defn}
\end{equation}
Suppose a random sample, of random size $K_{n}$, is now drawn, without replacement, from the urn. Let $V_{i}$ indicate the number of balls of colour $i$ that appear in this sample, for each $i\in\{1,2,3\}$. Once again, going forward, the object of interest will be $\E[g(V_{1}/K_{n},V_{2}/K_{n})]$. 

Conditioned on $K_{n}=k$, the random variable $V_{i}$, for each $i\in\{1,2\}$, follows a Hypergeometric distribution in which the total population size is $n$, the number of items corresponding to `success' present in the population is $r_{i}$, and the sample drawn, without replacement, is of size $k$, so that
\begin{equation}\label{exp:2_expectations_variances}
\E\left[\frac{V_{i}}{K_{n}}-\frac{r_{i}}{n}\Big|K_{n}=k\right]=0 \text{ and } \E\left[\left(\frac{V_{i}}{K_{n}}-\frac{r_{i}}{n}\right)^{2}\Big|K_{n}=k\right]=\frac{r_{i}(n-r_{i})(n-k)}{k(n-1)n^{2}} \text{ for each }i\in\{1,2\}.
\end{equation}
\end{enumerate}

\section{Literature review}\label{sec:lit_review}

In this section, we briefly review the recent developments in the literature related to the elephant random walk model. The original formulation of the elephant random walk, introduced in \cite{schutz2004elephants}, is a one-dimensional stochastic process $\{S_n\}_{n \in \mathbb{Z}_+}$ evolving on the integer lattice $\mathbb{Z}$. The walker starts at an initial position $S_0$, typically taken to be $0$, at time $t=0$. At time $t=1$, the walker moves one unit to the right with probability $q$ and one unit to the left with probability $(1-q)$, where $q \in (0,1)$ is fixed. From its second step onward, the dynamics of the walker are governed by its memory of the entire past trajectory. Specifically, at time $t=(n+1)$, for each $n \in \mathbb{N}$, the walker makes a unit displacement that is equal to the random variable $X_{n+1}$, and its position, right after $t=(n+1)$, is given by $S_{n+1} = S_n + X_{n+1}$, where the increment $X_{n+1} \in \{+1,-1\}$ is defined recursively as follows: a past time index $U_n$ is selected uniformly at random from $[n]$, following which we set
\begin{align*}
    X_{n+1} =
\begin{cases}
 X_{U_n}, & \text{with probability } p, \\
- X_{U_n}, & \text{with probability } 1-p,
\end{cases}
\end{align*}
where $p \in [0,1]$ is referred to as the \emph{memory parameter} or the \emph{self-excitation parameter}. 

In the existing literature, several extensions of this classical model have been proposed, each focusing on a particular structural aspect of the walk. For instance, works such as \cite{gut2021variations,gut2022elephant,dhillon2025elephantrandomwalkrandom,nakano2025limit} study variants of the model in which the walker retains the memory of its past only partially (of deterministic or random length), rather than recalling its entire history, while \cite{dedecker2023rates,fan2024cramer,nakano2025elephant} consider versions where the walk evolves via step-sizes that are not necessarily equal to $\pm1$. An interesting direction is to consider the walk on domains beyond the real line. For instance, \cite{laulin2025elephants} studies a rotational variant of the walk on the complex plane $\mathbb{C}$, \cite{shibata2025functional} analyses the walk on higher-dimensional integer lattice structures such as triangular and hexagonal lattices, and \cite{mukherjee2025elephant} investigates the evolution of the classical walk on an infinite Cayley tree. Multidimensional extensions of the elephant random walk have also received considerable attention \citep{bercu2025multidimensionalelephantrandomwalk, bertenghi2022functional, marquioni2019multidimensional,maulik2024asymptotic,ghosh2026limiting}. Other significant variants of the walk can be found in works such as \citep{bercu2022elephant,harbola2014memory,harris2015random,kim2014anomalous,kursten2016random,roy2024elephant, serva2013scaling}. We note that the work of \cite{franchini2025elephant} is closely related to ours, as it also considers multiple extractions; however, memory in that model is incorporated through a majority rule, whereby the walker follows the dominant pattern in the sampled steps. In contrast, our framework allows for a general reinforcement function. Moreover, we build upon the setup introduced in \cite{podder2026elephant} and extend it to a substantially more general and mathematically involved modeling structure.

The long-term behavior of the elephant random walk has been extensively studied \citep{boyer2014solvable, coletti2017central, cressoni2013exact, kumar2010memory}. Within this line of research, \cite{da2013non} provides compelling numerical evidence showing that, for certain ranges of the memory parameter $p$, the displacement distribution deviates from both Gaussian and L\'evy limits. A major methodological advancement has been achieved in \cite{bercu2017martingale} and further developed in \cite{laulin2022new}, where a martingale-based framework has been employed to establish asymptotic results. In particular, weak convergence and large deviation properties have been obtained via the martingale central limit theorem and the law of the iterated logarithm, respectively. This martingale methodology has since been adopted in several subsequent works, including \cite{bercu2021center}, \cite{coletti2017central}, and \cite{roy2024phase}.

In the present paper, we draw inspiration from related studies such as \cite{das2024elephant, maulik2024asymptotic, podder2026elephant}, and instead develop our asymptotic theory using tools from stochastic approximation \citep{duflo2013random}. Stochastic approximation (SA) was first introduced by \cite{robbins1951stochastic} as an iterative procedure for locating the zero of a regression function through successive updates. Since then, SA techniques have been widely adopted across a broad range of scientific disciplines (see \cite{kushner2010stochastic} and \cite{lai2003stochastic} for comprehensive surveys). Early developments in SA primarily addressed almost sure convergence toward a single limiting value, but subsequent work extended the theory to encompass weak convergence results, especially in the contexts of control and optimization \citep{asi2019stochastic, dupuis1985stochastic, dupuis1989stochastic, zhang2024constant}. To our knowledge, \cite{gangopadhyay2019stochastic} was the first to employ SA techniques in the study of urn models with random replacement mechanisms, thereby establishing a framework for applying SA to processes that can be interpreted as generalized urn models, including the random walk setting considered here. This direction was further pursued in \cite{gangopadhyay2022almost}, where a multidimensional extension of the elephant random walk on the nonnegative integer lattice was analysed, incorporating memory-based reinforcement and possible delays at each step. A recent and significant advancement in this area is due to \cite{maulik2024asymptotic}, wherein a stochastic approximation framework for multidimensional elephant random walks has been developed.

\section{Crucial notations and definitions}\label{sec:notations_definitions}
We dedicate \S\ref{sec:notations_definitions} to the introduction of notations and definitions that are going to be utilized throughout this paper, and especially in the statements of our main results in \S\ref{sec:main_results}. 

\subsection{Specific functions arising from our analysis of the model in \S\ref{sec:model}}\label{subsec:notations_1}
Recalling the reinforcement function $F$ defined via \eqref{X_{n+1}_distribution} (or, alternatively, via \eqref{A_{n+1}_distribution}, \eqref{B_{n+1}_distribution}, \eqref{C_{n+1}_distribution} and \eqref{D_{n+1}_distribution}), we define $g:\mathcal{S}\rightarrow[0,1]$ as
\begin{equation}
g(x,y)=pF(x,y)+(1-p)\{1-F(x,y)\} \quad \text{for all } (x,y)\in\mathcal{S},\label{g_defn}
\end{equation}
where $\mathcal{S}$ is the set defined in \eqref{domain_defn}. Since $F$ takes values in $[0,1]$ and $p\in[0,1]$, the function $g$ takes values in $[0,1]$ as well. Going forward, the fact that $g$ is bounded plays an important role in multiple contexts.

Recall that, in our stochastic process, the sample drawn at the beginning of epoch $(n+1)$ is of size $K_{n}$. When we are in Scenario~\eqref{Scenario_1}, we define the polynomial
\begin{equation}
H_{0}(x,y)=\sum_{k=1}^{M}\sum_{i=0}^{k}\sum_{j=0}^{k-i}g\left(\frac{i}{k},\frac{j}{k}\right){k\choose i}{k-i\choose j}x^{i}y^{j}(1-x-y)^{k-i-j}\mu(k) \quad \text{for } (x,y)\in\mathcal{S}.\label{H_{0}_defn}
\end{equation}
Since $g$ takes values in $[0,1]$, it is immediate from \eqref{H_{0}_defn} that 
\begin{equation}\label{H_{0}_bounds}
0\leqslant H_{0}(x,y)\leqslant 1 \text{ for all }(x,y)\in\mathcal{S}.
\end{equation}
Since $H_{0}$ is a polynomial, it is possible to define it over the entire $\mathbb{R}^{2}$ (however, the bounds stated in \eqref{H_{0}_bounds} are only guaranteed to hold for $(x,y)\in\mathcal{S}$). We define the functions $H$ and $h$ on $\mathcal{S}$ as follows:
\begin{align}
{}&H(x,y)=\begin{bmatrix}
q_{1}H_{0}(x,y)\\
q_{2}\left\{1-H_{0}(x,y)\right\}
\end{bmatrix} \quad \text{and} \quad h(x,y)=\begin{bmatrix}
h_{1}(x,y)\\
h_{2}(x,y)
\end{bmatrix}=H(x,y)-\begin{bmatrix}
x\\
y
\end{bmatrix} \quad \text{for all } (x,y)\in\mathcal{S},\label{h_H_defns}
\end{align}
where $q_{1},q_{2}\in(0,1)$ are the parameters mentioned in \eqref{first_few_opinions} and after \eqref{X_{n+1}_distribution}. When in Scenario~\eqref{Scenario_2}, similar to \eqref{h_H_defns}, we define the functions $\hat{H}$ and $\hat{h}$ on $\mathcal{S}$ as
\begin{equation}\label{hat{H}_hat{h}_defn}
\hat{H}(x,y)=\begin{bmatrix}
q_{1}g(x,y)\\
q_{2}\left\{1-g(x,y)\right\}
\end{bmatrix} \quad \text{and} \quad \hat{h}(x,y)=\begin{bmatrix}
\hat{h}_{1}(x,y)\\
\hat{h}_{2}(x,y)
\end{bmatrix}=\hat{H}(x,y)-\begin{bmatrix}
x\\
y
\end{bmatrix} \quad \text{for } (x,y)\in\mathcal{S},
\end{equation}
for $g$ as defined in \eqref{g_defn}. While the functions defined in \eqref{h_H_defns} appear in results pertaining to the set-up described in Scenario~\eqref{Scenario_1}, the functions in \eqref{hat{H}_hat{h}_defn} show up in results that are true in the set-up described in Scenario~\eqref{Scenario_2}.  

\subsection{Notations regarding functions and domains in general}\label{subsec:notations_2}
Given any open subset $U$ of $\mathbb{R}^{d}$, we indicate by $\mathcal{C}^{(i)}(U)$ the set of all real-valued functions $f$, defined on $U$, such that the $i$-th order partial derivatives of $f$ exist and are continuous throughout $U$. In particular, when $i=0$, the notation $\mathcal{C}^{(0)}(U)$ indicates the set of all real-valued continuous functions on $U$. Given subsets $U$ and $V$ of $\mathbb{R}^{d}$ such that $V\subset U$, and a function $f$ defined on $U$, we let $f|_{V}$ indicate the restriction of $f$ to $V$. Often, given a compact subset $K$ of $\mathbb{R}^{d}$ and a function $f$ that is defined on $K$, we write, loosely, $f\in\mathcal{C}^{(i)}(K)$ to mean that there exist 
\begin{enumerate*}
\item \emph{some} open subset $U$ of $\mathbb{R}^{d}$ and 
\item \emph{some} function $\hat{f}$ defined on $U$, 
\end{enumerate*} 
such that $K\subset U$, the function $\hat{f}\in\mathcal{C}^{(i)}(U)$, and $\hat{f}|_{K}\equiv f$. Adopting a standard notation from the physics literature, we let $\dot{x}(t)$ indicate the derivative of the function $x(t)$ with respect to time, $t$.

The \emph{modulus of continuity} of a function $f:V\rightarrow\mathbb{R}$, where $V\subset\mathbb{R}^{d}$, at some $\delta>0$, is defined as $\omega(f;\delta)=\sup\left\{|f(x)-f(y)|:x,y\in V,||x-y||_{2}<\delta\right\}$, where $||\cdot||_{2}$ indicates the usual Euclidean or $L_{2}$ norm. When the partial derivatives of the function $f$ are well-defined throughout the set $V$, we set, for $(u_{1},\ldots,u_{d})\in V$,
\begin{equation}
\nabla f(u_{1},\ldots,u_{d})=\left(\frac{\partial f}{\partial x_{1}}\Big|_{(u_{1},\ldots,u_{d})},\frac{\partial f}{\partial x_{2}}\Big|_{(u_{1},\ldots,u_{d})},\ldots,\frac{\partial f}{\partial x_{d}}\Big|_{(u_{1},\ldots,u_{d})}\right)\label{nabla_defn}
\end{equation}
and we also define, for any $\delta>0$,
\begin{multline}
\omega(\nabla f;\delta)=\sup\{\left|\left|\nabla f(y_{1},\ldots,y_{d})-\nabla f(z_{1},\ldots,z_{d})\right|\right|_{2}:(y_{1},\ldots,y_{d}),(z_{1},\ldots,z_{d})\in V,\\ \left|\left|(y_{1},\ldots,y_{d})-(z_{1},\ldots,z_{d})\right|\right|_{2}<\delta\}.\label{derivative_g_modulus_of_continuity}
\end{multline}

Given two sequences $\{a_{n}:n\in\mathbb{N}\}$ and $\{b_{n}:n\in\mathbb{N}\}$, we write $a_{n}=O(b_{n})$ to mean that there exists some constant $C>0$ such that $a_{n}\leqslant C b_{n}$ for all sufficiently large $n\in\mathbb{N}$, and we write $a_{n}=o(b_{n})$ to mean that the ratio $a_{n}/b_{n}$ converges to $0$ as $n\rightarrow\infty$. We write $a_{n}=\Theta(b_{n})$ to indicate that there exist constants $c_{1}>c_{2}>0$ such that $c_{2}b_{n}\leqslant a_{n}\leqslant c_{1}b_{n}$ for all sufficiently large $n\in\mathbb{N}$.

\section{Main results of this paper}\label{sec:main_results}
The main results of this paper may be classified into two categories: 
\begin{enumerate*}
\item those belonging to \S\ref{subsec:main_results_strong} concern themselves with the almost sure convergence of the sequence $\{(A_{n}/n,B_{n}/n):n\geqslant N\}$, and subsequently, of $\{C_{n}/n:n\geqslant N\}$, where $A_{n}$, $B_{n}$ and $C_{n}$ are as defined in the urn process described in \S\ref{sec:model}, while 
\item those belonging to \S\ref{subsec:main_results_weak} are concerned with the convergence in distribution of suitably scaled versions of $\{(A_{n}/n,B_{n}/n,C_{n}/n):n\geqslant N\}$, where the scaling depends on the \emph{regime} we are in.
\end{enumerate*}

\subsection{Results pertaining to almost sure convergence}\label{subsec:main_results_strong} When we are in Scenario \eqref{Scenario_1}, the almost sure convergence of $\{(n^{-1}A_{n},n^{-1}B_{n})\}$ is addressed by 
\begin{enumerate*}
\item Theorem~\ref{thm:main_1} when the samples are drawn with replacement, and
\item Theorem~\ref{thm:main_3} when the samples are drawn without replacement,
\end{enumerate*}
whereas when we are in Scenario \eqref{Scenario_2}, it is addressed by 
\begin{enumerate*}
\item Theorem~\ref{thm:main_2} when the samples are drawn with replacement, and
\item Theorem~\ref{thm:main_4} when the samples are drawn without replacement.
\end{enumerate*}
Theorems~\ref{thm:main_1_special} and \ref{thm:main_2_special} propose easily verifiable criteria under which $\{(n^{-1}A_{n},n^{-1}B_{n})\}$ converges almost surely to a constant $(x^{*},y^{*})\in\mathcal{S}$ (which may be estimated numerically using the parameter-pair $(q_{1},q_{2})$, and either the function $H_{0}$, defined in \eqref{H_{0}_defn}, or the function $g$, defined in \eqref{g_defn}). Finally, Theorem~\ref{thm:a.s._convergence_C_{n}} addresses the almost sure convergence of $\{n^{-1}C_{n}\}$.

\begin{theorem}\label{thm:main_1}
When we are in Scenario \eqref{Scenario_1}, and the samples are drawn with replacement, the sequence $\left\{\left(n^{-1}A_{n},n^{-1}B_{n}\right)\right\}$ converges, almost surely, to a (possibly sample path dependent) compact connected internally chain transitive invariant set, contained in $\mathcal{S}$, corresponding to the autonomous ODE:
\begin{equation}
\left(\dot{x}(t),\dot{y}(t)\right)=h\left(x(t),y(t)\right) \text{ for } t\geqslant 0, \quad \text{with} \quad \left(x(0),y(0)\right)\in\mathcal{S},\label{ODE_1}
\end{equation}
where the function $h$ is as defined in \eqref{h_H_defns} and $\mathcal{S}$ is as defined in \eqref{domain_defn}.
\end{theorem}
\begin{theorem}\label{thm:main_2}
Assume that we are in Scenario~\eqref{Scenario_2}, and that each sample is drawn with replacement. Assume that $g$, defined in \eqref{g_defn}, satisfies one of the following criteria:
\begin{enumerate}[label=(B\arabic*), ref=B\arabic*]
\item \label{thm:main_2:Lipschitz} that $g$ is Lipschitz on $\mathcal{S}$, and that the series 
\begin{equation}\label{thm:main_2:Lipschitz_series_convergence}
\sum_{n\geqslant N}(n+1)^{-1}\E\left[K_{n}^{-1/2}\right] \text{ converges};
\end{equation}
\item \label{thm:main_2:C^{1}} that $g\in\mathcal{C}^{(1)}(\mathcal{S})$, and that the series 
\begin{equation}\label{thm:main_2:C^{1}_series_convergence}
\sum_{n\geqslant N}\frac{1}{n+1}\min\left\{\omega\left(\nabla g;\sqrt{\E\left[K_{n}^{-1}\right]}\right)\sqrt{\E\left[K_{n}^{-1}\right]},\omega\left(\nabla g;\frac{\E\left[K_{n}^{-1}\right]}{\E\left[K_{n}^{-1/2}\right]}\right)\E\left[K_{n}^{-1/2}\right]\right\} \text{ converges};
\end{equation}
\item \label{thm:main_2:C^{2}} that $g\in\mathcal{C}^{(2)}(\mathcal{S})$, and that the series 
\begin{equation}\label{thm:main_2:C^{2}_series_convergence}
\sum_{n\geqslant N}(n+1)^{-1}\E\left[K_{n}^{-1}\right]\text{ converges}.
\end{equation}
\end{enumerate}
Then the sequence $\left\{\left(n^{-1}A_{n},n^{-1}B_{n}\right)\right\}$ converges almost surely to a (possibly sample path dependent) compact connected internally chain transitive invariant set, contained in $\mathcal{S}$, corresponding to the autonomous ODE:
\begin{equation}
\left(\dot{x}(t),\dot{y}(t)\right)=\hat{h}\left(x(t),y(t)\right) \text{ for }t\geqslant 0, \quad \text{with} \quad \left(x(0),y(0)\right)\in\mathcal{S},\label{ODE_2}
\end{equation}
where the function $\hat{h}$ is as defined in \eqref{hat{H}_hat{h}_defn}, and $\mathcal{S}$ is as defined in \eqref{domain_defn}.  
\end{theorem}
Recall, from the definitions introduced in \S\ref{subsec:notations_2}, that for any $i\in\{1,2\}$, by $g\in\mathcal{C}^{(i)}(\mathcal{S})$ we mean that there exist an open set $\mathcal{U}$ of $\mathbb{R}^{2}$ with $\mathcal{S}\subset\mathcal{U}$, and a function $\hat{g}$ defined on $\mathcal{U}$ with $\hat{g}\in \mathcal{C}^{(i)}(\mathcal{U})$, such that $\hat{g}\big|_{\mathcal{S}}\equiv g$. Note that in \eqref{thm:main_2:C^{1}_series_convergence}, the modulus of continuity of $\nabla g$ is computed on $\mathcal{S}$, i.e.\ when calculating $\omega(\nabla g;\delta)$ for $\delta>0$, we consider the supremum of $\left|\left|\nabla g(x_{1},y_{1})-\nabla g(x_{2},y_{2})\right|\right|_{2}$ over all $(x_{1},y_{1}), (x_{2},y_{2})\in\mathcal{S}$, with $\left|\left|(x_{1},y_{1})-(x_{2},y_{2})\right|\right|_{2}<\delta$.  

Note that having $\E[K_{n}^{-1}]=O(n^{-\alpha})$ as $n\rightarrow \infty$, for any $\alpha>0$, is sufficient to ensure that each of \eqref{thm:main_2:Lipschitz_series_convergence}, \eqref{thm:main_2:C^{1}_series_convergence} and \eqref{thm:main_2:C^{2}_series_convergence} is true (the first two require an application of Jensen's inequality). In particular, if we assume the law $\mu_{n}$ of $K_{n}$ to be supported on some subset $I_{n}=\{a_{n},a_{n}+1,\ldots,b_{n}\}$ of $\mathbb{N}$ such that $a_{n}=\Theta(n^{\alpha})$ for some $0<\alpha<1$, then automatically, $\E[K_{n}^{-1}]=O(n^{-\alpha})$. In fact, in Proposition~\ref{prop:K_{n}_distributions}, we allow the law $\mu_{n}$ of $K_{n}$ to belong to a class of very commonly studied probability distributions, and examine whether the various series convergence criteria (stated in Theorem~\ref{thm:main_2}, and in Theorem~\ref{thm:main_6} of \S\ref{subsec:main_results_weak}) are satisfied or not.

\begin{theorem}\label{thm:main_3}
Assume that we are in Scenario~\eqref{Scenario_1}, and that each sample is drawn without replacement. Then the conclusion drawn in Theorem~\ref{thm:main_1} remains true. 
\end{theorem}

\begin{theorem}\label{thm:main_4}
Assume that we are in Scenario~\eqref{Scenario_2}, and that each sample is drawn without replacement. If one of \eqref{thm:main_2:Lipschitz}, \eqref{thm:main_2:C^{1}} and \eqref{thm:main_2:C^{2}} holds, then the conclusion of Theorem~\ref{thm:main_2} remains true. 
\end{theorem} 

It is often nicer to impose additional restrictions -- for instance, on the function $H_{0}$, defined in \eqref{H_{0}_defn}, when we are in Scenario~\eqref{Scenario_1}, or on the function $g$, defined in \eqref{g_defn}, when we are in Scenario~\eqref{Scenario_2} -- in order to obtain more \emph{specific} results, such as the almost sure convergence of $\{(n^{-1}A_{n},n^{-1}B_{n})\}$ to a constant value (i.e.\ a degenerate random variable) in $\mathcal{S}$. This is what has been accomplished in Theorems~\ref{thm:main_1_special} and \ref{thm:main_2_special}. 
\begin{theorem}\label{thm:main_1_special}
Suppose we are in Scenario~\eqref{Scenario_1}, and let $H_{0}$, defined in \eqref{H_{0}_defn}, satisfy the criterion:
\begin{equation}
(q_{1}+q_{2})\max\left\{\sup\left\{\left|\frac{\partial}{\partial x}H_{0}(x,y)\right|:(x,y)\in\mathcal{S}\right\},\sup\left\{\left|\frac{\partial}{\partial y}H_{0}(x,y)\right|:(x,y)\in\mathcal{S}\right\}\right\}<1.\label{contraction_criterion_H_{0}}
\end{equation}
Then there exists a unique root, say $(x^{*},y^{*})$, of the function $h$, defined in \eqref{h_H_defns}, in the set $\mathcal{S}$, and irrespective of whether the samples are drawn with replacement or without, the sequence $\{(n^{-1}A_{n},n^{-1}B_{n})\}$ converges almost surely to $(x^{*},y^{*})$. The same conclusion remains true if $g\in\mathcal{C}^{(1)}(\mathcal{S})$ and 
\begin{equation}
(q_{1}+q_{2})\max\left\{\sup\left\{\left|\frac{\partial}{\partial x}g(x,y)\right|:(x,y)\in\mathcal{S}\right\},\sup\left\{\left|\frac{\partial}{\partial y}g(x,y)\right|:(x,y)\in\mathcal{S}\right\}\right\}<1.\label{contraction_criterion_g}
\end{equation}
\end{theorem} 

\begin{theorem}\label{thm:main_2_special}
In Scenario~\eqref{Scenario_2}, if one of \eqref{thm:main_2:C^{1}} and \eqref{thm:main_2:C^{2}} is true and \eqref{contraction_criterion_g} is satisfied, there exists a unique root, say $(x^{*},y^{*})$, of the function $\hat{h}$, defined in \eqref{hat{H}_hat{h}_defn}, in the set $\mathcal{S}$, and irrespective of whether the samples are drawn with replacement or without, the sequence $\{(n^{-1}A_{n},n^{-1}B_{n})\}$ converges almost surely to $(x^{*},y^{*})$.
\end{theorem}

\begin{theorem}\label{thm:a.s._convergence_C_{n}}
\sloppy Under Scenario~\eqref{Scenario_1}, the sequence $\{n^{-1}C_{n}\}$ converges almost surely. Under Scenario~\eqref{Scenario_2}, $\{n^{-1}C_{n}\}$ converges almost surely as long as $\{(n^{-1}A_{n},n^{-1}B_{n})\}$ does and one of the following is true:
\begin{enumerate}[label=(C\arabic*), ref=C\arabic*]
\item \label{thm:C_{n}_a.s.:g_Lipschitz} the function $g$, defined in \eqref{g_defn}, is Lipschitz on $\mathcal{S}$, and $\E[K_{n}^{-1/2}]\rightarrow 0$ as $n\rightarrow\infty$;
\item \label{thm:C_{n}_a.s.:g_C^{1}} the function $g\in\mathcal{C}^{(1)}(\mathcal{S})$, and, as $n\rightarrow\infty$:
\begin{equation}
\min\{\omega(\nabla g;(\E[K_{n}^{-1}])^{1/2})(\E[K_{n}^{-1}])^{1/2},\omega(\nabla g;\E[K_{n}^{-1}]/\E[K_{n}^{-1/2}])\E[K_{n}^{-1/2}]\}\rightarrow 0;\nonumber
\end{equation}
\item \label{thm:C_{n}_a.s.:g_C^{2}} the function $g\in\mathcal{C}^{(2)}(\mathcal{S})$, and $\E[K_{n}^{-1}]\rightarrow 0$ as $n\rightarrow\infty$.
\end{enumerate}
Furthermore, if $\{(n^{-1}A_{n},n^{-1}B_{n})\}$ converges almost surely to some $(x^{*},y^{*})\in\mathcal{S}$, then $\{n^{-1}C_{n}\}$ converges almost surely to $(1-q_{1})x^{*}/q_{1}$.
\end{theorem}
It is quite intuitive, in fact, that $n^{-1}C_{n}$ should converge almost surely whenever $n^{-1}A_{n}$ does. This is because, of the customers who choose product A, the proportion of those who are satisfied is roughly equal to $	q_{1}$, and the proportion of those who are discontented is roughly equal to $(1-q_{1})$, in the long run. Consequently, for large values of $n$, it is not unreasonable to expect that $C_{n}/A_{n}$ will take values close to $(1-q_{1})/q_{1}$. 

\subsection{Results pertaining to convergence in distribution}\label{subsec:main_results_weak} In what follows, we address convergence in distribution of suitably scaled versions of the sequence $\{(n^{-1}A_{n},n^{-1}B_{n},n^{-1}C_{n})\}$, where the scaling depends on the \emph{regime} that we are in. Theorem~\ref{thm:main_5} concerns itself with Scenario~\eqref{Scenario_1}, whereas Theorem~\ref{thm:main_6} addresses Scenario~\eqref{Scenario_2}. Note, from \eqref{H_{0}_defn}, that $H_{0}\in\mathcal{C}^{(1)}(\mathcal{S})$. If \eqref{contraction_criterion_H_{0}} holds, or if $g\in\mathcal{C}^{(1)}(\mathcal{S})$ and \eqref{contraction_criterion_g} holds, then, as shown in the proof of Theorem~\ref{thm:main_1_special}), the function $h$ has a unique root, denoted $(x^{*},y^{*})$, in $\mathcal{S}$.
\begin{theorem}\label{thm:main_5}
Consider Scenario~\eqref{Scenario_1}. Assume that \eqref{contraction_criterion_H_{0}} holds, or that $g\in\mathcal{C}^{(1)}(\mathcal{S})$ and \eqref{contraction_criterion_g} holds. Let
\begin{equation}\label{thm:main_5:scalar_defns}
\alpha^{*}=\frac{\partial H_{0}}{\partial x}\Big|_{(x^{*},y^{*})}, \quad \beta^{*}=\frac{\partial H_{0}}{\partial y}\Big|_{(x^{*},y^{*})}, \quad \rho=\min\left\{1,1-q_{1}\alpha^{*}+q_{2}\beta^{*}\right\}, \quad z^{*}=\frac{(1-q_{1})x^{*}}{q_{1}},
\end{equation} 
where $(x^{*},y^{*})$ is the unique root of $h$ in $\mathcal{S}$. We now define the matrices
\begin{equation}\label{thm:main_5:matrix_defns}
\Gamma=\begin{bmatrix}
x^{*}(1-x^{*}) & -x^{*}y^{*} & -x^{*}z^{*}\\
-x^{*}y^{*} & y^{*}(1-y^{*}) & -y^{*}z^{*}\\
-x^{*}z^{*} & -y^{*}z^{*} & z^{*}(1-z^{*})
\end{bmatrix}, \quad T=\begin{bmatrix}
-\beta^{*} & 0 & -q_{1}\\
\alpha^{*} & 0 & q_{2}\\
0 & 1 & -1+q_{1}
\end{bmatrix}, \quad \overline{T}=\begin{bmatrix}
-q_{1} & -1/\alpha^{*} & 0\\
q_{2} & 0 & 0\\
-1+q_{1} & 0 & 1
\end{bmatrix}.
\end{equation}
Next, we define, writing $\kappa=q_{1}\alpha^{*}-q_{2}\beta^{*}$ for the sake of brevity:
\begin{align}
A_{1,1}=&\kappa^{-2}\{q_{2}^{2}x^{*}(1-x^{*})-2q_{1}q_{2}x^{*}y^{*}+q_{1}^{2}y^{*}(1-y^{*})\},\nonumber\\
A_{1,2}=&-\kappa^{-2}(1-q_{1})\{q_{2}\alpha^{*}x^{*}(1-x^{*})-(\alpha^{*}q_{1}+\beta^{*}q_{2})x^{*}y^{*}+q_{1}\beta^{*}y^{*}(1-y^{*})\}-\kappa^{-1}z^{*}(q_{2}x^{*}+q_{1}y^{*}),\nonumber\\
A_{1,3}=&-\kappa^{-2}\{q_{2}\alpha^{*}x^{*}(1-x^{*})-(\alpha^{*}q_{1}+\beta^{*}q_{2})x^{*}y^{*}+q_{1}\beta^{*}y^{*}(1-y^{*})\},\nonumber\\
A_{2,2}=&\kappa^{-2}(1-q_{1})^{2}\{{\alpha^{*}}^{2}x^{*}(1-x^{*})-2\alpha^{*}\beta^{*}x^{*}y^{*}+{\beta^{*}}^{2}y^{*}(1-y^{*})\}+2\kappa^{-1}(1-q_{1})z^{*}\left(\alpha^{*}x^{*}+\beta^{*}y^{*}\right)+z^{*}(1-z^{*}),\nonumber\\
A_{2,3}=&\kappa^{-2}(1-q_{1})\{{\alpha^{*}}^{2}x^{*}(1-x^{*})-2\alpha^{*}\beta^{*}x^{*}y^{*}+{\beta^{*}}^{2}y^{*}(1-y^{*})\}+\kappa^{-1}z^{*}(\alpha^{*}x^{*}+\beta^{*}y^{*}),\nonumber\\
A_{3,3}=&\kappa^{-2}\{{\alpha^{*}}^{2}x^{*}(1-x^{*})-2\alpha^{*}\beta^{*}x^{*}y^{*}+{\beta^{*}}^{2}y^{*}(1-y^{*})\}.\nonumber
\end{align}
We further define
\begin{align}
B_{1,1}=&q_{2}^{-2}y^{*}(1-y^{*}),\quad B_{1,2}=q_{2}^{-1}\alpha^{*}x^{*}y^{*}-q_{2}^{-2}q_{1}\alpha^{*}y^{*}(1-y^{*}),\quad B_{1,3}=q_{2}^{-2}(1-q_{1})y^{*}(1-y^{*})-q_{2}^{-1}y^{*}z^{*},\nonumber\\
B_{2,2}=&{\alpha^{*}}^{2}\{x^{*}(1-x^{*})-2q_{2}^{-1}q_{1}x^{*}y^{*}+q_{2}^{-2}q_{1}^{2}y^{*}(1-y^{*})\},\nonumber\\
B_{2,3}=&q_{2}^{-1}\alpha^{*}\{(1-q_{1})x^{*}+q_{1}z^{*}\}y^{*}-q_{2}^{-2}\alpha^{*}q_{1}(1-q_{1})y^{*}(1-y^{*})+\alpha^{*}x^{*}z^{*},\nonumber\\
B_{3,3}=& q_{2}^{-2}(1-q_{1})^{2}y^{*}(1-y^{*})-2q_{2}^{-1}(1-q_{1})y^{*}z^{*}+z^{*}(1-z^{*}).\nonumber
\end{align}
\begin{enumerate}[label=(D\arabic*), ref=D\arabic*]
\item \label{thm:main_5:regime_1} When $\rho=1/2$, or equivalently, $\kappa=1/2$, we have, with $A_{3,3}$ as defined above:
\begin{equation}\label{thm:main_5:regime_1:conclusion}
\sqrt{\frac{n}{\log n}}\left(\begin{bmatrix}
n^{-1}A_{n}\\
n^{-1}B_{n}\\
n^{-1}C_{n}
\end{bmatrix}-\begin{bmatrix}
x^{*}\\
y^{*}\\
(1-q_{1})x^{*}/q_{1}
\end{bmatrix}\right)\convd N\left(0,A_{3,3}\begin{bmatrix}
q_{1}^{2} & -q_{1}q_{2} & q_{1}(1-q_{1})\\
-q_{1}q_{2} & q_{2}^{2} & -q_{2}(1-q_{1})\\
q_{1}(1-q_{1}) & -q_{2}(1-q_{1}) & (1-q_{1})^{2}
\end{bmatrix}\right).
\end{equation}
\item \label{thm:main_5:regime_2} When $0<\rho<1/2$, or, equivalently, $\kappa>1/2$, we have
\begin{equation}
n^{\rho}\left(\begin{bmatrix}
n^{-1}A_{n}\\
n^{-1}B_{n}\\
n^{-1}C_{n}
\end{bmatrix}-\begin{bmatrix}
x^{*}\\
y^{*}\\
(1-q_{1})x^{*}/q_{1}
\end{bmatrix}\right)-W\left(T^{-1}\right)^{t}\begin{bmatrix}
0\\
0\\
1
\end{bmatrix}\convas 0\label{thm:main_5:regime_2:conclusion}
\end{equation}
for some finite, scalar-valued random variable $W$, and $T$ as in \eqref{thm:main_5:matrix_defns}.
\item \label{thm:main_5:regime_3} We now consider $1/2<\rho\leqslant 1$. When $\kappa\in(-1,1/2)\setminus\{0\}$, we have
\begin{equation}\label{thm:main_5:regime_3.1:conclusion} 
\sqrt{n}\left(\begin{bmatrix}
n^{-1}A_{n}\\
n^{-1}B_{n}\\
n^{-1}C_{n}
\end{bmatrix}-\begin{bmatrix}
x^{*}\\
y^{*}\\
(1-q_{1})x^{*}/q_{1}
\end{bmatrix}\right)\convd N\left(0,T\begin{bmatrix}
A_{1,1} & A_{1,2} & C_{1,3}\\
A_{1,2} & A_{2,2} & C_{2,3}\\
C_{1,3} & C_{2,3} & C_{3,3}
\end{bmatrix}T^{t}\right),
\end{equation}
for $A_{i,j}$ as defined above, $C_{i,j}=(1-\kappa)^{-1}A_{i,j}$ for each $(i,j)\in\{(1,3),(2,3)\}$, and $C_{3,3}=(1-2\kappa)^{-1}A_{3,3}$, and $T$ is as defined in \eqref{thm:main_5:matrix_defns}. When $\alpha^{*}=\beta^{*}=0$ (which forces $\kappa=0$), we have 
\begin{equation}\label{thm:main_5:regime_3.2:conclusion} 
\sqrt{n}\left(\begin{bmatrix}
n^{-1}A_{n}\\
n^{-1}B_{n}\\
n^{-1}C_{n}
\end{bmatrix}-\begin{bmatrix}
x^{*}\\
y^{*}\\
(1-q_{1})x^{*}/q_{1}
\end{bmatrix}\right)\convd N\left(0,\Gamma\right), \text{ with }\Gamma\text{ as defined in \eqref{thm:main_5:matrix_defns}}.
\end{equation}
Finally, when $\kappa=0$ but $\alpha^{*}\neq 0$ (so that $\beta^{*}\neq 0$ as well), we have, for $B_{i,j}$ defined above:
\begin{equation}\label{thm:main_5:regime_3.3:conclusion} 
\sqrt{n}\left(\begin{bmatrix}
n^{-1}A_{n}\\
n^{-1}B_{n}\\
n^{-1}C_{n}
\end{bmatrix}-\begin{bmatrix}
x^{*}\\
y^{*}\\
(1-q_{1})x^{*}/q_{1}
\end{bmatrix}\right)\convd N\left(0,\overline{T}\begin{bmatrix}
B_{1,1}+2B_{1,2}+2B_{2,2} & B_{1,2}+B_{2,2} & B_{1,3}+B_{2,3}\\
B_{1,2}+B_{2,2} & B_{2,2} & B_{2,3}\\
B_{1,3}+B_{2,3} & B_{2,3} & B_{3,3}
\end{bmatrix}\overline{T}^{t}\right).
\end{equation}
\end{enumerate}
\end{theorem} 

We now address the convergence in distribution of suitably scaled versions of $\{(n^{-1}A_{n},n^{-1}B_{n},n^{-1}C_{n})\}$ when we are in Scenario~\eqref{Scenario_2}. The reader should bear in mind that, to maintain the analogy with Theorem~\ref{thm:main_5}, nearly every object of interest in the statement (and proof) of Theorem~\ref{thm:main_6} has been assigned the same name / notation as the analogous object of interest in the statement of Theorem~\ref{thm:main_5}. For instance, $(x^{*},y^{*})$ in Theorem~\ref{thm:main_6} refers to the (unique) root of $\hat{h}$ in $\mathcal{S}$, as opposed to Theorem~\ref{thm:main_5} where $(x^{*},y^{*})$ refers to the (unique) root of $h$ in $\mathcal{S}$, where $\hat{h}$ is as defined in \eqref{hat{H}_hat{h}_defn} and $h$ is as defined in \eqref{h_H_defns}.
\begin{theorem}\label{thm:main_6}
Consider Scenario~\eqref{Scenario_2}. We consider the following different regimes:
\begin{enumerate}[label=(E\arabic*), ref=E\arabic*]
\item \label{thm:main_6:regime_1} Suppose that \eqref{thm:C_{n}_a.s.:g_C^{2}} and \eqref{contraction_criterion_g} are satisfied, and, almost surely as $n\rightarrow\infty$, 
\begin{equation}\label{thm:main_6:regime_1:series_criteria}
\text{either }\sum_{i=N}^{n}\E[K_{i}^{-1}]=o(\sqrt{n/\log n}) \quad \text{or} \quad \sum_{i=N}^{n}\E[K_{i}^{-1}](i+1)^{-1/2}=o(\sqrt{\log n}).
\end{equation}
Then the function $\hat{h}$ has a unique root, $(x^{*},y^{*})$, in $\mathcal{S}$. Let $\alpha^{*}=\partial g/\partial x|_{(x^{*},y^{*})}$ and $\beta^{*}=\partial g/\partial y|_{(x^{*},y^{*})}$. Let the scalars $\rho$ and $z^{*}$ be defined as in \eqref{thm:main_5:scalar_defns}, the matrices $\Gamma$, $T$ and $\overline{T}$ be defined as in \eqref{thm:main_5:matrix_defns}, and $\kappa$, $A_{i,j}$ and $B_{i,j}$ for each $(i,j)\in\{(1,1),(1,2),\ldots,(3,3)\}$, and $C_{i,j}$ for each $(i,j)\in\{(1,3),(2,3),(3,3)\}$, be defined as in the statement of Theorem~\ref{thm:main_5}, but with $x^{*}$, $y^{*}$, $\alpha^{*}$ and $\beta^{*}$ as defined in the context of Theorem~\ref{thm:main_6}. If, now, $\rho=1/2$, or equivalently, $\kappa=1/2$, the conclusion given by \eqref{thm:main_5:regime_1:conclusion} remains true.

\item \label{thm:main_6:regime_2} Let \eqref{thm:C_{n}_a.s.:g_C^{2}} and \eqref{contraction_criterion_g} be satisfied, let $0<\rho<1/2$, or, equivalently, $\kappa>1/2$, and let  
\begin{equation}\label{thm:main_6:regime_2:series_criteria}
\sum_{i=N}^{n}\E[K_{i}^{-1}]=o(n^{1-\rho-\epsilon}) \text{ almost surely}
\end{equation}
for some $0<\epsilon<1-\rho$. Then, for some finite, scalar-valued random variable $W$, and $T$ as described in \eqref{thm:main_6:regime_1}, the same conclusion as shown in \eqref{thm:main_5:regime_2:conclusion} remains true.

\item \label{thm:main_6:regime_3} Suppose that either \eqref{thm:C_{n}_a.s.:g_C^{1}} and \eqref{contraction_criterion_g} are satisfied, and, almost surely, 
\begin{equation}\label{thm:main_6:regime_3.1:series_criteria}
\sum_{i=N}^{n}\min\{\omega(\nabla g;(\E[K_{i}^{-1}])^{1/2})(\E[K_{i}^{-1}])^{1/2},\omega(\nabla g;\E[K_{i}^{-1}]/\E[K_{i}^{-1/2}])\E[K_{i}^{-1/2}]\}=o(\sqrt{n}),
\end{equation}
or that \eqref{thm:C_{n}_a.s.:g_C^{2}} and \eqref{contraction_criterion_g} are satisfied, and, almost surely, 
\begin{equation}\label{thm:main_6:regime_3.2:series_criteria}
\sum_{i=N}^{n}\E[K_{i}^{-1}]=o(\sqrt{n}).
\end{equation}
Let $1/2<\rho\leqslant 1$. When $\kappa\in(-1,1/2)\setminus\{0\}$, the conclusion of \eqref{thm:main_5:regime_3.1:conclusion} remains true. When $\alpha^{*}=\beta^{*}=0$ (so that $\kappa=0$), the conclusion of \eqref{thm:main_5:regime_3.2:conclusion} remains true. Finally, when $\kappa=0$ but $\alpha^{*}\neq 0$ (so that $\beta^{*}\neq 0$ as well), the conclusion of \eqref{thm:main_5:regime_3.3:conclusion} remains true. 
\end{enumerate} 
\end{theorem}

In each of Theorem~\ref{thm:main_2}, Theorem~\ref{thm:main_4}, Theorem~\ref{thm:main_2_special}, the second part of Theorem~\ref{thm:a.s._convergence_C_{n}} and Theorem~\ref{thm:main_6}, we are in Scenario~\eqref{Scenario_2}, and we have imposed several criteria on the sequence $\{K_{n}\}$ (or, rather, on their inverse moments). In Proposition~\ref{prop:K_{n}_distributions}, we show that the criteria corresponding to the case where the drift function $g\in\mathcal{C}^{(2)}(\mathcal{S})$ are all satisfied when $K_{n}$ follows one of several commonly studied distributions:
\begin{prop}\label{prop:K_{n}_distributions}
Assume $g\in\mathcal{C}^{(2)}(\mathcal{S})$ and \eqref{contraction_criterion_g} holds. Then the conclusion drawn in each of Theorem~\ref{thm:main_2_special}, the second part of Theorem~\ref{thm:a.s._convergence_C_{n}}, and Theorem~\ref{thm:main_6} remains true for 
\begin{enumerate*}
\item $K_{n}$ that is distributed as discrete uniform over $\{1,2,\ldots,n\}$;
\item $K_{n}=\min\{n,G_{n}\}$, where $G_{n}$ follows Geometric$(p_{n})$ with probability of success $p_{n}=c n^{-\alpha}$ for $\alpha>1/2$;
\item $K_{n}=1+T_{n-1}$, where $T_{n-1}$ follows Binomial$(n-1,p_{n})$ with $p_{n}=c n^{-\alpha}$ for $0<\alpha<1/2$;
\item $K_{n}=\min\{n,1+P_{n}\}$, where $P_{n}$ follows Poisson$(\lambda_{n})$ with $\lambda_{n}=c n^{\alpha}$ for $\alpha>1/2$.
\end{enumerate*}
\end{prop}

\section{Results from the literature that have been repeatedly used in the proofs of our main results}\label{sec:lit_results_useful}
\sloppy In order to prove each of Theorems~\ref{thm:main_1}, \ref{thm:main_2}, \ref{thm:main_3} and \ref{thm:main_4}, we have relied on a slightly generalized version of [Theorem 2, Chapter 2, \cite{borkar2008stochastic}] that we have stated below, followed by a brief discussion on how this generalization can be proved using arguments very similar to those used for proving the original version:
\begin{theorem}[Generalized Theorem 2, Chapter 2, \cite{borkar2008stochastic}]\label{thm:borkar_a.s.}
Consider a stochastic approximation process $\{Z_{n}:n\in\mathbb{N}_{0}\}$, given by
\begin{equation}
Z_{n+1}=Z_{n}+\gamma_{n}\left[G(Z_{n})+\Delta M_{n+1}+\delta_{n}\right],\label{sa_general}
\end{equation}
where each of $Z_{n}$, $\Delta M_{n+1}$ and $\delta_{n}$ is a random variable taking values in $\mathbb{R}^{d}$ for some $d\in\mathbb{N}$, and the following assumptions are satisfied:
\begin{enumerate}[label=(F\arabic*), ref=F\arabic*]
\item \label{gen_assump_a.s._1} the \emph{drift function} $G:\mathbb{R}^{d}\rightarrow\mathbb{R}^{d}$ is Lipschitz,
\item \label{gen_assump_a.s._2} the sequence $\{\gamma_{n}\}$ of \emph{step-sizes}, with each $\gamma_{n}$ a positive scalar, satisfies $\sum_{n}\gamma_{n}=\infty$ and $\sum_{n}\gamma_{n}^{2}<\infty$,
\item \label{gen_assump_a.s._3} the sequence $\{\Delta M_{n}\}$ is a martingale difference sequence with respect to the filtration $\{\mathcal{F}_{n}\}$, where $\mathcal{F}_{n}$ is the $\sigma$-field generated by the stochastic process up to and including epoch $n$, and the sequence $\{\delta_{n}\}$ is adapted to $\{\mathcal{F}_{n}\}$, and the series 
\begin{equation}
\sum_{n\in\mathbb{N}_{0}}\gamma_{n}\left(\Delta M_{n+1}+\delta_{n}\right) \text{ converges almost surely},\label{borkar_convergence_criterion_generalized}
\end{equation}
\item \label{gen_assump_a.s._4} the iterates of \eqref{sa_general} remain bounded almost surely, i.e.\ $\sup_{n}\left|\left|Z_{n}\right|\right|_{2}<\infty$ almost surely.
\end{enumerate}
Then, almost surely, the sequence $\{Z_{n}\}$ converges to a (possibly sample path dependent) compact connected internally chain transitive invariant set corresponding to the ODE: $\dot{x}(t)=G(x(t))$ for $t\geqslant 0$.
\end{theorem}
The stochastic approximation in \eqref{sa_general} incorporates into it an additional error term, $\delta_{n}$, that is absent in the statement of [Theorem 2, Chapter 2, \cite{borkar2008stochastic}]. However, the only place where the proof of Theorem~\ref{thm:borkar_a.s.} differs from the proof of [Theorem 2, Chapter 2, \cite{borkar2008stochastic}] is that, instead of the almost sure convergence of $\sum_{n\in\mathbb{N}_{0}}\gamma_{n}\Delta M_{n+1}$, we now must ensure that the series in \eqref{borkar_convergence_criterion_generalized} converges almost surely (this, in turn, is used for establishing [Lemma 1, Chapter 2, \cite{borkar2008stochastic}]).

\begin{remark}\label{rem:assumption_4_holds}
Recall that, apart from Theorem~\ref{thm:a.s._convergence_C_{n}}, all the main results stated in \S\ref{subsec:main_results_strong} are concerned with the almost sure convergence of $(n^{-1}A_{n},n^{-1}B_{n})$, and it is in the proof of each of these results that we implement Theorem~\ref{thm:borkar_a.s.}. The definitions of $A_{n}$ and $B_{n}$, as part of the urn process described in \S\ref{sec:model}, make it clear that $A_{n}+B_{n}\leqslant n$, so that $||(n^{-1}A_{n},n^{-1}B_{n})||_{2}\leqslant 1$ for each $n\geqslant N$. Therefore, Assumption~\eqref{gen_assump_a.s._4} is satisfied, and we no longer include its verification in the proofs that follow.
\end{remark}

\begin{remark}\label{rem:Lipschitz_extension}
In each of the proofs where Theorem~\ref{thm:borkar_a.s.} has been implemented, the drift function $G$, whether it equals $h$ defined in \eqref{h_H_defns} or $\hat{h}$ defined in \eqref{hat{H}_hat{h}_defn}, has been shown to be Lipschitz on the compact set $\mathcal{S}$ defined in \eqref{domain_defn}. This ensures, by Kirszbraun's Theorem (see \cite{kirszbraun1934zusammenziehende} and Theorem 1.2 of \cite{azagra2021kirszbraun}), the existence of a function $\overline{G}:\mathbb{R}^{2}\rightarrow\mathbb{R}^{2}$ such that $\overline{G}|_{\mathcal{S}}\equiv G$ and $\overline{G}$ is Lipschitz throughout $\mathbb{R}^{2}$ with the Lipschitz constant being the same as that for $G$ on $\mathcal{S}$. This argument, required for the verification of Assumption~\eqref{gen_assump_a.s._1}, has not been repeated in any of the proofs that follow, for the sake of brevity.
\end{remark}

In order to prove Theorems~\ref{thm:main_5} and \ref{thm:main_6}, we rely on [Theorems 2.1, 2.2 and 2.3 of \cite{zhang2016central}]. The premise for each of these results is, once again, the stochastic approximation process given by \eqref{sa_general}, with $\gamma_{n}=(n+1)^{-1}$ for each $n\in\mathbb{N}_{0}$. Some subset of the following three assumptions has been utilized in each of the three main results of \cite{zhang2016central}:
\begin{enumerate}[label=(G\arabic*), ref=G\arabic*]
\item \label{gen_assump_dist_1} Let $G(\theta^{*})=0$ for some $\theta^{*}\in\mathbb{R}^{d}$, such that $G$ is differentiable at $\theta^{*}$, and all eigenvalues of the Jacobian, $J_{G}(\theta^{*})$, of $G$ at $\theta^{*}$, have strictly negative real parts.
\item \label{gen_assump_dist_2} In addition to Assumption~\eqref{gen_assump_dist_1}, here we assume the existence of some $\eta>0$ such that
\begin{equation}
G(\theta)=J_{G}(\theta^{*})(\theta-\theta^{*})+o\left(\left|\left|\theta-\theta^{*}\right|\right|_{2}^{1+\eta}\right).\nonumber
\end{equation}
\item \label{gen_assump_dist_3} We assume the the martingale difference sequence $\{\Delta M_{n+1}\}$ satisfies the following Lindeberg condition: for each $\epsilon>0$,
\begin{equation}\label{Lindeberg_cond}
\frac{1}{n}\sum_{i=0}^{n-1}\E\left[\left|\left|\Delta M_{i+1}\right|\right|_{2}^{2}\chi\left\{\left|\left|\Delta M_{i+1}\right|\right|_{2}\geqslant \epsilon\sqrt{n}\right\}\Big|\mathcal{F}_{i}\right]\convas 0\text{ as }n\rightarrow\infty.
\end{equation}
Moreover, 
\begin{equation}
\frac{1}{n}\sum_{i=0}^{n-1}\E\left[\Delta M_{i+1}\left(\Delta M_{i+1}\right)^{t}\Big|\mathcal{F}_{i}\right]\convas \Gamma,\label{MDS_quadratic_variation_Cesaro_conv}
\end{equation}
where $\Gamma$ is a symmetric, positive-semidefinite, random matrix (of dimension $d\times d$). Here and going forward, for any matrix (or vector) $A$, the notation $A^{t}$ indicates its transpose.
\end{enumerate}
Under Assumption~\eqref{gen_assump_dist_1}, let the Jordan canonical form corresponding to $-J_{G}(\theta^{*})$ be given by $\text{diag}(J_{1},J_{2},\ldots,J_{s})$, where $J_{i}$ is a Jordan block of dimension $\nu_{i}\times\nu_{i}$, with each of its diagonal entries equal to $-\lambda_{i}$, where $\lambda_{1},\ldots,\lambda_{s}$ are the (not necessarily distinct) eigenvalues of $J_{G}(\theta^{*})$. Let $T$ be a non-singular matrix such that 
\begin{equation}
-T^{-1}J_{G}(\theta^{*})T=\text{diag}(J_{1},J_{2},\ldots,J_{s}).\label{similarity_eq_general}
\end{equation}
We define $\rho=\min\{\text{Re}(-\lambda_{i}):i\in[s]\}$, where $\text{Re}(\lambda)$, for $\lambda\in\mathbb{C}$, indicates the real part of $\lambda$. We let $\nu$ equal the maximum of $\nu_{t}$ for all those $t\in[s]$ such that $\text{Re}(-\lambda_{t})=\rho$. In what follows, we indicate by $0$ the $d$-dimensional column vector each of whose coordinates is equal to $0$, and by $I_{d}$ the $d\times d$ identity matrix.

\begin{theorem}[Theorem 2.1 of \cite{zhang2016central}]\label{thm:zhang_2.1}
Suppose $\{Z_{n}\}$ converges to the constant $\theta^{*}$ almost surely, and that Assumptions~\eqref{gen_assump_dist_2} and \eqref{gen_assump_dist_3} are satisfied. Suppose $\rho=1/2$, and that either
\begin{equation}
\sum_{i=0}^{n-1}\delta_{i}=o\left(\sqrt{n/\log n}\right) \text{ almost surely,}\label{series_conv_zhang_2.1_1}
\end{equation}
or that 
\begin{equation}
\sum_{i=0}^{n-1}\frac{\left|\left|\delta_{i}\right|\right|_{2}}{\sqrt{i+1}}=o\left(\sqrt{\log n}\right) \text{ almost surely}.\label{series_conv_zhang_2.1_2}
\end{equation}
Then
\begin{equation}\label{conclusion:zhang_2.1}
\frac{\sqrt{n}}{(\log n)^{\nu-1/2}}\left(Z_{n}-\theta^{*}\right)\convd N(0,\Sigma),
\end{equation}
where, letting $\Gamma$ be the matrix mentioned in Assumption~\eqref{gen_assump_dist_3}, we set 
\begin{equation}\label{sigma:zhang_2.1}
\Sigma=\lim_{n\rightarrow\infty}\frac{1}{(\log n)^{2\nu-1}}\int_{0}^{\log n}\exp\left\{\left(J_{G}(\theta^{*})+I_{d}/2\right)u\right\}\Gamma\left[\exp\left\{\left(J_{G}(\theta^{*})+I_{d}/2\right)u\right\}\right]^{t}du.
\end{equation}
\end{theorem}

\begin{theorem}[Theorem 2.2 of \cite{zhang2016central}]\label{thm:zhang_2.2}
Suppose $\{Z_{n}\}$ converges to the constant $\theta^{*}$ almost surely, and that Assumption~\eqref{gen_assump_dist_2} is satisfied. Suppose $0<\rho<1/2$, that
\begin{equation}
\sum_{i=0}^{n-1}\E\left[\Delta M_{i+1}\left(\Delta M_{i+1}\right)^{t}\Big|\mathcal{F}_{i}\right]=O(n) \text{ almost surely},\label{series_conv_zhang_2.2_1}
\end{equation}
and that there exists some $0<\epsilon<1-\rho$ such that
\begin{equation}
\sum_{i=0}^{n-1}\delta_{i}=o\left(n^{1-\rho-\epsilon}\right) \text{ almost surely}.\label{series_conv_zhang_2.2_2}
\end{equation}
Then there exist complex random variables $W_{1},\ldots,W_{s}$ such that 
\begin{equation}
\frac{n^{\rho}}{(\log n)^{\nu-1}}(Z_{n}-\theta^{*})-\sum_{j\in[s]:\text{Re}(-\lambda_{j})=\rho,\nu_{j}=\nu}\exp\left\{i\text{Im}(\lambda_{j})\log n\right\}W_{j}\left(T^{-1}\right)^{t}e_{j}\convas 0,\label{conclusion:zhang_2.2}
\end{equation}
where $i=\sqrt{-1}$, the notation $\text{Im}(\lambda)$ indicates the imaginary part of $\lambda$, the matrix $T$ is as defined via \eqref{similarity_eq_general}, and $e_{j}$ is the $d$-dimensional column vector such that, if its first $\nu_{1}$ coordinates make up its first block, the next $\nu_{2}$ coordinates make up its second block, and so on, then the $\nu_{j}$-th coordinate of its $j$-th block equals $1$, and all other coordinates equal $0$. 
\end{theorem}
 
\begin{theorem}[Theorem 2.3 of \cite{zhang2016central}]\label{thm:zhang_2.3}
Suppose $\{Z_{n}\}$ converges to the constant $\theta^{*}$ almost surely, and that Assumptions~\eqref{gen_assump_dist_1} and \eqref{gen_assump_dist_3} are satisfied. Suppose $\rho>1/2$, and that
\begin{equation}\label{series_conv_zhang_2.3}
\sum_{i=0}^{n-1}\delta_{i}=o\left(\sqrt{n}\right) \text{ almost surely}.
\end{equation}
Then 
\begin{equation}\label{conclusion:zhang_2.3}
\sqrt{n}\left(Z_{n}-\theta^{*}\right)\convd N(0,\Sigma),
\end{equation}
where, with $\Gamma$ is the matrix mentioned in Assumption~\eqref{gen_assump_dist_3}, we set
\begin{equation}\label{sigma:zhang_2.3}
\Sigma=\int_{0}^{\infty}\exp\left\{\left(J_{G}(\theta^{*})+I_{d}/2\right)u\right\}\Gamma\left[\exp\left\{\left(J_{G}(\theta^{*})+I_{d}/2\right)u\right\}\right]^{t}du.
\end{equation}
\end{theorem}

\section{Proofs of our main results}\label{sec:proofs}
Before we begin in earnest the proof of each of the results stated in \S\ref{sec:main_results}, we mention, for the reader's convenience, that the central theme of each of these proofs is the representation of our stochastic process as a suitable stochastic approximation process. The specific stochastic approximation process we obtain depends on whether we are in Scenario~\eqref{Scenario_1} or \eqref{Scenario_2}, and whether the samples are drawn with or without replacement. Throughout the rest of the paper, we let $\mathcal{F}_{n}$ indicate the $\sigma$-field constituting all information regarding our stochastic process up to and including epoch $n$. 

\begin{proof}[Proof of Theorem~\ref{thm:main_1}]
When we are in \eqref{Scenario_1}, and the samples are drawn with replacement, we obtain the simplest of the stochastic approximation processes, given by \eqref{sa_1} for all $n\geqslant N$:
\begin{align}
\begin{bmatrix}
(n+1)^{-1}A_{n+1}\\
(n+1)^{-1}B_{n+1}
\end{bmatrix}={}&\begin{bmatrix}
n^{-1}A_{n}\\
n^{-1}B_{n}
\end{bmatrix}+\frac{1}{n+1}\left\{\begin{bmatrix}
A_{n+1}-A_{n}\\
B_{n+1}-B_{n}
\end{bmatrix}-\begin{bmatrix}
n^{-1}A_{n}\\
n^{-1}B_{n}
\end{bmatrix}\right\}\nonumber\\
={}&\begin{bmatrix}
n^{-1}A_{n}\\
n^{-1}B_{n}
\end{bmatrix}+\frac{1}{n+1}\left\{\begin{bmatrix}
\Delta M_{n+1,1}\\
\Delta M_{n+1,2}
\end{bmatrix}+h\left(\begin{bmatrix}
n^{-1}A_{n}\\
n^{-1}B_{n}
\end{bmatrix}\right)\right\},\label{sa_1}
\end{align}
where the step-size $\gamma_{n}=(n+1)^{-1}$ for each $n\geqslant N$, validating, right away, Assumption~\eqref{gen_assump_a.s._2}. We define
\begin{equation}\label{martingale_difference_defn_scenario_1_with_replacement}
\begin{cases}
{}&\Delta M_{n+1,1}=A_{n+1}-A_{n}-q_{1}H_{0}\left(n^{-1}A_{n},n^{-1}B_{n}\right),\\
{}&\Delta M_{n+1,2}=B_{n+1}-B_{n}-q_{2}\left\{1-H_{0}\left(n^{-1}A_{n},n^{-1}B_{n}\right)\right\},
\end{cases}
\end{equation}
where the functions $H_{0}$ and $h$ are as defined in \eqref{H_{0}_defn} and \eqref{h_H_defns} respectively. Simple computations utilizing \eqref{A_{n+1}_distribution} and \eqref{B_{n+1}_distribution} reveal that $\{\Delta M_{n+1}:n\geqslant N\}=\{(\Delta M_{n+1,1},\Delta M_{n+1,2})^{t}:n\geqslant N\}$ forms a martingale difference sequence with respect to the filtration $\{\mathcal{F}_{n+1}:n\geqslant N\}$. Furthermore, this observation, along with the fact that each of $(A_{n+1}-A_{n})$ and $(B_{n+1}-B_{n})$ is a Bernoulli random variable, yields:
\begin{align}
\sum_{i=N}^{n}\E\left[\gamma_{i}^{2}\left|\left|\Delta M_{i+1}\right|\right|_{2}^{2}\big|\mathcal{F}_{i}\right]={}&q_{1}\sum_{i=N}^{n}\frac{1}{(i+1)^{2}}H_{0}\left(\frac{A_{i}}{i},\frac{B_{i}}{i}\right)\left\{1-q_{1}H_{0}\left(\frac{A_{i}}{i},\frac{B_{i}}{i}\right)\right\}\nonumber\\&+q_{2}\sum_{i=N}^{n}\frac{1}{(i+1)^{2}}\left\{1-H_{0}\left(\frac{A_{i}}{i},\frac{B_{i}}{i}\right)\right\}\left[1-q_{2}\left\{1-H_{0}\left(\frac{A_{i}}{i},\frac{B_{i}}{i}\right)\right\}\right],\label{quadratic_variation_with_replacement_scenario_1}
\end{align}
which, by \eqref{H_{0}_bounds}, evidently converges as $n\rightarrow\infty$. This, in turn, implies, via [Theorem 4.5.2.\ of \cite{durrett2019probability}], that the martingale $\{\sum_{i=N}^{n}\gamma_{i}\Delta M_{i+1}:n\geqslant N\}$ converges almost surely. This completes the verification of Assumption~\eqref{gen_assump_a.s._3}. Since $H_{0}$ is a polynomial, $H_{0}\in\mathcal{C}^{(1)}(\mathcal{S})$. This, along with the fact that $\mathcal{S}$ is compact, ensures that $h$ is Lipschitz on $\mathcal{S}$. This, along with Remark~\ref{rem:Lipschitz_extension}, completes the verification of Assumption~\eqref{gen_assump_a.s._1}. The assertion made in the statement of Theorem~\ref{thm:main_1} now follows from an application of Theorem~\ref{thm:borkar_a.s.}.
\end{proof}

\begin{proof}[Proof of Theorem~\ref{thm:main_2}]
Here, we are in Scenario~\eqref{Scenario_2}, but similar to the set-up in Theorem~\ref{thm:main_1}, the samples are drawn with replacement. The corresponding stochastic approximation process, for $n\geqslant N$, is given by:
\begin{align}
\begin{bmatrix}
(n+1)^{-1}A_{n+1}\\
(n+1)^{-1}B_{n+1}
\end{bmatrix}={}&\begin{bmatrix}
n^{-1}A_{n}\\
n^{-1}B_{n}
\end{bmatrix}+\frac{1}{n+1}\left\{\begin{bmatrix}
\Delta M_{n+1,1}\\
\Delta M_{n+1,2}
\end{bmatrix}+\begin{bmatrix}
\delta_{n,1}\\
\delta_{n,2}
\end{bmatrix}+\hat{h}\left(\begin{bmatrix}
n^{-1}A_{n}\\
n^{-1}B_{n}
\end{bmatrix}\right)\right\},\label{sa_2}
\end{align}
where the step-sizes $\gamma_{n}=(n+1)^{-1}$, for $n\geqslant N$, ensure that Assumption~\eqref{gen_assump_a.s._2} is true. The function $\hat{h}$, defined in \eqref{hat{H}_hat{h}_defn}, is Lipschitz on the compact set $\mathcal{S}$ since so is $g$ defined in \eqref{g_defn} (the first criterion stated in each of \eqref{thm:main_2:Lipschitz}, \eqref{thm:main_2:C^{1}} and \eqref{thm:main_2:C^{2}} ensures this). This, along with Remark~\ref{rem:Lipschitz_extension}, validates Assumption~\eqref{gen_assump_a.s._1}. We set
\begin{equation}\label{many_defns_with_replacement_scenario_2}
\begin{cases}
{}&\Delta M_{n+1,1}=A_{n+1}-A_{n}-q_{1}H_{n}\left(n^{-1}A_{n},n^{-1}B_{n}\right),\\
{}&\Delta M_{n+1,2}=B_{n+1}-B_{n}-q_{2}\left\{1-H_{n}\left(n^{-1}A_{n},n^{-1}B_{n}\right)\right\},\\
{}&\delta_{n,1}=q_{1}\left\{H_{n}\left(n^{-1}A_{n},n^{-1}B_{n}\right)-g\left(n^{-1}A_{n},n^{-1}B_{n}\right)\right\},\\
{}&\delta_{n,2}=q_{2}\left\{g\left(n^{-1}A_{n},n^{-1}B_{n}\right)-H_{n}\left(n^{-1}A_{n},n^{-1}B_{n}\right)\right\},
\end{cases}
\end{equation}
where the sequence $\{H_{n}:n\in\mathbb{N},n\geqslant N\}$ is defined as follows (recall $I_{n}$ and $\mu_{n}(\cdot)$ from  Scenario~\eqref{Scenario_2}):
\begin{equation}
H_{n}(x,y)=\sum_{k\in I_{n}}\sum_{i=0}^{k}\sum_{j=0}^{k-i}g\left(\frac{i}{k},\frac{j}{k}\right){k\choose i}{k-i\choose j}x^{i}y^{j}(1-x-y)^{k-i-j}\mu_{n}(k) \quad \text{for all } (x,y)\in\mathcal{S}.\label{H_{n}_defn}
\end{equation}
If we consider the random experiment in \eqref{Exp_1}, then $H_{n}(x,y)=\E[g(V_{1}/K_{n},V_{2}/K_{n})]$, and since the function $g$ takes values in $[0,1]$, as mentioned right before \eqref{g_defn}, we have 
\begin{equation}\label{H_{n}_bounds}
0\leqslant H_{n}(x,y)\leqslant 1 \text{ for each }(x,y)\in\mathcal{S}, \text{ for each }n\geqslant N.
\end{equation}
Simple computations utilizing \eqref{A_{n+1}_distribution} and \eqref{B_{n+1}_distribution} reveal that $\{\Delta M_{n+1}:n\geqslant N\}=\{(\Delta M_{n+1,1},\Delta M_{n+1,2})^{t}:n\geqslant N\}$, defined via \eqref{many_defns_with_replacement_scenario_2}, forms a martingale difference sequence with respect to the filtration $\{\mathcal{F}_{n+1}:n\geqslant N\}$. Moreover, an identity analogous to \eqref{quadratic_variation_with_replacement_scenario_1} is true, with $H_{0}(i^{-1}A_{i},i^{-1}B_{i})$ replaced by $H_{i}(i^{-1}A_{i},i^{-1}B_{i})$ for each $i\in\{N,\ldots,n\}$, so that by \eqref{H_{n}_bounds}, we conclude that the series $\sum_{n\geqslant N}\E[\gamma_{n}^{2}||\Delta M_{n+1}||_{2}^{2}\big|\mathcal{F}_{n}]$ converges. Once again, [Theorem 4.5.2.\ of \cite{durrett2019probability}] ensures that the martingale $\{\sum_{i=N}^{n}\gamma_{i}\Delta M_{i+1}:n\geqslant N\}$ converges almost surely.

The final conclusion drawn in the previous paragraph ensures one part of Assumption~\eqref{gen_assump_a.s._3} -- to ensure the rest, we must show that $\sum_{n}\gamma_{n}\delta_{n}$ converges, where $\delta_{n}=(\delta_{n,1},\delta_{n,2})^{t}$ is defined using \eqref{many_defns_with_replacement_scenario_2}. 
\begin{enumerate}
\item When $g$ is Lipschitz, it is H\"{o}lder continuous with exponent $\alpha=1$ (see \eqref{Holder_cont_cond}), so that by \eqref{Bernstein_error_Lipschitz_H_{n}} of Lemma~\ref{lem:Bernstein_approximation_errors_Holder_continuous}, we obtain:
\begin{equation}
\left|\left|\sum_{i=N}^{n}\gamma_{i}\delta_{i}\right|\right|_{2}\leqslant\sum_{i=1}^{n}\gamma_{i}\left|\left|\delta_{i}\right|\right|_{2}\leqslant 2^{-1/2}L\left(q_{1}^{2}+q_{2}^{2}\right)^{1/2}\sum_{i=1}^{n}(i+1)^{-1}\E\left[K_{i}^{-1/2}\right],\nonumber
\end{equation}
which converges because of \eqref{thm:main_2:Lipschitz_series_convergence}.
\item When $g\in\mathcal{C}^{(1)}(\mathcal{S})$, by \eqref{Bernstein_error_C^{1}_H_{n}} of Lemma~\ref{lem:Bernstein_approximation_errors_C^{1}}, we obtain:
\begin{align}
\left|\left|\sum_{i=N}^{n}\gamma_{i}\delta_{i}\right|\right|_{2}\leqslant{}&C\left(q_{1}^{2}+q_{2}^{2}\right)^{1/2}\sum_{i=1}^{n}(i+1)^{-1}\min\Big\{\omega\Big(\nabla g;\sqrt{\E\left[K_{i}^{-1}\right]}\Big)\sqrt{\E\left[K_{i}^{-1}\right]},\nonumber\\&\omega\Big(\nabla g;\E\left[K_{i}^{-1}\right]\left(\E\left[K_{i}^{-1/2}\right]\right)^{-1}\Big)\E\left[K_{i}^{-1/2}\right]\Big\},\nonumber
\end{align}
and this converges as $n\rightarrow\infty$ when \eqref{thm:main_2:C^{1}_series_convergence} is true.
\item When $g\in\mathcal{C}^{(2)}(\mathcal{S})$, by \eqref{Bernstein_error_C^{2}_H_{n}} of Lemma~\ref{lem:Bernstein_approximation_errors_C^{2}}, we obtain:
\begin{align}
\left|\left|\sum_{i=N}^{n}\gamma_{i}\delta_{i}\right|\right|_{2}\leqslant{}&C\left(q_{1}^{2}+q_{2}^{2}\right)^{1/2}\sum_{i=1}^{n}(i+1)^{-1}\E\left[K_{i}^{-1}\right],\nonumber
\end{align}
and this converges as $n\rightarrow\infty$ when \eqref{thm:main_2:C^{2}_series_convergence} is true.
\end{enumerate}
This completes the verification of Assumption~\eqref{gen_assump_a.s._3}. Theorem~\eqref{thm:main_2} now follows from an application of Theorem~\ref{thm:borkar_a.s.}.
\end{proof}

\begin{proof}[Proof of Theorem~\ref{thm:main_3}]
When we are in Scenario~\eqref{Scenario_1} and the samples are drawn without replacement, the stochastic approximation process, for $n\geqslant N$, is given by
\begin{align}
\begin{bmatrix}
(n+1)^{-1}A_{n+1}\\
(n+1)^{-1}B_{n+1}
\end{bmatrix}={}&\begin{bmatrix}
n^{-1}A_{n}\\
n^{-1}B_{n}
\end{bmatrix}+\frac{1}{n+1}\left\{\begin{bmatrix}
\Delta M_{n+1,1}\\
\Delta M_{n+1,2}
\end{bmatrix}+\begin{bmatrix}
\delta_{n,1}\\
\delta_{n,2}
\end{bmatrix}+h\left(\begin{bmatrix}
n^{-1}A_{n}\\
n^{-1}B_{n}
\end{bmatrix}\right)\right\},\label{sa_3}
\end{align}
with the drift function $h$ as defined in \eqref{h_H_defns}, and the step-sizes $\gamma_{n}=(n+1)^{-1}$ satisfying \eqref{gen_assump_a.s._2}. That Assumption~\eqref{gen_assump_a.s._1} is true is justified as in the proof of Theorem~\ref{thm:main_1}. We set 
\begin{equation}\label{many_defns_without_replacement_identical_law}
\begin{cases}
{}&\Delta M_{n+1,1}=A_{n+1}-A_{n}-q_{1}F_{n}\left(n^{-1}A_{n},n^{-1}B_{n}\right),\\
{}&\Delta M_{n+1,2}=B_{n+1}-B_{n}-q_{2}\left\{1-F_{n}\left(n^{-1}A_{n},n^{-1}B_{n}\right)\right\},\\
{}&\delta_{n,1}=q_{1}\left\{F_{n}\left(n^{-1}A_{n},n^{-1}B_{n}\right)-H_{0}\left(n^{-1}A_{n},n^{-1}B_{n}\right)\right\},\\
{}&\delta_{n,2}=q_{2}\left\{H_{0}\left(n^{-1}A_{n},n^{-1}B_{n}\right)-F_{n}\left(n^{-1}A_{n},n^{-1}B_{n}\right)\right\},
\end{cases}
\end{equation}
where we define $F_{n}:\mathcal{S}_{n}\rightarrow[0,1]$ for each $n\geqslant N$, with $\mathcal{S}_{n}$ as defined in \eqref{S_{n}_defn}, as follows:
\begin{equation}\label{F_{n}_defn}
F_{n}\left(\frac{r_{1}}{n},\frac{r_{2}}{n}\right)=\sum_{k=1}^{M}\sum_{i=0}^{k}\sum_{j=0}^{k-i}g\left(\frac{i}{k},\frac{j}{k}\right){r_{1} \choose i}{r_{2} \choose j}{n-r_{1}-r_{2}\choose k-i-j}\left\{{n\choose k}\right\}^{-1}\mu(k) \text{ for } \left(\frac{r_{1}}{n},\frac{r_{2}}{n}\right)\in\mathcal{S}_{n}.
\end{equation}
For each $(r_{1}/n,r_{2}/n)\in \mathcal{S}_{n}$, we have $F_{n}(r_{1}/n,r_{2}/n)=\E[g(V_{1}/K_{n},V_{2}/K_{n})]$ for $V_{1},V_{2}$ as in Experiment~\eqref{Exp_2}, $g$ as in \eqref{g_defn} and $K_{n}$ as in Scenario~\eqref{Scenario_1}. Since $0\leqslant g(x,y)\leqslant 1$ for $(x,y)\in\mathcal{S}$, this ensures
\begin{equation}\label{F_{n}_bounds}
0\leqslant F_{n}(r_{1}/n,r_{2}/n)\leqslant 1 \text{ for each }(r_{1}/n,r_{2}/n)\in \mathcal{S}_{n}.
\end{equation}
Simple computations using \eqref{A_{n+1}_distribution}, \eqref{B_{n+1}_distribution} and the fact that the samples are drawn without replacement reveal that $\{\Delta M_{n+1}:n\geqslant N\}$, where $\Delta M_{n+1}=(\Delta M_{n+1,1},\Delta M_{n+1,2})^{t}$ is defined using \eqref{many_defns_without_replacement_identical_law}, forms a martingale difference sequence with respect to $\{\mathcal{F}_{n+1}:n\geqslant N\}$. An identity analogous to \eqref{quadratic_variation_with_replacement_scenario_1}, with $H_{0}(i^{-1}A_{i},i^{-1}B_{i})$ replaced by $F_{i}(i^{-1}A_{i},i^{-1}B_{i})$ for $i\in\{N,\ldots,n\}$, is true in this set-up, which, along with \eqref{F_{n}_bounds}, allows us to conclude, using [Theorem 4.5.2.\ of \cite{durrett2019probability}], that the martingale $\{\sum_{i=N}^{n}\gamma_{i}\Delta M_{i+1}:n\geqslant N\}$ converges almost surely. Setting $\delta_{n}=(\delta_{n,1},\delta_{n,2})^{t}$, defined using \eqref{many_defns_without_replacement_identical_law}, and noting that $(i^{-1}A_{i},i^{-1}B_{i})\in\mathcal{S}_{i}$ for each $i\in\mathbb{N}$, from Lemma~\ref{lem:hypergeometric_approx_sample_size_iid}, we get $||\sum_{i=N}^{n}\gamma_{i}\delta_{i}||_{2}\leqslant \sum_{i=N}^{n}(i+1)^{-1}||\delta_{i}||_{2}=O(\sum_{i=N}^{n}i^{-2})$, which converges as $n\rightarrow\infty$. This completes the verification of Assumption~\eqref{gen_assump_a.s._3}, and the conclusion follows from Theorem~\ref{thm:borkar_a.s.}.
\end{proof}

\begin{proof}[Proof of Theorem~\ref{thm:main_4}]
When we are in Scenario~\eqref{Scenario_2} and the samples are drawn without replacement, the stochastic approximation process, for $n\geqslant N$, is given by
\begin{align}
\begin{bmatrix}
(n+1)^{-1}A_{n+1}\\
(n+1)^{-1}B_{n+1}
\end{bmatrix}={}&\begin{bmatrix}
n^{-1}A_{n}\\
n^{-1}B_{n}
\end{bmatrix}+\frac{1}{n+1}\left\{\begin{bmatrix}
\Delta M_{n+1,1}\\
\Delta M_{n+1,2}
\end{bmatrix}+\begin{bmatrix}
\delta_{n,1}\\
\delta_{n,2}
\end{bmatrix}+\hat{h}\left(\begin{bmatrix}
n^{-1}A_{n}\\
n^{-1}B_{n}
\end{bmatrix}\right)\right\},\label{sa_4}
\end{align}
with the drift function $\hat{h}$ as defined in \eqref{hat{H}_hat{h}_defn}, and step-sizes $\gamma_{n}=(n+1)^{-1}$ satisfying Assumption~\eqref{gen_assump_a.s._2}. That Assumption~\eqref{gen_assump_a.s._1} holds is justified as in the proof of Theorem~\ref{thm:main_2}. We set
\begin{equation}\label{many_defns_without_replacement_varying_law}
\begin{cases}
{}&\Delta M_{n+1,1}=A_{n+1}-A_{n}-q_{1}E_{n}\left(n^{-1}A_{n},n^{-1}B_{n}\right),\\
{}&\Delta M_{n+1,2}=B_{n+1}-B_{n}-q_{2}\left\{1-E_{n}\left(n^{-1}A_{n},n^{-1}B_{n}\right)\right\},\\
{}&\delta_{n,1}=q_{1}\left\{E_{n}\left(n^{-1}A_{n},n^{-1}B_{n}\right)-g\left(n^{-1}A_{n},n^{-1}B_{n}\right)\right\},\\
{}&\delta_{n,2}=q_{2}\left\{g\left(n^{-1}A_{n},n^{-1}B_{n}\right)-E_{n}\left(n^{-1}A_{n},n^{-1}B_{n}\right)\right\},
\end{cases}
\end{equation}
where the sequence of functions $\{E_{n}:n\geqslant N\}$ is defined as
\begin{equation}\label{E_{n}_defn}
E_{n}\left(\frac{r_{1}}{n},\frac{r_{2}}{n}\right)=\sum_{k\in I_{n}}\sum_{i=0}^{k}\sum_{j=0}^{k-i}g\left(\frac{i}{k},\frac{j}{k}\right){r_{1} \choose i}{r_{2} \choose j}{n-r_{1}-r_{2}\choose k-i-j}\left\{{n\choose k}\right\}^{-1}\mu_{n}(k), \text{ for } \left(\frac{r_{1}}{n},\frac{r_{2}}{n}\right)\in \mathcal{S}_{n},
\end{equation}
with $\mathcal{S}_{n}$ as defined in \eqref{S_{n}_defn}. Note that $E_{n}(r_{1}/n,r_{2}/n)=\E[g(V_{1}/K_{n},V_{2}/K{n})]$ for $V_{1},V_{2}$ as in Experiment~\eqref{Exp_2}, $g$ as in \eqref{g_defn} and $K_{n}$ as in Scenario~\eqref{Scenario_2}. Since $0\leqslant g(x,y)\leqslant 1$ for $(x,y)\in\mathcal{S}$, this ensures
\begin{equation}\label{E_{n}_bounds}
0\leqslant E_{n}(r_{1}/n,r_{2}/n)\leqslant 1 \text{ for each }(r_{1}/n,r_{2}/n)\in \mathcal{S}_{n}.
\end{equation}
Simple computations using \eqref{A_{n+1}_distribution}, \eqref{B_{n+1}_distribution} and the fact that the samples are drawn without replacement reveal that $\{\Delta M_{n+1}:n\geqslant N\}$, where $\Delta M_{n+1}=(\Delta M_{n+1,1},\Delta M_{n+1,2})^{t}$ is defined using \eqref{many_defns_without_replacement_varying_law}, forms a martingale difference sequence with respect to $\{\mathcal{F}_{n+1}:n\geqslant N\}$. An identity analogous to \eqref{quadratic_variation_with_replacement_scenario_1}, with $H_{0}(i^{-1}A_{i},i^{-1}B_{i})$ replaced by $E_{i}(i^{-1}A_{i},i^{-1}B_{i})$ for each $i\in\{N,\ldots,n\}$, is true in this set-up, which, along with \eqref{E_{n}_bounds}, allows us to conclude that the martingale $\{\sum_{i=N}^{n}\gamma_{i}\Delta M_{i+1}:n\geqslant N\}$ converges almost surely. 

To verify the rest of Assumption~\eqref{gen_assump_a.s._3}, we set $\delta_{n}=(\delta_{n,1},\delta_{n,2})^{t}$, with $\delta_{n,i}$ as defined in \eqref{many_defns_without_replacement_varying_law} for each $i\in\{1,2\}$. Note that $(n^{-1}A_{n},n^{-1}B_{n})$ always takes values in $\mathcal{S}_{n}$. That the series $\sum_{n}\gamma_{n}\delta_{n}$ converges almost surely is established exactly as it was done in the proof of Theorem~\ref{thm:main_2} -- the only difference being: 
\begin{enumerate}
\item when $g$ is Lipschitz, we use \eqref{Bernstein_error_Lipschitz_E_{n}} of Lemma~\ref{lem:Bernstein_approximation_errors_Holder_continuous}, with $\alpha=1$, instead of \eqref{Bernstein_error_Lipschitz_H_{n}}; 
\item when $g\in \mathcal{C}^{(1)}(\mathcal{S})$, we make use of \eqref{Bernstein_error_C^{1}_E_{n}} of Lemma~\ref{lem:Bernstein_approximation_errors_C^{1}}, instead of \eqref{Bernstein_error_C^{1}_H_{n}};
\item and finally, when $g\in\mathcal{C}^{(2)}(\mathcal{S})$, we make use of \eqref{Bernstein_error_C^{2}_E_{n}} of Lemma~\ref{lem:Bernstein_approximation_errors_C^{2}}, instead of \eqref{Bernstein_error_C^{2}_H_{n}}.
\end{enumerate}
Thus, Assumption~\eqref{gen_assump_a.s._3} holds, and the conclusion follows via Theorem~\ref{thm:borkar_a.s.}.
\end{proof}

\begin{proof}[Proof of Theorem~\ref{thm:main_1_special}]
It suffices for us to show that, when \eqref{contraction_criterion_H_{0}} holds (with $H_{0}$ as defined in \eqref{H_{0}_defn}), 
\begin{enumerate}
\item the function $h$, defined in \eqref{h_H_defns}, has a unique root, say $(x^{*},y^{*})$, in $\mathcal{S}$, 
\item and the only compact, connected, internally chain transitive invariant set corresponding to \eqref{ODE_1} that is contained in $\mathcal{S}$ is the singleton, $\{(x^{*},y^{*})\}$.
\end{enumerate}
These have been established via several steps, each of which has been highlighted below.

\textbf{Proving that $h$ has a unique root in $\mathcal{S}$:} This is equivalent to showing that $H$, defined in \eqref{h_H_defns}, has a unique fixed point in $\mathcal{S}$, when \eqref{contraction_criterion_H_{0}} holds. We first argue that $H$ has at least one fixed point in $\mathcal{S}$. By \eqref{H_{0}_bounds},
\begin{equation}\label{H_maps_domain_to_itself}
\begin{cases}
&q_{1}H_{0}(x,y) \geqslant 0 \quad \text{and} \quad q_{2}\left\{1-H_{0}(x,y)\right\}\geqslant 0 \text{ for all } (x,y)\in\mathcal{S},\\
&0<\min\{q_{1},q_{2}\}\leqslant q_{1}H_{0}(x,y)+q_{2}\left\{1-H_{0}(x,y)\right\}\leqslant\max\{q_{1},q_{2}\}<1 \text{ for all } (x,y)\in\mathcal{S},
\end{cases}
\end{equation}
implying, from \eqref{h_H_defns} and \eqref{domain_defn}, that the function $H$ maps the compact, convex subset $\mathcal{S}$ of the Euclidean space $\mathbb{R}^{2}$, to itself. Moreover, $H$ is continuous throughout $\mathcal{S}$ since $H_{0}$ is a polynomial. Consequently, Brouwer's Fixed Point Theorem yields the conclusion that $H$ has at least one fixed point in $\mathcal{S}$.

Next, we establish the uniqueness of this fixed point when \eqref{contraction_criterion_H_{0}} holds, by establishing that $H$ is a contraction on $\mathcal{S}$. Given any matrix $A$ of order $m\times n$, for any $m, n \in\mathbb{N}$, in which the entry in the $i$-th row and the $j$-th column is denoted by $a_{i,j}$ for $i\in[m]$ and $j\in[n]$, its \emph{$1$-norm}, indicated by $||A||_{1}$, is defined as $||A||_{1}=\max\{\sum_{i=1}^{m}|a_{i,j}|:j\in[n]\}$. Simple computations yield $||AB||_{1}\leqslant ||A||_{1}||B||_{1}$ for any two matrices $A$ and $B$ of orders $m\times n$ and $n\times p$ respectively, proving that the $1$-norm is \emph{submultiplicative}.  

Since $H_{0}$ is a polynomial, for any $(x_{1},y_{1}), (x_{2},y_{2})\in\mathcal{S}$, we can write, using the mean value theorem,
\begin{align}
H(x_{1},y_{1})-H(x_{2},y_{2})={}&J_{H}(x_{3},y_{3})\left(\begin{bmatrix}
x_{1}\\
y_{1}
\end{bmatrix}-\begin{bmatrix}
x_{2}\\
y_{2}
\end{bmatrix}\right),\nonumber
\end{align}
where $J_{H}(x_{3},y_{3})$ is the Jacobian matrix for $H$ evaluated at some point $(x_{3},y_{3})$ which lies on the line segment joining the points $(x_{1},y_{1})$ and $(x_{2},y_{2})$. Since $\mathcal{S}$ is convex and $(x_{1},y_{1}), (x_{2},y_{2})\in\mathcal{S}$, hence $(x_{3},y_{3})\in\mathcal{S}$ as well. Applying the submultiplicative nature of the $1$-norm defined in the previous paragraph, we obtain:
\begin{align}
{}&\left|\left|H(x_{1},y_{1})-H(x_{2},y_{2})\right|\right|_{1}\leqslant (q_{1}+q_{2})\max\left\{\left|\frac{\partial H_{0}}{\partial x}\Big|_{(x_{3},y_{3})}\right|,\left|\frac{\partial H_{0}}{\partial y}\Big|_{(x_{3},y_{3})}\right|\right\}\left|\left|\begin{bmatrix}
x_{1}\\
y_{1}
\end{bmatrix}-\begin{bmatrix}
x_{2}\\
y_{2}
\end{bmatrix}\right|\right|_{1}\nonumber\\
\leqslant{}&(q_{1}+q_{2})\max\left\{\sup\left\{\left|\frac{\partial H_{0}}{\partial x}\right|:(x,y)\in\mathcal{S}\right\},\sup\left\{\left|\frac{\partial H_{0}}{\partial y}\right|:(x,y)\in\mathcal{S}\right\}\right\}\left|\left|\begin{bmatrix}
x_{1}\\
y_{1}
\end{bmatrix}-\begin{bmatrix}
x_{2}\\
y_{2}
\end{bmatrix}\right|\right|_{1},\nonumber
\end{align}
so that $H$ is immediately seen to be a contraction on $\mathcal{S}$ when \eqref{contraction_criterion_H_{0}} holds. This ensures, by the Banach fixed point theorem, that there exists, in $\mathcal{S}$, a \emph{unique} fixed point, say $(x^{*},y^{*})$, of the function $H$.

\textbf{Proving that $\mathcal{S}$ is a positively invariant set corresponding to \eqref{ODE_1}:} Here, we need to invoke some notions and terminology from the literature on ordinary differential equations. First, we prove that $\mathcal{S}$ is a \emph{positively invariant set} (see Definition 4.1 of \cite{blanchini2008set}) with respect to the autonomous ODE in \eqref{ODE_1}. This is accomplished by making use of Nagumo's Theorem (see \cite{nagumo1942lage} as well as \S~4.2 of \cite{blanchini2008set}), which, in turn, makes use of \emph{Bouligand's tangent cone} (see Definition 4.6 of \cite{blanchini2008set}, or \cite{bouligand1932introduction}): given a closed subset $S$ of $\mathbb{R}^{d}$ for any $d\in\mathbb{N}$, and $\pmb{x}\in S$, the \emph{tangent cone} to $S$ at $\pmb{x}$ is defined as
\begin{equation}
\mathcal{T}_{S}(\pmb{x})=\left\{\pmb{y}\in\mathbb{R}^{d}:\liminf_{t\rightarrow 0+}\frac{\text{dist}(\pmb{x}+t\pmb{y},S)}{t}=0\right\},\label{tangent_cone}
\end{equation}
where, for $\pmb{z}\in\mathbb{R}^{d}$, we define dist$(\pmb{z},S)=\inf\left\{||\pmb{z}-\pmb{w}||:\pmb{w}\in S\right\}$, with $||\cdot||$ any prespecified norm on $\mathbb{R}^{d}$. By \eqref{domain_defn}, the set $\mathcal{S}$ is indeed closed, and for any $(x,y)\in\mathcal{S}$ and any $t\in(0,1]$, using \eqref{h_H_defns}, we have:
\begin{align}
\begin{bmatrix}
x\\
y
\end{bmatrix}+t
\cdot h(x,y)=\begin{bmatrix}
x\\
y
\end{bmatrix}+t\begin{bmatrix}
q_{1}H_{0}(x,y)-x\\
q_{2}\left\{1-H_{0}(x,y)\right\}-y
\end{bmatrix}=(1-t)\begin{bmatrix}
x\\
y
\end{bmatrix}+t\begin{bmatrix}
q_{1}H_{0}(x,y)\\
q_{2}\left\{1-H_{0}(x,y)\right\}.\label{convex_combination_1}
\end{bmatrix}
\end{align}
Note that this is a convex combination of two points both of which belong to $\mathcal{S}$ (this follows from \eqref{H_maps_domain_to_itself}), and since $\mathcal{S}$ itself is a convex set, the point obtained in \eqref{convex_combination_1} must be in $\mathcal{S}$ as well. Therefore, 
\begin{equation}
\text{dist}\left(\begin{bmatrix}
x\\
y
\end{bmatrix}+t\cdot h(x,y), \mathcal{S}\right)=0 \text{ for each } t\in(0,1] \implies h(x,y)\in\mathcal{T}_{\mathcal{S}}\left(\begin{bmatrix}
x\\
y
\end{bmatrix}\right),\label{h(x,y)_in_tangent_cone}
\end{equation}
for each $(x,y)\in\mathcal{S}$, with $\mathcal{T}_{\mathcal{S}}(\cdot)$ defined as in \eqref{tangent_cone}. Next, we have to show that \eqref{ODE_1} has a unique \emph{global} solution (i.e.\ a solution defined for all $t\geqslant 0$) for each $(x(0),y(0))\in\mathcal{S}$. From the fact that $h$ is Lipschitz on $\mathcal{S}$ (as aruged in the proof of Theorem~\ref{thm:main_1}) and by Remark~\ref{rem:Lipschitz_extension}, there exists an extension of $h$ to a function $\overline{h}:\mathbb{R}^{2}\rightarrow\mathbb{R}^{2}$ such that $\overline{h}$ is Lipschitz throughout $\mathbb{R}^{2}$, with Lipschitz constant the same as that for $h$ on $\mathcal{S}$. From [\cite{borkar2008stochastic}, Appendix B], we then conclude that the autonomous ODE
\begin{equation}
\left(\dot{x}(t),\dot{y}(t)\right)=\overline{h}\left(x(t),y(t)\right) \text{ for }t\geqslant 0, \text{ with } \left(x(0),y(0)\right)\in\mathbb{R}^{2},\label{ODE_3}
\end{equation}
is well-posed, i.e.\ there exists a \emph{unique} solution, $(x(t),y(t))$, to the initial value problem in \eqref{ODE_3}, for all $t\geqslant 0$. This conclusion, along with \eqref{h(x,y)_in_tangent_cone}, allows us to establish, via Nagumo's Theorem (see Corollary 4.8 of \cite{blanchini2008set}), that $\mathcal{S}$ is a positively invariant set corresponding to \eqref{ODE_3}, and consequently, to \eqref{ODE_1} (since \eqref{ODE_3} reduces to \eqref{ODE_1} when $(x(0),y(0))\in\mathcal{S}$). 

\textbf{An application of the Poincar\'{e}-Bendixson Theorem and the Bendixson Criterion:} Let $\omega(x,y)$ denote the \emph{$\omega$-limit set} of $(x,y)$ (see Definitions 8.1.1 and 8.1.2 of \cite{wiggins2003introduction}) corresponding to \eqref{ODE_1}, for each $(x,y)\in\mathcal{S}$. Since $\mathcal{S}$ is a compact and simply connected set that is positively invariant corresponding to \eqref{ODE_1}, we conclude, by the Poincar\'{e}-Bendixson Theorem (see Theorem~9.0.6 of \cite{wiggins2003introduction}) that
\begin{enumerate}
\item\label{Poincare_Bendixson_1} either $\omega(x,y)$ is a singleton consisting of a root of the function $h$,
\item\label{Poincare_Bendixson_2} or $\omega(x,y)$ is a \emph{closed orbit} (also called a \emph{periodic orbit} -- see Definition 4.0.1 of \cite{wiggins2003introduction}), and since $\mathcal{S}$ has already been shown to be positively invariant, such an orbit must lie entirely in $\mathcal{S}$,
\item\label{Poincare_Bendixson_3} or $\omega(x,y)$ consists of a finite number of roots of $h$, say $(p_{1},q_{1}), (p_{2},q_{2}), \ldots, (p_{n},q_{n})$, and orbits $\gamma_{i,j}$, such that $\alpha(\gamma_{i,j})=\{(p_{i},q_{i})\}$ and $\omega(\gamma_{i,j})=\{(p_{j},q_{j})\}$ for each $i$, $j$ (again, we refer to Definitions 8.1.1 and 8.1.2 of \cite{wiggins2003introduction} for the definition of \emph{$\alpha$-limit sets}).
\end{enumerate} 
Since $h$ has a unique root in $\mathcal{S}$, hence we need only focus on Options \eqref{Poincare_Bendixson_1} and \eqref{Poincare_Bendixson_2}. To rule out option \eqref{Poincare_Bendixson_2}, we apply Bendixson's criterion (Theorem 4.1.1 of \cite{wiggins2003introduction}). From \eqref{h_H_defns} and \eqref{contraction_criterion_H_{0}}, we have, for $(u,v)\in\mathcal{S}$:
\begin{align}
{}&\frac{\partial h_{1}}{\partial x}\Big|_{(u,v)}+\frac{\partial h_{2}}{\partial y}\Big|_{(u,v)}=q_{1}\frac{\partial H_{0}}{\partial x}\Big|_{(u,v)}-1-q_{2}\frac{\partial H_{0}}{\partial y}\Big|_{(u,v)}-1\leqslant q_{1}\left|\frac{\partial H_{0}}{\partial x}\Big|_{(u,v)}\right|+q_{2}\left|\frac{\partial H_{0}}{\partial y}\Big|_{(u,v)}\right|-2\nonumber\\
\leqslant{}& (q_{1}+q_{2})\max\left\{\sup\left\{\left|\frac{\partial H_{0}}{\partial x}\right|:(x,y)\in\mathcal{S}\right\},\sup\left\{\left|\frac{\partial H_{0}}{\partial y}\right|:(x,y)\in\mathcal{S}\right\}\right\}-2<-1,\label{Bendixson_sign}
\end{align}
showing that on the simply connected region $\mathcal{S}$, the expression at the beginning of \eqref{Bendixson_sign} is neither identically zero nor does it change sign, so that there exists no closed orbit that lies entirely in $\mathcal{S}$. This rules out option \eqref{Poincare_Bendixson_2}, allowing us to conclude that for each $(x,y)\in\mathcal{S}$, we have $\omega(x,y)=\{(x^{*},y^{*})\}$, where $(x^{*},y^{*})$ is the unique root of $h$ in $\mathcal{S}$.

\textbf{Proving that the only compact, connected, internally chain transitive invariant set corresponding to \eqref{ODE_1} is the singleton $\{(x^{*},y^{*})\}$:} Given an autonomous ODE, $\dot{x}(t)=f(x(t))$, where $f:\mathbb{R}^{d}\rightarrow\mathbb{R}^{d}$, a closed subset $\mathcal{O}\subset\mathbb{R}^{d}$ is said to be an \emph{internally chain transitive invariant set} corresponding to this ODE if
\begin{enumerate}[label=(I\arabic*), ref=I\arabic*]
\item \label{I1} for each initial value $x(0)\in\mathcal{O}$, we have $x(t)\in\mathcal{O}$ for each $t\in\mathbb{R}$ (we replace invariance by positive invariance if this assertion is true for $t\geqslant 0$, and by negative invariance if this is true for $t\leqslant 0$),
\item \label{I2} and for any $u,v\in\mathcal{O}$, any $\epsilon>0$ and any $T>0$, there exist $n\in\mathbb{N}$ and points $u_{1},u_{2},\ldots,u_{n-1}\in\mathcal{O}$ such that, if we set $u_{0}=u$ and $u_{n}=v$, then the trajectory of the above-mentioned ODE that is initiated at $u_{i}$ meets with the $\epsilon$-neighbourhood of $u_{i+1}$ after a time $\geqslant T$, for each $i\in\{0,1,\ldots,n-1\}$.
\end{enumerate}
Evidently, an internally chain transitive invariant set has to be both positively and negatively invariant, and hence, it suffices for us to show that the only compact, connected, internally chain transitive positively invariant set corresponding to \eqref{ODE_1} (with the initial value $(x(0),y(0))$ contained in $\mathcal{S}$) is $\{(x^{*},y^{*})\}$.

Since we have already established $\mathcal{S}$ to be positively invariant corresponding to \eqref{ODE_1}, if $M$ is a compact, connected, internally chain transitive positively invariant set corresponding to \eqref{ODE_1} with $(x(0),y(0))\in\mathcal{S}$, then $M\subset\mathcal{S}$. There are two possibilities:
\begin{enumerate}
\item Suppose $(x^{*},y^{*})\in M$. If possible, let $M$ contain a second point, say $(x',y')$. Since $h(x^{*},y^{*})=0$, the solution to \eqref{ODE_1} initiated at $(x^{*},y^{*})$ remains forever at $(x^{*},y^{*})$, and the condition in \eqref{I2} cannot be satisfied for any $T>0$, any sufficiently small $\epsilon>0$ and any $n\in\mathbb{N}$ when we set $u_{0}=(x^{*},y^{*})$ and $u_{n}=(x',y')$. This brings us to a contradiction, proving that in this case, $M=\{(x^{*},y^{*})\}$. 
\item Suppose $(x^{*},y^{*})\notin M$. In this case, $M$ cannot be a singleton, because if it is, then due to the positive invariance of $M$, the point in $M$ must equal a constant solution to \eqref{ODE_1}, and therefore, a root of $h$ in $\mathcal{S}$ -- but we have shown that the only root of $h$ in $\mathcal{S}$ is $(x^{*},y^{*})$. 

Next, since $M$ is closed, there must exist some $\delta>0$ such that $B\left((x^{*},y^{*}),\delta\right)\cap M=\emptyset$, where $B\left((x^{*},y^{*}),\delta\right)$ indicates the ball of radius $\delta$ around $(x^{*},y^{*})$. Let $(x_{0},y_{0})\in M$, and since, as proved above, the $\omega$-limit set $\omega(x_{0},y_{0})=\{(x^{*},y^{*})\}$, hence there exists some $T>0$ such that, if $(x(t),y(t))$ indicates the solution to \eqref{ODE_1} initiated at $(x_{0},y_{0})$, then 
\begin{equation}
\left|\left|(x(t),y(t))-(x^{*},y^{*})\right|\right|_{1}<\delta/2 \text{ for all } t\geqslant T.\label{eq_1}
\end{equation}
Since $M$ is not a singleton, there must exist at least one other point, say $(x',y')$, in $M$. For $\epsilon=\delta/2$ and $T$ as defined via \eqref{eq_1}, we can find some $n\in\mathbb{N}$, and $(x_{1},y_{1}),(x_{2},y_{2}),\ldots,(x_{n-1},y_{n-1})\in M$, such that, setting $u_{i}=(x_{i},y_{i})$ for $i\in\{0,1,\ldots,n-1\}$ and $u_{n}=(x',y')$, the criterion stated in \eqref{I2} is satisfied. Thus, the solution to \eqref{ODE_1}, initiated at $(x_{0},y_{0})$, enters the $(\delta/2)$-neighbourhood of $(x_{1},y_{1})$ after time $T$, so that we have
\begin{equation}
\left|\left|(x(t'),y(t'))-(x_{1},y_{1})\right|\right|_{1}<\delta/2 \text{ for some } t'\geqslant T.\label{eq_2}
\end{equation}
Combining \eqref{eq_1} and \eqref{eq_2} (with $t=t'$ in \eqref{eq_1}), we see that $\left|\left|(x_{1},y_{1})-(x^{*},y^{*})\right|\right|_{1}<\delta$, even though $(x_{1},y_{1})\in M$. This brings us to a contradiction, proving that the case $(x^{*},y^{*})\notin M$ is simply not feasible.
\end{enumerate}
From Theorem~\ref{thm:main_1} and the conclusion drawn above, the sequence $\left\{\left(n^{-1}A_{n},n^{-1}B_{n}\right)\right\}$ converges almost surely to the unique root, $(x^{*},y^{*})$, of $h$ in $\mathcal{S}$ when \eqref{contraction_criterion_H_{0}} holds. This proves the first part of Theorem~\ref{thm:main_1_special}.

To prove the second part of Theorem~\ref{thm:main_1_special}, we show that if \eqref{contraction_criterion_g} holds, then so does \eqref{contraction_criterion_H_{0}}. Fix any $(u,v)\in\mathcal{S}$. By \eqref{H_{0}_defn} and the mean value theorem, for \emph{some} $\xi_{i}$ that lies between $i/k$ and $(i+1)/k$,  
\begin{align}
{}&\left|\frac{\partial H_{0}}{\partial x}\Big|_{(u,v)}\right|=\left|\sum_{k=1}^{M}\mu(k)k\sum_{i=0}^{k-1}\sum_{j=0}^{k-1-i}\left\{g\left(\frac{i+1}{k},\frac{j}{k}\right)-g\left(\frac{i}{k},\frac{j}{k}\right)\right\}{k-1\choose i}{k-1-i\choose j}u^{i}v^{j}(1-u-v)^{k-1-i-j}\right|\nonumber\\
\leqslant{}&\sum_{k=1}^{M}\mu(k)\sum_{i=0}^{k-1}\sum_{j=0}^{k-1-i}\left|\frac{\partial g}{\partial x}\Big|_{(\xi_{i},j/k)}\right|{k-1\choose i}{k-1-i\choose j}u^{i}v^{j}(1-u-v)^{k-1-i-j}\leqslant\sup\left\{\left|\frac{\partial g}{\partial x}\right|:(x,y)\in\mathcal{S}\right\}.\nonumber
\end{align}
Likewise, we can show that $|\partial H_{0}/\partial y|_{(u,v)}|\leqslant\sup\{|\partial g/\partial y|:(x,y)\in\mathcal{S}\}$. These inequalities imply that when \eqref{contraction_criterion_g} holds, the inequality in \eqref{contraction_criterion_H_{0}} is satisfied as well, and the same conclusion as asserted upon in the first part of Theorem~\ref{thm:main_1_special} follows.
\end{proof}

\begin{proof}[Proof of Theorem~\ref{thm:main_2_special}]
Theorem~\ref{thm:main_2_special} is proved the same way as Theorem~\ref{thm:main_1_special}, with the functions $\hat{H}$ and $\hat{h}$, defined in \eqref{hat{H}_hat{h}_defn}, replacing, respectively, the functions $H$ and $h$ defined in \eqref{h_H_defns}. Since one of \eqref{thm:main_2:C^{1}} and \eqref{thm:main_2:C^{2}} is true, $g\in\mathcal{C}^{(1)}(\mathcal{S})$, so that the partial derivatives of $g$ are well-defined and continuous throughout $\mathcal{S}$.

Similar to \eqref{H_maps_domain_to_itself}, we have, since $0\leqslant g(x,y)\leqslant 1$ for all $(x,y)\in\mathcal{S}$:
\begin{equation}
\min\left\{q_{1}g(x,y),q_{2}[1-g(x,y)]\right\}\geqslant 0, \quad q_{1}g(x,y)+q_{2}\{1-g(x,y)\}\leqslant \max\{q_{1},q_{2}\}\leqslant 1,\label{hat{H}_coordinates_sum_bounds}
\end{equation}
so that $\hat{H}$ maps the compact, convex $\mathcal{S}$ to itself, and must, therefore, have at least one fixed point in $\mathcal{S}$, by Brouwer's fixed point theorem. Using the submultiplicative nature of the $1$-norm and the mean value theorem, we can write, for any $(x_{1},y_{1}), (x_{2},y_{2})\in\mathcal{S}$ and for \emph{some} $(\xi_{1},\xi_{2})$ lying on the line segment joining $(x_{1},y_{1})$ and $(x_{2},y_{2})$ (which also means that $(\xi_{1},\xi_{2})\in\mathcal{S}$, since $\mathcal{S}$ is convex):
\begin{align}
{}&\left|\left|\hat{H}(x_{1},y_{1})-\hat{H}(x_{2},y_{2})\right|\right|_{1}\leqslant\left|\left|\begin{bmatrix}
q_{1}\partial g/\partial x\big|_{(\xi_{1},\xi_{2})} & q_{1}\partial g/\partial y\big|_{(\xi_{1},\xi_{2})}\\
-q_{2}\partial g/\partial x\big|_{(\xi_{1},\xi_{2})} & -q_{2}\partial g/\partial y\big|_{(\xi_{1},\xi_{2})}
\end{bmatrix}\right|\right|_{1}\left|\left|\begin{bmatrix}
x_{1}\\
y_{1}
\end{bmatrix}-\begin{bmatrix}
x_{2}\\
y_{2}
\end{bmatrix}\right|\right|_{1}\nonumber\\
={}&(q_{1}+q_{2})\max\left\{\left|\partial g/\partial x\big|_{(\xi_{1},\xi_{2})}\right|,\left|\partial g/\partial y\big|_{(\xi_{1},\xi_{2})}\right|\right\}\left|\left|\begin{bmatrix}
x_{1}\\
y_{1}
\end{bmatrix}-\begin{bmatrix}
x_{2}\\
y_{2}
\end{bmatrix}\right|\right|_{1}\nonumber\\
\leqslant{}&(q_{1}+q_{2})\max\left\{\sup\left\{\left|\partial g/\partial x\right|:(x,y)\in\mathcal{S}\right\},\sup\left\{\left|\partial g/\partial y\right|:(x,y)\in\mathcal{S}\right\}\right\}\left|\left|\begin{bmatrix}
x_{1}\\
y_{1}
\end{bmatrix}-\begin{bmatrix}
x_{2}\\
y_{2}
\end{bmatrix}\right|\right|_{1},\nonumber
\end{align}
so that $\hat{H}$ is evidently a contraction on $\mathcal{S}$ when \eqref{contraction_criterion_g} holds. By the Banach fixed point theorem, $\hat{H}$ has a unique fixed point, say $(x^{*},y^{*})$, in $\mathcal{S}$.

Using \eqref{h_H_defns}, the fact that $\hat{H}(x,y)\in\mathcal{S}$ for each $(x,y)\in\mathcal{S}$, as evident from \eqref{hat{H}_coordinates_sum_bounds}, and the fact that $\mathcal{S}$ is a convex set, we have $(x,y)+t\cdot\hat{h}(x,y)=(1-t)(x,y)+t\cdot\hat{H}(x,y)\in\mathcal{S}$ for any $(x,y)\in\mathcal{S}$. Therefore, the distance between the point $(x,y)+t\cdot\hat{h}(x,y)$ and the set $\mathcal{S}$ equals $0$ for each $t\in(0,1]$, leading to $\hat{h}(x,y)\in\mathcal{T}_{\mathcal{S}}(x,y)$ for all $(x,y)\in\mathcal{S}$, where $\mathcal{T}_{\mathcal{S}}(\cdot)$ is as defined in \eqref{tangent_cone}. Since $\hat{h}$ is Lipschitz throughout $\mathcal{S}$ (as argued in the proof of Theorem~\ref{thm:main_2}), by Remark~\ref{rem:Lipschitz_extension}, there exists an extension of $\hat{h}$ to $\tilde{h}:\mathbb{R}^{2}\rightarrow\mathbb{R}^{2}$ such that $\tilde{h}$ is Lipschitz on $\mathbb{R}^{2}$ with the same Lipschitz constant as that of $\hat{h}$ on $\mathcal{S}$. Therefore, the ODE 
\begin{equation}
\left(\dot{x}(t),\dot{y}(t)\right)=\tilde{h}\left(x(t),y(t)\right) \text{ for } t\geqslant 0, \text{ with } \left(x(0),y(0)\right)\in\mathbb{R}^{2},\label{ODE_4}
\end{equation}
is well-posed, so that by Nagumo's Theorem, we conclude that $\mathcal{S}$ is a positively invariant set corresponding to \eqref{ODE_4}, and hence, also, corresponding to \eqref{ODE_2} (since \eqref{ODE_4} boils down to \eqref{ODE_2} when $(x(0),y(0))\in\mathcal{S}$).

Similar to \eqref{Bendixson_sign}, from \eqref{hat{H}_hat{h}_defn} and \eqref{contraction_criterion_g}, we have, for any $(x,y)\in\mathcal{S}$:
\begin{align}
\frac{\partial \hat{h}_{1}}{\partial x}+\frac{\partial \hat{h}_{2}}{\partial y}=q_{1}\frac{\partial g}{\partial x}-1-q_{2}\frac{\partial g}{\partial y}-1\leqslant q_{1}\left|\frac{\partial g}{\partial x}\right|+q_{2}\left|\frac{\partial g}{\partial y}\right|-2<-1,\label{Bendixson_criterion_g}
\end{align} 
so that on the simply connected region $\mathcal{S}$, the leftmost expression of \eqref{Bendixson_criterion_g} is neither identically zero nor does it change sign, so that by Bendixson's criterion and the Poincar\'{e}-Bendixson Theorem, we conclude that the $\omega$-limit set $\omega(x,y)=\{(x^{*},y^{*})\}$ for every $(x,y)\in\mathcal{S}$. Finally, the claim that the only compact, connected, internally chain transitive invariant (in particular, positive invariant) set corresponding to \eqref{ODE_2} is the singleton set $\{(x^{*},y^{*})\}$, is established exactly as has been shown in the proof of Theorem~\ref{thm:main_1_special}.
\end{proof}

\begin{proof}[Proof of Theorem~\ref{thm:a.s._convergence_C_{n}}]
Recall, from \S\ref{sec:model}, that $C_{n}$ is the number of balls of colour $3$, up to and including epoch $n$, in the urn. Let $Y_{i}$ be the indicator for the event that the $i$-th ball added to the urn is of colour $3$. Then $C_{n}=\sum_{i=1}^{n}Y_{i}$, and \eqref{C_{n+1}_distribution} yields $\E[Y_{n+1}|\mathcal{F}_{n}]=(1-q_{1})f_{n}(n^{-1}A_{n},n^{-1}B_{n})$ for each $n\geqslant N$, where
\begin{equation}\label{f_{n}_defn_C_{n}}
\begin{cases} 
f_{n}(x,y)=H_{0}(x,y) \text{ for } (x,y)\in\mathcal{S}, &\text{ Scenario~\eqref{Scenario_1}, with replacement;}\\
f_{n}(x,y)=H_{n}(x,y) \text{ for } (x,y)\in\mathcal{S}, &\text{ Scenario~\eqref{Scenario_2}, with replacement;}\\
f_{n}\left(r_{1}/n,r_{2}/n\right)=F_{n}\left(r_{1}/n,r_{2}/n\right) \text{ for } \left(r_{1}/n,r_{2}/n\right)\in\mathcal{S}_{n}, &\text{ Scenario~\eqref{Scenario_1}, without replacement;}\\
f_{n}\left(r_{1}/n,r_{2}/n\right)=E_{n}\left(r_{1}/n,r_{2}/n\right) \text{ for }\left(r_{1}/n,r_{2}/n\right)\in \mathcal{S}_{n}, &\text{ Scenario~\eqref{Scenario_2}, without replacement;}
\end{cases}
\end{equation}
where $H_{0}$, $H_{n}$, $F_{n}$ and $E_{n}$ are as defined in \eqref{H_{0}_defn}, \eqref{H_{n}_defn}, \eqref{F_{n}_defn} and \eqref{E_{n}_defn} respectively. We then have
\begin{align}
\frac{C_{n+1}}{n+1}={}&\frac{1}{n+1}\sum_{i=0}^{N-1}Y_{i+1}+\frac{1}{n+1}\sum_{i=N}^{n}\Delta M_{i+1}+\left(\frac{n-N+1}{n+1}\right)\left(\frac{1-q_{1}}{n-N+1}\right)\sum_{i=N}^{n}f_{i}\left(\frac{A_{i}}{i},\frac{B_{i}}{i}\right).\label{C_{n}_SA}
\end{align}
It is immediate that the first term of \eqref{C_{n}_SA} converges to $0$ almost surely as $n\rightarrow\infty$. It is evident that $\{\sum_{i=1}^{n}\Delta M_{i}:n\in\mathbb{N}\}$, with $\Delta M_{i}=0$ for each $i\in[N]$ and $\Delta M_{n+1}=Y_{n+1}-(1-q_{1})f_{n}(n^{-1}A_{n},n^{-1}B_{n})$ for each $n\geqslant N$, forms a martingale with respect to $\{\mathcal{F}_{n}:n\in\mathbb{N}\}$. Elementary computations yield, along with the bounds shown in \eqref{H_{0}_bounds}, \eqref{H_{n}_bounds}, \eqref{F_{n}_bounds} and \eqref{E_{n}_bounds}:
\begin{align}
{}&\E\left[\Delta M_{n+1}^{2}\big|\mathcal{F}_{n}\right]=(1-q_{1})f_{n}\left(\frac{A_{n}}{n},\frac{B_{n}}{n}\right)\left\{1-(1-q_{1})f_{n}\left(\frac{A_{n}}{n},\frac{B_{n}}{n}\right)\right\}\leqslant\frac{1}{4}\implies \sum_{n}\frac{\E\left[\Delta M_{n+1}^{2}\big|\mathcal{F}_{n}\right]}{(n+1)^{2}} \text{ converges a.s.},\nonumber
\end{align}
so that by [\cite{hall2014martingale}, Theorem 2.18], we conclude that the second term in \eqref{C_{n}_SA} converges almost surely to $0$. 

From the urn process described in \S\ref{sec:model}, we know that $(n^{-1}A_{n},n^{-1}B_{n})$ takes values in $\mathcal{S}_{n}\subset\mathcal{S}$ for each $n\geqslant N$, where each $\mathcal{S}_{n}$ (as defined in \eqref{S_{n}_defn}) as well as $\mathcal{S}$ (as defined in \eqref{domain_defn}) is a compact subset of $\mathbb{R}^{2}$. When in Scenario~\eqref{Scenario_1}, Theorems~\ref{thm:main_1} and \ref{thm:main_3} guarantee that $Q_{n}=(n^{-1}A_{n},n^{-1}B_{n})$ converges almost surely to some $Q=(\overline{A},\overline{B})$ which must take values in $\mathcal{S}$. When in Scenario~\eqref{Scenario_2}, this is ensured by the assumption made in the statement of Theorem~\ref{thm:a.s._convergence_C_{n}}. In what follows, we use \eqref{f_{n}_defn_C_{n}}:
\begin{enumerate}
\item When in Scenario~\eqref{Scenario_1} and the samples are drawn with replacement, setting $\mathcal{A}_{n}=\mathcal{A}=\mathcal{S}$ and $f_{n}\equiv f\equiv H_{0}$ for each $n\geqslant N$, we apply the very last assertion made in Lemma~\ref{lem:a.s.convergence} to conclude that the third term of \eqref{C_{n}_SA} converges almost surely to $(1-q_{1})H_{0}(\overline{A},\overline{B})$.
\item When in Scenario~\eqref{Scenario_1} and the samples are drawn without replacement, we set $\mathcal{A}=\mathcal{S}$ and $f\equiv H_{0}$, so that $f\in\mathcal{C}^{(0)}(\mathcal{A})$ since $H_{0}$ is a polynomial, and we set $\mathcal{A}_{n}=\mathcal{S}_{n}$ for each $n\geqslant N$. By Lemma~\ref{lem:hypergeometric_approx_sample_size_iid}, we have $\left|f_{n}(r_{1}/n,r_{2}/n)-f(r_{1}/n,r_{2}/n)\right|=\left|F_{n}(r_{1}/n,r_{2}/n)-H_{0}(r_{1}/n,r_{2}/n)\right|=O(n^{-1})$ for each $(r_{1}/n,r_{2}/n)\in\mathcal{S}_{n}$. By the first part of Lemma~\ref{lem:a.s.convergence}, we deduce that the third term of \eqref{C_{n}_SA} converges almost surely to $(1-q_{1})H_{0}(\overline{A},\overline{B})$.
\item When in Scenario~\eqref{Scenario_2} and the samples are drawn with replacement, we set $\mathcal{A}_{n}=\mathcal{A}=\mathcal{S}$ for each $n\geqslant N$, and $f\equiv g$, for $g$ as defined in \eqref{g_defn}. Each $H_{n}$, as evident from \eqref{H_{n}_defn}, is a polynomial, so that $H_{n}\in\mathcal{C}^{(0)}(\mathcal{S})$. That $\{H_{n}\}$ converges uniformly to $g$ on $\mathcal{S}$ is ensured by
\begin{enumerate}
\item \eqref{Bernstein_error_Lipschitz_H_{n}} of Lemma~\ref{lem:Bernstein_approximation_errors_Holder_continuous} and \eqref{thm:C_{n}_a.s.:g_Lipschitz} when $g$ is Lipschitz on $\mathcal{S}$;
\item \eqref{Bernstein_error_C^{1}_H_{n}} of Lemma~\ref{lem:Bernstein_approximation_errors_C^{1}} and \eqref{thm:C_{n}_a.s.:g_C^{1}} when $g\in\mathcal{C}^{(1)}(\mathcal{S})$;
\item \eqref{Bernstein_error_C^{2}_H_{n}} of Lemma~\ref{lem:Bernstein_approximation_errors_C^{2}} and \eqref{thm:C_{n}_a.s.:g_C^{2}} when $g\in\mathcal{C}^{(2)}(\mathcal{S})$.
\end{enumerate}
By the second part of Lemma~\ref{lem:a.s.convergence}, we deduce that the third term of \eqref{C_{n}_SA} converges almost surely to $(1-q_{1})g(\overline{A},\overline{B})$. 
\item Finally, when in Scenario~\eqref{Scenario_2} and the samples are drawn without replacement, we set $\mathcal{A}=\mathcal{S}$, $f\equiv g$ and $\mathcal{A}_{n}=\mathcal{S}_{n}$ for each $n\geqslant N$. Since one of \eqref{thm:C_{n}_a.s.:g_Lipschitz}, \eqref{thm:C_{n}_a.s.:g_C^{1}} and \eqref{thm:C_{n}_a.s.:g_C^{2}} is true, we have $g\in\mathcal{C}^{(0)}(\mathcal{S})$. Referring to \eqref{uniform_convergence_criterion} of Lemma~\ref{lem:a.s.convergence}, we observe that 
\begin{enumerate}
\item $\beta_{n}=2^{-1/2}L\E[K_{n}^{-1/2}]$, by \eqref{Bernstein_error_Lipschitz_E_{n}} of Lemma~\ref{lem:Bernstein_approximation_errors_Holder_continuous}, if $g$ is Lipschitz on $\mathcal{S}$, and $\beta_{n}\rightarrow0$ by \eqref{thm:C_{n}_a.s.:g_Lipschitz};
\item $\beta_{n}=C\min\{\omega(\nabla g;{\E[K_{n}^{-1}]}^{1/2}){\E[K_{n}^{-1}]}^{1/2},\omega(\nabla g;\E[K_{n}^{-1}]/\E[K_{n}^{-1/2}])\E[K_{n}^{-1/2}]\}$, by \eqref{Bernstein_error_C^{1}_E_{n}} of Lemma~\ref{lem:Bernstein_approximation_errors_C^{1}}, when $g\in\mathcal{C}^{(1)}(\mathcal{S})$, and $\beta_{n}\rightarrow0$ by \eqref{thm:C_{n}_a.s.:g_C^{1}};
\item $\beta_{n}=C\E\left[K_{n}^{-1}\right]$, by \eqref{Bernstein_error_C^{2}_E_{n}} of Lemma~\ref{lem:Bernstein_approximation_errors_C^{2}}, when $g\in\mathcal{C}^{(2)}(\mathcal{S})$, and $\beta_{n}\rightarrow0$ by \eqref{thm:C_{n}_a.s.:g_C^{2}}.
\end{enumerate}
By the first part of Lemma~\ref{lem:a.s.convergence}, we now conclude that the third term of \eqref{C_{n}_SA} converges almost surely to $(1-q_{1})g(\overline{A},\overline{B})$.
\end{enumerate}

In particular, if $(\overline{A},\overline{B})=(x^{*},y^{*})$ for some $(x^{*},y^{*})\in\mathcal{S}$, then 
\begin{enumerate*}
\item $H(x^{*},y^{*})=(x^{*},y^{*})$ for $H$ as defined in \eqref{h_H_defns} when in Scenario~\eqref{Scenario_1}, so that $(1-q_{1})H_{0}(\overline{A},\overline{B})=(1-q_{1})x^{*}/q_{1}$, and
\item $\hat{H}(x^{*},y^{*})=(x^{*},y^{*})$ for $\hat{H}$ as defined in \eqref{hat{H}_hat{h}_defn} when in Scenario~\eqref{Scenario_2}, so that $(1-q_{1})g(\overline{A},\overline{B})=(1-q_{1})x^{*}/q_{1}$.  
\end{enumerate*}
This yields the second claim made in the statement of Theorem~\ref{thm:a.s._convergence_C_{n}}.
\end{proof}

\begin{proof}[Proof of Theorem~\ref{thm:main_5}]
The principle idea is to implement Theorems~\ref{thm:zhang_2.1}, \ref{thm:zhang_2.2} and \ref{thm:zhang_2.3}, for which we set $Z_{n}=(n^{-1}A_{n},n^{-1}B_{n},n^{-1}C_{n})$ for each $n\geqslant N$. Since we are in Scenario~\eqref{Scenario_1}, and either \eqref{contraction_criterion_H_{0}} holds or $g\in\mathcal{C}^{(1)}(\mathcal{S})$ and \eqref{contraction_criterion_g} is true, by Theorems~\ref{thm:main_1_special} and \ref{thm:a.s._convergence_C_{n}}, the sequence $\{(n^{-1}A_{n},n^{-1}B_{n},n^{-1}C_{n})\}$ converges almost surely to $\theta^{*}=(x^{*},y^{*},z^{*})$, where $(x^{*},y^{*})$ is the unique root of $h$ in $\mathcal{S}$ and $z^{*}=(1-q_{1})x^{*}/q_{1}$.

When the samples are drawn with replacement, we have, from \eqref{sa_1} and the first row of \eqref{f_{n}_defn_C_{n}}:
\begin{align}
\begin{bmatrix}
(n+1)^{-1}A_{n+1}\\
(n+1)^{-1}B_{n+1}\\
(n+1)^{-1}C_{n+1}
\end{bmatrix}={}&\begin{bmatrix}
n^{-1}A_{n}\\
n^{-1}B_{n}\\
n^{-1}C_{n}
\end{bmatrix}+\frac{1}{n+1}\left\{\begin{bmatrix}
\Delta M_{n+1,1}\\
\Delta M_{n+1,2}\\
\Delta M_{n+1,3}
\end{bmatrix}+G\left(\begin{bmatrix}
n^{-1}A_{n}\\
n^{-1}B_{n}\\
n^{-1}C_{n}
\end{bmatrix}\right)\right\},\label{sa_1_appended}
\end{align}
where $\Delta M_{n+1,1}$, $\Delta M_{n+1,2}$ are as defined in \eqref{martingale_difference_defn_scenario_1_with_replacement}, and for $\mathcal{T}=\{(x,y,z)\in[0,1]^{3}:x+y+z\leqslant 1\}$, we set
\begin{equation}\label{G_defn}
G(x,y,z)=\begin{bmatrix}
q_{1}H_{0}(x,y)-x\\
q_{2}\{1-H_{0}(x,y)\}-y\\
(1-q_{1})H_{0}(x,y)-z
\end{bmatrix} \text{ for }(x,y,z)\in\mathcal{T}, \ \Delta M_{n+1,3}=C_{n+1}-C_{n}-(1-q_{1})H_{0}\left(\frac{A_{n}}{n},\frac{B_{n}}{n}\right).
\end{equation}
From the previous paragraph, it is evident that if either \eqref{contraction_criterion_H_{0}} or \eqref{contraction_criterion_g} (when $g\in\mathcal{C}^{(1)}(\mathcal{S})$) is true, $\theta^{*}=(x^{*},y^{*},z^{*})$ is the unique root of $G$ in $\mathcal{T}$.

When the samples are drawn without replacement, we have, from \eqref{sa_3} and the third row of \eqref{f_{n}_defn_C_{n}}:
\begin{align}
\begin{bmatrix}
(n+1)^{-1}A_{n+1}\\
(n+1)^{-1}B_{n+1}\\
(n+1)^{-1}C_{n+1}
\end{bmatrix}={}&\begin{bmatrix}
n^{-1}A_{n}\\
n^{-1}B_{n}\\
n^{-1}C_{n}
\end{bmatrix}+\frac{1}{n+1}\left\{\begin{bmatrix}
\Delta M_{n+1,1}\\
\Delta M_{n+1,2}\\
\Delta M_{n+1,3}
\end{bmatrix}+\begin{bmatrix}
\delta_{n,1}\\
\delta_{n,2}\\
\delta_{n,3}
\end{bmatrix}+G\left(\begin{bmatrix}
n^{-1}A_{n}\\
n^{-1}B_{n}\\
n^{-1}C_{n}
\end{bmatrix}\right)\right\},\label{sa_2_appended}
\end{align}
where $\Delta M_{n+1,1}$, $\Delta M_{n+1,2}$, $\delta_{n,1}$ and $\delta_{n,2}$ are as defined in \eqref{many_defns_without_replacement_identical_law}, $G$ as defined in \eqref{G_defn}, and we set
\begin{equation}\label{third_error_term_without_replacement_identical_law}
\Delta M_{n+1,3}=C_{n+1}-C_{n}-(1-q_{1})F_{n}\left(\frac{A_{n}}{n},\frac{B_{n}}{n}\right) \text{ and } \delta_{n,3}=(1-q_{1})\left\{F_{n}\left(\frac{A_{n}}{n},\frac{B_{n}}{n}\right)-H_{0}\left(\frac{A_{n}}{n},\frac{B_{n}}{n}\right)\right\}.
\end{equation}

\textbf{Verification of Assumption~\eqref{gen_assump_dist_1}:} Since $H_{0}$, defined in \eqref{H_{0}_defn}, is a polynomial, hence $G\in\mathcal{C}^{(1)}(\mathcal{T})$, with
\begin{equation}\label{Jacobian_G_scenario_1}
J_{G}(x,y,z)=\begin{bmatrix}
q_{1}\partial H_{0}/\partial x-1 & q_{1}\partial H_{0}/\partial y & 0\\
-q_{2}\partial H_{0}/\partial x & -q_{2}\partial H_{0}/\partial y-1 & 0\\
(1-q_{1})\partial H_{0}/\partial x & (1-q_{1})\partial H_{0}/\partial y & -1
\end{bmatrix} \text{ for each }(x,y,z)\in\mathcal{T},
\end{equation} 
with eigenvalues (both of which are real)
\begin{equation}\label{gen_eigenvalues_Jacobian_G_scenario_1}
\begin{cases}
\lambda_{1}(x,y,z)=-1 &\text{ with algebraic multiplicity } 2,\\
\lambda_{2}(x,y,z)=-1+q_{1}\partial H_{0}/\partial x-q_{2}\partial H_{0}/\partial y &\text{ with algebraic multiplicity } 1.
\end{cases}
\end{equation}
By \eqref{contraction_criterion_H_{0}}, we have $\lambda_{2}(x,y,z)\leqslant -1+q_{1}|\partial H_{0}/\partial x|+q_{2}|\partial H_{0}/\partial y|<0$. In particular, this shows that each of $\lambda_{1}(\theta^{*})$ and $\lambda_{2}(\theta^{*})$ is strictly negative. This validates Assumption~\eqref{gen_assump_dist_1}.

\textbf{Verification of Assumption~\eqref{gen_assump_dist_2}:} Since $H_{0}$ is a polynomial, we have $H_{0}\in\mathcal{C}^{(2)}(\mathcal{S})$, and since $\mathcal{S}$ is compact, we have $M_{i,j}<\infty$ for each $(i,j)\in\{(1,1),(1,2),(2,2)\}$, where we define
\begin{equation}
M_{i,j}=\sup\left\{\left|\frac{\partial^{2}H_{0}}{\partial x_{i}\partial x_{j}}\right|:(x_{1},x_{2})\in\mathcal{S}\right\}.\nonumber
\end{equation}
A Taylor expansion of $H_{0}(x,y)$ around $(x^{*},y^{*})$ for any $(x,y,z)\in\mathcal{T}$, along with the AM-GM inequality, shows that $||G(x,y,z)-J_{G}(\theta^{*})[(x,y,z)-\theta^{*}]^{t}||_{2}$ is bounded above by $\{q_{1}^{2}+q_{2}^{2}+(1-q_{1})^{2}\}^{1/2}\{(M_{1,1}+M_{1,2})(x-x^{*})^{2}+(M_{1,2}+M_{2,2})(y-y^{*})^{2}\}/2$, which is $o(||(x,y,z)-\theta^{*}||_{2}^{1+\eta})$ for any $0<\eta<1$.

\textbf{Verification of Assumption~\eqref{gen_assump_dist_3}:} We first verify that \eqref{Lindeberg_cond} is satisfied for each $\epsilon>0$. Setting $\Delta M_{n+1}=(\Delta M_{n+1,1},\Delta M_{n+1,2},\Delta M_{n+1,3})^{t}$ for each $n\geqslant N$, we have, for any $\epsilon>0$:
\begin{align}
{}&\frac{1}{n-N}\sum_{i=N}^{n-1}\E\left[\left|\left|\Delta M_{i+1}\right|\right|_{2}^{2}\chi\left\{\left|\left|\Delta M_{i+1}\right|\right|_{2}\geqslant \epsilon\sqrt{n}\right\}\Big|\mathcal{F}_{i}\right]\nonumber\\
={}&\frac{\epsilon^{-2} n^{-1}}{n-N}\sum_{i=N}^{n-1}\E\left[\epsilon^{2}n\left|\left|\Delta M_{i+1}\right|\right|_{2}^{2}\chi\left\{\left|\left|\Delta M_{i+1}\right|\right|_{2}\geqslant \epsilon\sqrt{n}\right\}\Big|\mathcal{F}_{i}\right]\nonumber\\
\leqslant{}&\frac{\epsilon^{-2} n^{-1}}{n-N}\sum_{i=N}^{n-1}\E\left[\left|\left|\Delta M_{i+1}\right|\right|_{2}^{4}\chi\left\{\left|\left|\Delta M_{i+1}\right|\right|_{2}\geqslant \epsilon\sqrt{n}\right\}\Big|\mathcal{F}_{i}\right]\leqslant \frac{\epsilon^{-2} n^{-1}}{n-N}\sum_{i=N}^{n-1}\E\left[\left|\left|\Delta M_{i+1}\right|\right|_{2}^{4}\Big|\mathcal{F}_{i}\right]\nonumber\\
={}&\frac{3\epsilon^{-2} n^{-1}}{n-N}\sum_{i=N}^{n-1}\left\{\E\left[\Delta M_{i+1,1}^{4}\Big|\mathcal{F}_{i}\right]+\E\left[\Delta M_{i+1,2}^{4}\Big|\mathcal{F}_{i}\right]+\E\left[\Delta M_{i+1,3}^{4}\Big|\mathcal{F}_{i}\right]\right\} \quad \text{by AM-GM inequality.}\label{intermediate_15}
\end{align}
Since $\Delta M_{i+1,1}=A_{i+1}-A_{i}-q_{1}f_{i}(i^{-1}A_{i},i^{-1}B_{i})$ equals either $1-q_{1}f_{i}(i^{-1}A_{i},i^{-1}B_{i})$ or $-q_{1}f_{i}(i^{-1}A_{i},i^{-1}B_{i})$, where $f_{n}$ is given by either the first or the third row of \eqref{f_{n}_defn_C_{n}},  we conclude, by either \eqref{H_{0}_bounds} or \eqref{F_{n}_bounds}, that $|\Delta M_{i+1,1}|\leqslant 1$. The same bound is true for each of $|\Delta M_{i+1,2}|$ and $|\Delta M_{i+1,3}|$. Applying these bounds to the final expression in \eqref{intermediate_15}, we obtain the final bound of $9\epsilon^{-2}n^{-1}$, ensuring that \eqref{Lindeberg_cond} holds in this set-up. 

Next, we verify that \eqref{MDS_quadratic_variation_Cesaro_conv} is satisfied. Identities analogous to \eqref{quadratic_variation_with_replacement_scenario_1}, and computations utilizing the fact that at most one of $(A_{n+1}-A_{n})$, $(B_{n+1}-B_{n})$ and $(C_{n+1}-C_{n})$ can be non-zero, reveal that the element in the $i$-th row and the $j$-th column of $\E[\Delta M_{n+1}\Delta M_{n+1}^{t}|\mathcal{F}_{n}]$ equals: 
\begin{equation}\label{MDS_squared_entries}
\begin{cases}
q_{1}f_{n}\left(n^{-1}A_{n},n^{-1}B_{n}\right)\left\{1-q_{1}f_{n}\left(n^{-1}A_{n},n^{-1}B_{n}\right)\right\} &\text{ if }i=j=1,\\
-q_{1}q_{2}f_{n}\left(n^{-1}A_{n},n^{-1}B_{n}\right)\left\{1-f_{n}\left(n^{-1}A_{n},n^{-1}B_{n}\right)\right\} &\text{ if }\{i,j\}=\{1,2\},\\
q_{2}\left\{1-f_{n}\left(n^{-1}A_{n},n^{-1}B_{n}\right)\right\}\left[1-q_{2}\left\{1-f_{n}\left(n^{-1}A_{n},n^{-1}B_{n}\right)\right\}\right] &\text{ if }i=j=2,\\
-q_{1}(1-q_{1})f_{n}^{2}\left(n^{-1}A_{n},n^{-1}B_{n}\right) &\text{ if }\{i,j\}=\{1,3\},\\
-q_{2}(1-q_{1})f_{n}\left(n^{-1}A_{n},n^{-1}B_{n}\right)\left\{1-f_{n}\left(n^{-1}A_{n},n^{-1}B_{n}\right)\right\} &\text{ if }\{i,j\}=\{2,3\},\\
(1-q_{1})f_{n}\left(n^{-1}A_{n},n^{-1}B_{n}\right)\left\{1-(1-q_{1})f_{n}\left(n^{-1}A_{n},n^{-1}B_{n}\right)\right\} &\text{ if }i=j=3.
\end{cases}
\end{equation} 
When samples are drawn with replacement, by the first row of \eqref{f_{n}_defn_C_{n}}, we have $f_{n}\equiv H_{0}$, continuous throughout the compact set $\mathcal{S}$, and $(n^{-1}A_{n},n^{-1}B_{n})$ converges almost surely to $(x^{*},y^{*})$. By the last assertion made in the statement of Lemma~\ref{lem:a.s.convergence}, and since $(x^{*},y^{*})$ is a fixed point of $H$ defined in \eqref{h_H_defns}, we conclude that
\begin{align}
\frac{1}{n-N}\sum_{i=N}^{n-1}\E\left[(1,1)\text{-th entry of }\Delta M_{i+1}\Delta M_{i+1}^{t}\big|\mathcal{F}_{i}\right]\convas q_{1}H_{0}(x^{*},y^{*})\left\{1-q_{1}H_{0}(x^{*},y^{*})\right\}=x^{*}(1-x^{*}).\label{(1,1)-th_entry_limit}
\end{align}
When samples are drawn without replacement, by the third row of \eqref{f_{n}_defn_C_{n}}, we have $f_{n}\equiv F_{n}$ on $\mathcal{S}_{n}$ for each $n\geqslant N$. By Lemma~\ref{lem:hypergeometric_approx_sample_size_iid}, \eqref{F_{n}_bounds}, \eqref{H_{0}_bounds}, we have, for each $(r_{1}/n,r_{2}/n)\in\mathcal{S}_{n}$:
\begin{align}
{}&\left|q_{1}F_{n}\left(\frac{r_{1}}{n},\frac{r_{2}}{n}\right)\left\{1-q_{1}F_{n}\left(\frac{r_{1}}{n},\frac{r_{2}}{n}\right)\right\}-q_{1}H_{0}\left(\frac{r_{1}}{n},\frac{r_{2}}{n}\right)\left\{1-q_{1}H_{0}\left(\frac{r_{1}}{n},\frac{r_{2}}{n}\right)\right\}\right|\nonumber\\
\leqslant{}&q_{1}\left|F_{n}\left(\frac{r_{1}}{n},\frac{r_{2}}{n}\right)-H_{0}\left(\frac{r_{1}}{n},\frac{r_{2}}{n}\right)\right|+q_{1}^{2}\left|F_{n}^{2}\left(\frac{r_{1}}{n},\frac{r_{2}}{n}\right)-H_{0}^{2}\left(\frac{r_{1}}{n},\frac{r_{2}}{n}\right)\right|\leqslant \beta_{n},\nonumber
\end{align}
where $\beta_{n}=q_{1}(1+2q_{1})O(n^{-1})$ approaches $0$ as $n\rightarrow\infty$. Note that $(n^{-1}A_{n},n^{-1}B_{n})$ takes values in $\mathcal{S}_{n}$, and its almost sure limit, $(x^{*},y^{*})$, is in $\mathcal{S}$, and the function $(x,y)\mapsto q_{1}H_{0}(x,y)\{1-q_{1}H_{0}(x,y)\}$ is continuous throughout $\mathcal{S}$. By the first part of Lemma~\ref{lem:a.s.convergence}, we conclude that here too, \eqref{(1,1)-th_entry_limit} is true.

Via similar arguments applied to the remaining rows of \eqref{MDS_squared_entries}, we now conclude that 
\begin{align}\label{MDS_squared_convergence_to_Gamma}
\frac{1}{n-N}\sum_{i=N}^{n-1}\E\left[\Delta M_{i+1}\Delta M_{i+1}^{t}\big|\mathcal{F}_{i}\right]\convas\Gamma, \text{ where } \Gamma \text{ is as defined in \eqref{thm:main_5:matrix_defns}}.
\end{align} 
It is now easily verified (keeping in mind that $(n^{-1}A_{n},n^{-1}B_{n},n^{-1}C_{n})$, and hence its almost sure limit $(x^{*},y^{*},z^{*})$, must belong to the set $\mathcal{T}$ defined just before \eqref{G_defn}) that each principal minor of $\Gamma$ is non-negative: for instance, the principal minor obtained by deleting the third row and third column equals $x^{*}y^{*}(1-x^{*}-y^{*})$, while the determinant of $\Gamma$ itself equals $x^{*}y^{*}z^{*}(1-x^{*}-y^{*}-z^{*})$. Therefore, $\Gamma$ is a symmetric, positive semi-definite matrix, thus completing the verification of \eqref{MDS_quadratic_variation_Cesaro_conv}.

\textbf{Determination of the scalars $\rho$ and $\nu$, and the matrices $T$ and $\overline{T}$, under Scenario~\eqref{Scenario_1}:} From \eqref{gen_eigenvalues_Jacobian_G_scenario_1} and the definition of $\rho$ stated right after \eqref{similarity_eq_general}, we see that $\rho$ indeed equals the expression given by \eqref{thm:main_5:scalar_defns} in the set-up of Theorem~\ref{thm:main_5}. Thus, $\rho\leqslant 1$ \emph{always}. When $\lambda_{1}(\theta^{*})\neq\lambda_{2}(\theta^{*})$, simple computations show that $e_{1}=\left(-\beta^{*},\alpha^{*},0\right)^{t}$ and $e_{2}=(0,0,1)^{t}$ serve as eigenvectors for $J_{G}(\theta^{*})$ corresponding to $\lambda_{1}(\theta^{*})$, while $e_{3}=(-q_{1},q_{2},-1+q_{1})^{t}$ serves as an eigenvector for $J_{G}(\theta^{*})$ corresponding to $\lambda_{2}(\theta^{*})$. Thus, when $\lambda_{1}(\theta^{*})\neq\lambda_{2}(\theta^{*})$, each eigenvalue of $J_{G}(\theta^{*})$ is regular, so that $J_{G}(\theta^{*})$ is similar to a diagonal matrix:
\begin{equation}
T^{-1}J_{G}(\theta^{*})T=\begin{bmatrix}
-1 & 0 & 0\\
0 & -1 & 0\\
0 & 0 & -1+q_{1}\alpha^{*}-q_{2}\beta^{*}
\end{bmatrix}, \quad\text{with }T \text{ as defined in \eqref{thm:main_5:matrix_defns}}.\label{similarity_eq_J_{G}(theta^{*})}
\end{equation}  
This is true, in particular, when $\rho<1$, or, equivalently, $q_{1}\alpha^{*}>q_{2}\beta^{*}$, in which case $\rho=-\lambda_{2}(\theta^{*})=1-q_{1}\alpha^{*}+q_{2}\beta^{*}$.
From the definition of $\nu$, stated right before Theorem~\ref{thm:zhang_2.1}, we have $\nu=1$ whenever $\rho<1$.

Consider $\rho=1$. If $\lambda_{2}(\theta^{*})\neq\lambda_{1}(\theta^{*})$, which is equivalent to having $q_{1}\alpha^{*}\neq q_{2}\beta^{*}$, then, once again, each eigenvalue is regular and $J_{G}(\theta^{*})$ satisfies \eqref{similarity_eq_J_{G}(theta^{*})}. If $\lambda_{2}(\theta^{*})=\lambda_{1}(\theta^{*})$, or, equivalently, $q_{1}\alpha^{*}=q_{2}\beta^{*}$, and additionally, $\alpha^{*}=0$, which forces $\beta^{*}=0$, we have $J_{G}(\theta^{*})=-I_{3}$ (where $I_{d}$ indicates the $d\times d$ identity matrix). On the other hand, if $q_{1}\alpha^{*}=q_{2}\beta^{*}$ and $\alpha^{*}\neq 0$, which forces $\beta^{*}\neq 0$ as well, the only eigenvalue of $J_{G}(\theta^{*})$ is $-1$, with algebraic multiplicity $3$ and geometric multiplicity $2$. Standard results pertaining to Jordan canonical forms, along with an investigation of the eigenspace of $J_{G}(\theta^{*})$, corresponding to the eigenvalue $-1$, reveals that, for the non-singular $\overline{T}$ defined in \eqref{thm:main_5:matrix_defns}, we have
\begin{equation}\label{similarity_eq_Jordan_J_{G}(theta^{*})}
\overline{T}^{-1}J_{G}(\theta^{*})\overline{T}=\begin{bmatrix}
-1 & 1 & 0\\
0 & -1 & 0\\
0 & 0 & -1
\end{bmatrix}.
\end{equation}

\textbf{The proof of \eqref{thm:main_5:regime_1}:} Here, we are in the regime of $\rho=1/2$. When the samples are drawn with replacement, we have $\delta_{n,1}=\delta_{n,2}=\delta_{n,3}=0$ for each $n\geqslant N$, from \eqref{sa_1_appended}. Therefore, each of \eqref{series_conv_zhang_2.1_1} and \eqref{series_conv_zhang_2.1_2} is automatically true. When the samples are drawn without replacement, we have, from \eqref{sa_2_appended}, $\delta_{n}=(\delta_{n,1},\delta_{n,2},\delta_{n,3})^{t}$, with $\delta_{n,i}$, $i\in\{1,2,3\}$, as defined in \eqref{many_defns_without_replacement_identical_law} and \eqref{third_error_term_without_replacement_identical_law}. By Lemma~\ref{lem:hypergeometric_approx_sample_size_iid}, we have $||\delta_{n}||_{2}=O(n^{-1})$. This implies that 
\begin{equation}
\left|\left|\sum_{i=N}^{n-1}\delta_{i}\right|\right|_{2}\leqslant\sum_{i=N}^{n-1}\left|\left|\delta_{i}\right|\right|_{2}=O\left(\sum_{i=N}^{n-1}i^{-1}\right)=O(\log n)=o\left(\sqrt{\frac{n}{\log n}}\right),\label{partial_sum_delta_{i}_without_replacement_identical_law}
\end{equation}
thus ensuring that \eqref{series_conv_zhang_2.1_1} is true. Therefore, we may apply Theorem~\ref{thm:zhang_2.1} to this set-up. Since $\rho=1/2<1$, we have $\nu=1$, and \eqref{similarity_eq_J_{G}(theta^{*})} is true. We now carry out a computation that holds whenever \eqref{similarity_eq_J_{G}(theta^{*})} is true (in particular, when $\lambda_{1}(\theta^{*})\neq\lambda_{2}(\theta^{*})$), with $T$ and $\Gamma$ as defined in \eqref{thm:main_5:matrix_defns}:
\begin{align}\label{gen_computation_T^{t}Gamma T}
{}&\exp\left\{\left(J_{G}(\theta^{*})+I_{3}/2\right)u\right\}\Gamma\left[\exp\left\{\left(J_{G}(\theta^{*})+I_{3}/2\right)u\right\}\right]^{t}\nonumber\\
={}&T\begin{bmatrix}
e^{-u/2} & 0 & 0\\
0 & e^{-u/2} & 0\\
0 & 0 & e^{u(-1/2+q_{1}\alpha^{*}-q_{2}\beta^{*})}
\end{bmatrix}T^{-1}\Gamma (T^{-1})^{t}\begin{bmatrix}
e^{-u/2} & 0 & 0\\
0 & e^{-u/2} & 0\\
0 & 0 & e^{u(-1/2+q_{1}\alpha^{*}-q_{2}\beta^{*})}
\end{bmatrix}T^{t}\nonumber\\
={}&T\begin{bmatrix}
e^{-u}A_{1,1} & e^{-u}A_{1,2} & e^{u(-1+q_{1}\alpha^{*}-q_{2}\beta^{*})}A_{1,3}\\
e^{-u}A_{1,2} & e^{-u}A_{2,2} & e^{u(-1+q_{1}\alpha^{*}-q_{2}\beta^{*})}A_{2,3}\\
e^{u(-1+q_{1}\alpha^{*}-q_{2}\beta^{*})}A_{1,3} & e^{u(-1+q_{1}\alpha^{*}-q_{2}\beta^{*})}A_{2,3} & e^{u(-1+2q_{1}\alpha^{*}-2q_{2}\beta^{*})}A_{3,3}
\end{bmatrix}T^{t},
\end{align}
where $A_{1,1},\ldots,A_{3,3}$ are as defined in the statement of Theorem~\ref{thm:main_5}. In particular, when $\rho=1/2$, we have $q_{1}\alpha^{*}-q_{2}\beta^{*}=1/2$, so that, using \eqref{gen_computation_T^{t}Gamma T}, the expression for $\Sigma$, given by \eqref{sigma:zhang_2.1}, can be reduced to:
\begin{align}
\Sigma={}&\lim_{n\rightarrow\infty}\frac{1}{\log n}\int_{0}^{\log n}\exp\left\{\left(J_{G}(\theta^{*})+I_{3}/2\right)u\right\}\Gamma\left[\exp\left\{\left(J_{G}(\theta^{*})+I_{3}/2\right)u\right\}\right]^{t}du\nonumber\\
={}&T\left[\lim_{n\rightarrow\infty}\frac{1}{\log n}\int_{0}^{\log n}\begin{bmatrix}
e^{-u}A_{1,1} & e^{-u}A_{1,2} & e^{-u/2}A_{1,3}\\
e^{-u}A_{1,2} & e^{-u}A_{2,2} & e^{-u/2}A_{2,3}\\
e^{-u/2}A_{1,3} & e^{-u/2}A_{2,3} & A_{3,3}
\end{bmatrix}du\right]T^{t}=T\begin{bmatrix}
0 & 0 & 0\\
0 & 0 & 0\\
0 & 0 & A_{3,3}
\end{bmatrix}T^{t},\nonumber
\end{align}
where each $A_{i,j}$ is now written with $\kappa=1/2$. This, upon simplification, yields exactly the covariance matrix appearing in \eqref{thm:main_5:regime_1:conclusion}. The final conclusion, given by \eqref{thm:main_5:regime_1:conclusion}, now follows from \eqref{conclusion:zhang_2.1} of Theorem~\ref{thm:zhang_2.1}.
 
\textbf{The proof of \eqref{thm:main_5:regime_2}:} Here, we are in the regime of $0<\rho<1/2$. Since \eqref{MDS_squared_convergence_to_Gamma} holds, \eqref{series_conv_zhang_2.2_1} is true as well. When the samples are drawn with replacement, $\delta_{n,1}=\delta_{n,2}=\delta_{n,3}=0$ from \eqref{sa_1_appended}, so that \eqref{series_conv_zhang_2.2_2} is trivially satisfied. When the samples are drawn without replacement, as before, $||\delta_{n}||_{2}=(\delta_{n,1}^{2}+\delta_{n,2}^{2}+\delta_{n,3}^{2})^{1/2}=O(n^{-1})$ due to \eqref{sa_2_appended}, \eqref{many_defns_without_replacement_identical_law}, \eqref{third_error_term_without_replacement_identical_law} and Lemma~\ref{lem:hypergeometric_approx_sample_size_iid}. Similar to \eqref{partial_sum_delta_{i}_without_replacement_identical_law}, we have $||\sum_{i=N}^{n-1}\delta_{i}||_{2}=O(\log n)=o(n^{1-\rho-\epsilon})$ for any choice of $\epsilon\in(0,1-\rho)$, ensuring that \eqref{series_conv_zhang_2.2_2} is true. Thus, we may apply Theorem~\ref{thm:zhang_2.2} here. Since $\rho<1$, we have $\nu=1$, \eqref{similarity_eq_J_{G}(theta^{*})} is true and $\rho=-\lambda_{2}(\theta^{*})=1-q_{1}\alpha^{*}+q_{2}\beta^{*}$. By \eqref{conclusion:zhang_2.2}, we conclude about the existence of a finite, scalar-valued random variable $W$ such that \eqref{thm:main_5:regime_2:conclusion} is true. 

\textbf{The proof of \eqref{thm:main_5:regime_3}:} Here, we are in the regime of $\rho>1/2$, which, by \eqref{thm:main_5:scalar_defns}, is equivalent to having $q_{1}\alpha^{*}-q_{2}\beta^{*}<1/2$. When the samples are drawn with replacement, \eqref{series_conv_zhang_2.3} is automatically satisfied since $\delta_{n,1}=\delta_{n,2}=\delta_{n,3}=0$ from \eqref{sa_1_appended}. When the samples are drawn without replacement, similar to \eqref{partial_sum_delta_{i}_without_replacement_identical_law}, we have $||\sum_{i=N}^{n-1}\delta_{i}||_{2}=O(\log n)=o(\sqrt{n})$, due to \eqref{sa_2_appended}, \eqref{many_defns_without_replacement_identical_law}, \eqref{third_error_term_without_replacement_identical_law} and Lemma~\ref{lem:hypergeometric_approx_sample_size_iid}. Hence, \eqref{series_conv_zhang_2.3} is true, and we may apply Theorem~\ref{thm:zhang_2.3} here. 

Recall that, since \eqref{contraction_criterion_H_{0}} is true, we must have $q_{1}\alpha^{*}-q_{2}\beta^{*}>-1$. First, we consider the case of $q_{1}\alpha^{*}-q_{2}\beta^{*}\in(-1,1/2)\setminus\{0\}$, so that $\lambda_{1}(\theta^{*})\neq\lambda_{2}(\theta^{*})$ by \eqref{gen_eigenvalues_Jacobian_G_scenario_1}, and hence \eqref{similarity_eq_J_{G}(theta^{*})} remains true with $T$ as defined in \eqref{thm:main_5:matrix_defns}. Using \eqref{gen_computation_T^{t}Gamma T}, we may compute $\Sigma$ as defined in \eqref{sigma:zhang_2.3} as follows:
\begin{align}
\Sigma={}&\int_{0}^{\infty}\exp\left\{\left(J_{G}(\theta^{*})+I_{3}/2\right)u\right\}\Gamma\left[\exp\left\{\left(J_{G}(\theta^{*})+I_{3}/2\right)u\right\}\right]^{t}du\nonumber\\
={}&T\left(\int_{0}^{\infty}\begin{bmatrix}
e^{-u}A_{1,1} & e^{-u}A_{1,2} & e^{u(-1+q_{1}\alpha^{*}-q_{2}\beta^{*})}A_{1,3}\\
e^{-u}A_{1,2} & e^{-u}A_{2,2} & e^{u(-1+q_{1}\alpha^{*}-q_{2}\beta^{*})}A_{2,3}\\
e^{u(-1+q_{1}\alpha^{*}-q_{2}\beta^{*})}A_{1,3} & e^{u(-1+q_{1}\alpha^{*}-q_{2}\beta^{*})}A_{2,3} & e^{u(-1+2q_{1}\alpha^{*}-2q_{2}\beta^{*})}A_{3,3}
\end{bmatrix}du\right)T^{t}\nonumber\\
={}&T\begin{bmatrix}
A_{1,1} & A_{1,2} & (1-q_{1}\alpha^{*}+q_{2}\beta^{*})^{-1}A_{1,3}\\
A_{1,2} & A_{2,2} & (1-q_{1}\alpha^{*}+q_{2}\beta^{*})^{-1}A_{2,3}\\
(1-q_{1}\alpha^{*}+q_{2}\beta^{*})^{-1}A_{1,3} & (1-q_{1}\alpha^{*}+q_{2}\beta^{*})^{-1}A_{2,3} & (1-2q_{1}\alpha^{*}+2q_{2}\beta^{*})^{-1}A_{3,3}
\end{bmatrix}T^{t},\nonumber
\end{align}
so that the final conclusion, i.e.\eqref{thm:main_5:regime_3.1:conclusion} now follows from \eqref{conclusion:zhang_2.3} of Theorem~\ref{thm:zhang_2.3}. 

When $q_{1}\alpha^{*}-q_{2}\beta^{*}=0$ (which yields $\rho=1$) and $\alpha^{*}=0$ (so that $\beta^{*}=0$ as well), we have $J_{G}(\theta^{*})=-I_{3}$ from \eqref{Jacobian_G_scenario_1}, so that \eqref{sigma:zhang_2.3} immediately yields $\Sigma=\Gamma$ in this case, and \eqref{thm:main_5:regime_3.2:conclusion} follows. 

When $q_{1}\alpha^{*}-q_{2}\beta^{*}=0$ but $\alpha^{*}\neq 0$ (which forces $\beta^{*}\neq 0$ as well), \eqref{similarity_eq_Jordan_J_{G}(theta^{*})} is satisfied. A computation similar to \eqref{gen_computation_T^{t}Gamma T} yields the following form of $\Sigma$ defined in \eqref{sigma:zhang_2.3}:
\begin{align}
{}&\Sigma=\int_{0}^{\infty}\exp\left\{\left(J_{G}(\theta^{*})+I_{3}/2\right)u\right\}\Gamma\left[\exp\left\{\left(J_{G}(\theta^{*})+I_{3}/2\right)u\right\}\right]^{t}du\nonumber\\ 
={}&\overline{T}\left(\int_{0}^{\infty}\begin{bmatrix}
e^{-u/2} & u e^{-u/2} & 0\\
0 & e^{-u/2} & 0\\
0 & 0 & e^{-u/2}
\end{bmatrix}\overline{T}^{-1}\Gamma\left(\overline{T}^{-1}\right)^{t}\begin{bmatrix}
e^{-u/2} & 0 & 0\\
u e^{-u/2} & e^{-u/2} & 0\\
0 & 0 & e^{-u/2}
\end{bmatrix}du\right)\overline{T}^{t}\nonumber\\
={}&\overline{T}\left(\int_{0}^{\infty}\begin{bmatrix}
e^{-u}B_{1,1}+2u e^{-u}B_{1,2}+u^{2}e^{-u}B_{2,2} & e^{-u}B_{1,2}+ue^{-u}B_{2,2} & e^{-u}B_{1,3}+u e^{-u}B_{2,3}\\
e^{-u}B_{1,2}+ue^{-u}B_{2,2} & e^{-u}B_{2,2} & e^{-u}B_{2,3}\\
e^{-u}B_{1,3}+u e^{-u}B_{2,3} & e^{-u}B_{2,3} & e^{-u}B_{3,3}
\end{bmatrix}du\right)\overline{T}^{t}\nonumber\\
={}&\overline{T}\begin{bmatrix}
B_{1,1}+2B_{1,2}+2B_{2,2} & B_{1,2}+B_{2,2} & B_{1,3}+B_{2,3}\\
B_{1,2}+B_{2,2} & B_{2,2} & B_{2,3}\\
B_{1,3}+B_{2,3} & B_{2,3} & B_{3,3}
\end{bmatrix}\overline{T}^{t},\nonumber
\end{align}
where $B_{i,j}$, for $(i,j)\in\{(1,1),(1,2),(1,3),(2,2),(2,3),(3,3)\}$, are as given in the statement of Theorem~\ref{thm:main_5}. The final conclusion, i.e.\ \eqref{thm:main_5:regime_3.3:conclusion}, now follows from \eqref{conclusion:zhang_2.3} of Theorem~\ref{thm:zhang_2.3}.
\end{proof}

\begin{proof}[Proof of Theorem~\ref{thm:main_6}]
In the regime described in \eqref{thm:main_6:regime_1} (respectively, \eqref{thm:main_6:regime_2}), \eqref{thm:C_{n}_a.s.:g_C^{2}} implies $g\in\mathcal{C}^{(2)}(\mathcal{S})$, and \eqref{thm:main_6:regime_1:series_criteria} (respectively, \eqref{thm:main_6:regime_2:series_criteria}), via Lemma~\ref{lem:stronger_convergence_criteria_zhang}, implies \eqref{thm:main_2:C^{2}_series_convergence}, so that \eqref{thm:main_2:C^{2}} is true. In the regime described in \eqref{thm:main_6:regime_3}, either \eqref{thm:C_{n}_a.s.:g_C^{1}} implies $g\in\mathcal{C}^{(1)}(\mathcal{S})$ and \eqref{thm:main_6:regime_3.1:series_criteria}, via Lemma~\ref{lem:stronger_convergence_criteria_zhang}, implies \eqref{thm:main_2:C^{1}_series_convergence}, so that \eqref{thm:main_2:C^{1}} is true, or, similar to the previous two regimes, \eqref{thm:C_{n}_a.s.:g_C^{2}} implies $g\in\mathcal{C}^{(2)}(\mathcal{S})$ and \eqref{thm:main_6:regime_3.2:series_criteria}, via Lemma~\ref{lem:stronger_convergence_criteria_zhang}, implies \eqref{thm:main_2:C^{2}_series_convergence}, so that \eqref{thm:main_2:C^{2}} is true. Since we are in Scenario~\eqref{Scenario_2}, one of \eqref{thm:main_2:C^{1}} and \eqref{thm:main_2:C^{2}} is always true, and \eqref{contraction_criterion_g} holds, we conclude, via Theorem~\ref{thm:main_2_special}, that the function $\hat{h}$, defined in \eqref{hat{H}_hat{h}_defn}, has a unique root, $(x^{*},y^{*})$, in $\mathcal{S}$, and that $\{(n^{-1}A_{n},n^{-1}B_{n})\}$ converges almost surely to $(x^{*},y^{*})$. This, along with the assumption that one of \eqref{thm:C_{n}_a.s.:g_C^{1}} and \eqref{thm:C_{n}_a.s.:g_C^{2}} is true, allows us to conclude, by Theorem~\ref{thm:a.s._convergence_C_{n}}, that $\{n^{-1}C_{n}\}$ convegres almost surely to $z^{*}=(1-q_{1})x^{*}/q_{1}$. Therefore, $\{(n^{-1}A_{n},n^{-1}B_{n},n^{-1}C_{n})\}$ converges almost surely to $\theta^{*}=(x^{*},y^{*},z^{*})$.  

When the samples are drawn with replacement, we have, from \eqref{sa_2} and the second row of \eqref{f_{n}_defn_C_{n}}:
\begin{align}
\begin{bmatrix}
(n+1)^{-1}A_{n+1}\\
(n+1)^{-1}B_{n+1}\\
(n+1)^{-1}C_{n+1}
\end{bmatrix}={}&\begin{bmatrix}
n^{-1}A_{n}\\
n^{-1}B_{n}\\
n^{-1}C_{n}
\end{bmatrix}+\frac{1}{n+1}\left\{\begin{bmatrix}
\Delta M_{n+1,1}\\
\Delta M_{n+1,2}\\
\Delta M_{n+1,3}
\end{bmatrix}+\begin{bmatrix}
\delta_{n,1}\\
\delta_{n,2}\\
\delta_{n,3}
\end{bmatrix}+\hat{G}\left(\begin{bmatrix}
n^{-1}A_{n}\\
n^{-1}B_{n}\\
n^{-1}C_{n}
\end{bmatrix}\right)\right\},\label{sa_3_appended}
\end{align}
where $\Delta M_{n+1,1}$, $\Delta M_{n+1,2}$, $\delta_{n,1}$ and $\delta_{n,2}$ are as defined in \eqref{many_defns_with_replacement_scenario_2}, the function $\hat{G}:\mathcal{T}\rightarrow\mathcal{T}$is given by
\begin{equation}
\hat{G}(x,y,z)=\begin{bmatrix}
q_{1}g(x,y)-x\\
q_{2}\{1-g(x,y)\}-y\\
(1-q_{1})g(x,y)-z
\end{bmatrix}, \text{ where }\mathcal{T}\text{ is as defined just before \eqref{G_defn},}\label{hat{G}_defn}
\end{equation}
with $g$ as defined in \eqref{g_defn}, and we set, for each $n\geqslant N$, with $H_{n}$ as defined in \eqref{H_{n}_defn},
\begin{equation}\label{third_error_term_with_replacement_varying_law}
\Delta M_{n+1,3}=C_{n+1}-C_{n}-(1-q_{1})H_{n}\left(\frac{A_{n}}{n},\frac{B_{n}}{n}\right) \text{ and } \delta_{n,3}=(1-q_{1})\left\{H_{n}\left(\frac{A_{n}}{n},\frac{B_{n}}{n}\right)-g\left(\frac{A_{n}}{n},\frac{B_{n}}{n}\right)\right\}.
\end{equation}
When the samples are drawn without replacement, our stochastic process can, once again, be represented by \eqref{sa_3_appended}, but with $\Delta M_{n+1,1}$, $\Delta M_{n+1,2}$, $\delta_{n,1}$ and $\delta_{n,2}$ as defined in \eqref{many_defns_without_replacement_varying_law}, and we set, with $E_{n}$ as in \eqref{E_{n}_defn}:
\begin{equation}\label{third_error_term_without_replacement_varying_law}
\Delta M_{n+1,3}=C_{n+1}-C_{n}-(1-q_{1})E_{n}\left(\frac{A_{n}}{n},\frac{B_{n}}{n}\right) \text{ and } \delta_{n,3}=(1-q_{1})\left\{E_{n}\left(\frac{A_{n}}{n},\frac{B_{n}}{n}\right)-g\left(\frac{A_{n}}{n},\frac{B_{n}}{n}\right)\right\}.
\end{equation}
From the discussion in the first paragraph, it is evident that $\theta^{*}$ is the unique root of $\hat{G}$ in $\mathcal{T}$.

\textbf{Verification of Assumption~\eqref{gen_assump_dist_1}:} Since one of \eqref{thm:C_{n}_a.s.:g_C^{1}} and \eqref{thm:C_{n}_a.s.:g_C^{2}} is always assumed to hold in the set-up of Theorem~\ref{thm:main_6}, $g\in\mathcal{C}^{(1)}(\mathcal{S})$, so that, by \eqref{hat{G}_defn}, $\hat{G}\in\mathcal{C}^{(1)}(\mathcal{T})$, and its Jacobian, $J_{\hat{G}}(x,y,z)$, is well-defined for all $(x,y,z)\in\mathcal{T}$. The expression for $J_{\hat{G}}(x,y,z)$ is the same as that of $J_{G}(x,y,z)$ in \eqref{Jacobian_G_scenario_1}, with $\partial H_{0}/\partial x$ and $\partial H_{0}/\partial y$ replaced by $\partial g/\partial x$ and $\partial g/\partial y$ respectively. Similar to \eqref{gen_eigenvalues_Jacobian_G_scenario_1}, the eigenvalues of $J_{\hat{G}}(\theta^{*})$ are given by
$\lambda_{1}(\theta^{*})=-1$ and $\lambda_{2}(\theta^{*})=-1+q_{1}\alpha^{*}-q_{2}\beta^{*}$, with algebraic multiplicities $2$ and $1$ respectively (where $\alpha^{*}$ and $\beta^{*}$ are as defined right after \eqref{thm:main_6:regime_1:series_criteria}). Since \eqref{contraction_criterion_g} is true, $\lambda_{2}(\theta^{*})\leqslant-1+q_{1}|\alpha^{*}|+q_{2}|\beta^{*}|<0$, so that each of $\lambda_{1}(\theta^{*})$ and $\lambda_{2}(\theta^{*})$ is strictly negative, thus validating Assumption~\eqref{gen_assump_dist_1}.

\textbf{Verification of Assumption~\eqref{gen_assump_dist_2}:} Assumption~\eqref{gen_assump_dist_2} is needed only when we are in the regimes given by \eqref{thm:main_6:regime_1} and \eqref{thm:main_6:regime_2}. In both of these regimes, \eqref{thm:C_{n}_a.s.:g_C^{2}} is assumed to hold, so that $g\in\mathcal{C}^{(2)}(\mathcal{S})$, and Assumption~\eqref{gen_assump_dist_2} can be verified in exactly the same way as in the proof of Theorem~\ref{thm:main_5}.

\textbf{Verification of Assumption~\eqref{gen_assump_dist_3}:} We verify this when we are in the regime given by either \eqref{thm:main_6:regime_1} or \eqref{thm:main_6:regime_3}. Letting $\Delta M_{n+1}=(\Delta M_{n+1,1},\Delta M_{n+1,2},\Delta M_{n+1,3})^{t}$ for each $n\geqslant N$ (with the coordinates of $\Delta M_{n+1}$ defined as in \eqref{many_defns_with_replacement_scenario_2} and \eqref{third_error_term_with_replacement_varying_law} when sampling with replacement, or as in \eqref{many_defns_without_replacement_varying_law} and \eqref{third_error_term_without_replacement_varying_law} when sampling without replacement), we observe that the inequality in \eqref{intermediate_15} remains true here. Moreover, as before, we have $|\Delta M_{i+1,j}|\leqslant 1$ for each $j\in\{1,2,3\}$, argued using either \eqref{H_{n}_bounds} and the second row of \eqref{f_{n}_defn_C_{n}}, or \eqref{E_{n}_bounds} and the fourth row of \eqref{f_{n}_defn_C_{n}}. Applying these, the final expression in \eqref{intermediate_15} is bounded above by $9\epsilon^{-2}n^{-1}$, ensuring that \eqref{Lindeberg_cond} holds in this set-up.

To verify \eqref{MDS_quadratic_variation_Cesaro_conv}, we observe that \eqref{MDS_squared_entries} remains true, with
\begin{enumerate*}
\item $f_{n}\equiv H_{n}$ on $\mathcal{S}$, for each $n\geqslant N$, when the samples are drawn with replacement, and 
\item $f_{n}\equiv E_{n}$ on $\mathcal{S}_{n}$, for each $n\geqslant N$, when they are drawn without replacement.
\end{enumerate*}
We focus on the first row of \eqref{MDS_squared_entries}, and consider the following scenarios.

When we are in the regime given by \eqref{thm:main_6:regime_1}, we have $g\in\mathcal{C}^{(2)}(\mathcal{S})$ since \eqref{thm:C_{n}_a.s.:g_C^{2}} is true. When the samples are drawn with replacement, by \eqref{Bernstein_error_C^{2}_H_{n}} of Lemma~\ref{lem:Bernstein_approximation_errors_C^{2}}, the bounds in \eqref{H_{n}_bounds}, and the fact (mentioned right before \eqref{g_defn}) that $0\leqslant g(\cdot,\cdot)\leqslant 1$ on $\mathcal{S}$, we have, for all $(x,y)\in\mathcal{S}$,
\begin{align}
{}&|q_{1}H_{n}(x,y)\{1-q_{1}H_{n}(x,y)\}-q_{1}g(x,y)\{1-q_{1}g(x,y)\}|\nonumber\\
\leqslant{}&q_{1}|H_{n}(x,y)-g(x,y)|+q_{1}^{2}|H_{n}^{2}(x,y)-g^{2}(x,y)|\leqslant \beta_{n}, \text{ where }\beta_{n}=q_{1}(1+2q_{1})C\E[K_{n}^{-1}].\label{(1,1)-th_entry_convergence_rate_with_replacement_varying_law}
\end{align}
Since \eqref{thm:C_{n}_a.s.:g_C^{2}} is true, $\beta_{n}\rightarrow0$, so that $q_{1}H_{n}(\cdot,\cdot)\{1-q_{1}H_{n}(\cdot,\cdot)\}$ converges uniformly to $q_{1}g(\cdot,\cdot)\{1-q_{1}g(\cdot,\cdot)\}$ on $\mathcal{S}$. Each $H_{n}$, defined in \eqref{H_{n}_defn}, being a polynomial, we have $q_{1}H_{n}(\cdot,\cdot)\{1-q_{1}H_{n}(\cdot,\cdot)\}\in\mathcal{C}^{(0)}(\mathcal{S})$. By Lemma~\ref{lem:a.s.convergence}, and since $(x^{*},y^{*})$ is a fixed point of $\hat{H}$ defined in \eqref{hat{H}_hat{h}_defn}, we conclude that
\begin{align}
\frac{1}{n-N}\sum_{i=N}^{n-1}\E\left[(1,1)\text{-th entry of }\Delta M_{i+1}\Delta M_{i+1}^{t}\big|\mathcal{F}_{i}\right]\convas q_{1}g(x^{*},y^{*})\left\{1-q_{1}g(x^{*},y^{*})\right\}=x^{*}(1-x^{*}).\label{(1,1)-th_entry_limit_varying_law}
\end{align}
When we are in the regime given by \eqref{thm:main_6:regime_1} and the samples are drawn without replacement, by \eqref{Bernstein_error_C^{2}_E_{n}} of Lemma~\ref{lem:Bernstein_approximation_errors_C^{2}}, the bounds in \eqref{E_{n}_bounds}, and $0\leqslant g(\cdot,\cdot)\leqslant 1$ on $\mathcal{S}$, we have, for all $(r_{1}/n,r_{2}/n)\in\mathcal{S}_{n}$,
\begin{align}
{}&|q_{1}E_{n}(r_{1}/n,r_{2}/n)\{1-q_{1}E_{n}(r_{1}/n,r_{2}/n)\}-q_{1}g(r_{1}/n,r_{2}/n)\{1-q_{1}g(r_{1}/n,r_{2}/n)\}|\leqslant \beta_{n},\nonumber
\end{align}
with $\beta_{n}$ as in \eqref{(1,1)-th_entry_convergence_rate_with_replacement_varying_law}. Since \eqref{thm:C_{n}_a.s.:g_C^{2}} is true, we have 
\begin{enumerate*}
\item $\beta_{n}\rightarrow0$, and 
\item $g\in\mathcal{C}^{(2)}(\mathcal{S})$, which implies that $q_{1}g(\cdot,\cdot)\{1-q_{1}g(\cdot,\cdot)\}\in\mathcal{C}^{(0)}(\mathcal{S})$. 
\end{enumerate*}
Lemma~\ref{lem:a.s.convergence}, once again, yields the conclusion given in \eqref{(1,1)-th_entry_limit_varying_law}.

Consider the regime given by \eqref{thm:main_6:regime_3}. When \eqref{thm:C_{n}_a.s.:g_C^{2}} is true, that \eqref{(1,1)-th_entry_limit_varying_law} holds is argued exactly as above. Suppose, now, that \eqref{thm:C_{n}_a.s.:g_C^{1}} is true, instead, so that $g\in\mathcal{C}^{(1)}(\mathcal{S})$. When the samples are drawn with replacement, we have, by \eqref{Bernstein_error_C^{1}_H_{n}} of Lemma~\ref{lem:Bernstein_approximation_errors_C^{1}}, the bounds in \eqref{H_{n}_bounds}, and $0\leqslant g(\cdot,\cdot)\leqslant 1$ on $\mathcal{S}$, for all $(x,y)\in\mathcal{S}$,
\begin{multline}
|q_{1}H_{n}(x,y)\{1-q_{1}H_{n}(x,y)\}-q_{1}g(x,y)\{1-q_{1}g(x,y)\}|\leqslant \beta_{n}, \text{ where }\\ \beta_{n}=q_{1}(1+2q_{1})C\min\{\omega(\nabla g;\sqrt{\E[K_{n}^{-1}]})\sqrt{\E[K_{n}^{-1}]},\omega(\nabla g;\E[K_{n}^{-1}](\E[K_{n}^{-1/2}])^{-1})\E[K_{n}^{-1/2}]\},\label{(1,1)-th_entry_convergence_rate_with_replacement_varying_law_C^{1}}
\end{multline}
and $\beta_{n}\rightarrow 0$ since \eqref{thm:C_{n}_a.s.:g_C^{1}} is true, implying the uniform convergence of $q_{1}H_{n}(\cdot,\cdot)\{1-q_{1}H_{n}(\cdot,\cdot)\}$ to $q_{1}g(\cdot,\cdot)\{1-q_{1}g(\cdot,\cdot)\}$ on $\mathcal{S}$. As mentioned above, each $q_{1}H_{n}(\cdot,\cdot)\{1-q_{1}H_{n}(\cdot,\cdot)\}\in\mathcal{C}^{(0)}(\mathcal{S})$ since each $H_{n}$ is a polynomial. By Lemma~\ref{lem:a.s.convergence}, we conclude that \eqref{(1,1)-th_entry_limit_varying_law} holds. When the samples are drawn without replacement, we have, by \eqref{Bernstein_error_C^{1}_E_{n}} of Lemma~\ref{lem:Bernstein_approximation_errors_C^{1}}, the bounds in \eqref{E_{n}_bounds}, and $0\leqslant g(\cdot,\cdot)\leqslant 1$ on $\mathcal{S}$, for all $(r_{1}/n,r_{2}/n)\in\mathcal{S}_{n}$,
\begin{equation}
|q_{1}E_{n}(r_{1}/n,r_{2}/n)\{1-q_{1}E_{n}(r_{1}/n,r_{2}/n)\}-q_{1}g(r_{1}/n,r_{2}/n)\{1-q_{1}g(r_{1}/n,r_{2}/n)\}|\leqslant \beta_{n},\nonumber
\end{equation}
with $\beta_{n}$ as defined in \eqref{(1,1)-th_entry_convergence_rate_with_replacement_varying_law_C^{1}}. Since \eqref{thm:C_{n}_a.s.:g_C^{1}} is true, we have 
\begin{enumerate*}
\item $\beta_{n}\rightarrow0$, and 
\item $g\in\mathcal{C}^{(1)}(\mathcal{S})$, which implies that $q_{1}g(\cdot,\cdot)\{1-q_{1}g(\cdot,\cdot)\}\in\mathcal{C}^{(0)}(\mathcal{S})$. 
\end{enumerate*}
Lemma~\ref{lem:a.s.convergence}, once again, yields the conclusion given by \eqref{(1,1)-th_entry_limit_varying_law}.

Arguing similarly for the remaining rows of \eqref{MDS_squared_entries}, we conclude that \eqref{MDS_squared_convergence_to_Gamma} is satisfied in the regimes given by \eqref{thm:main_6:regime_1} and \eqref{thm:main_6:regime_3}, but with the matrix $\Gamma$ as defined in the statement of Theorem~\ref{thm:main_6}. This verifies \eqref{MDS_quadratic_variation_Cesaro_conv}.

\textbf{The proof of \eqref{thm:main_6:regime_1}:} We let $\delta_{n}=(\delta_{n,1},\delta_{n,2},\delta_{n,3})^{t}$ for $n\geqslant N$, where $\delta_{n,j}$, for $j\in\{1,2,3\}$, are as defined in \eqref{many_defns_with_replacement_scenario_2} and \eqref{third_error_term_with_replacement_varying_law} when the samples are drawn with replacement, and as defined in \eqref{many_defns_without_replacement_varying_law} and \eqref{third_error_term_without_replacement_varying_law} when the samples are drawn without replacement. Since \eqref{thm:C_{n}_a.s.:g_C^{2}} has been assumed, $g\in\mathcal{C}^{(2)}(\mathcal{S})$ in this regime. 

Suppose the first criterion stated in \eqref{thm:main_6:regime_1:series_criteria} is true. We make use of either \eqref{Bernstein_error_C^{2}_H_{n}} or \eqref{Bernstein_error_C^{2}_E_{n}} of Lemma~\ref{lem:Bernstein_approximation_errors_C^{2}}, according to whether the samples are drawn with replacement or without, to obtain:
\begin{equation}
\left|\left|\sum_{i=N}^{n}\delta_{i}\right|\right|_{2}\leqslant \sqrt{q_{1}^{2}+q_{2}^{2}+(1-q_{1})^{2}}C\sum_{i=N}^{n}\E\left[K_{i}^{-1}\right]=o\left(\sqrt{\frac{n}{\log n}}\right),\label{delta_{i}_partial_sum_C^{2}}
\end{equation} 
which shows that \eqref{series_conv_zhang_2.1_1} is true. Likewise, when the second criterion stated in \eqref{thm:main_6:regime_1:series_criteria} is true, we obtain:
\begin{equation}
\sum_{i=N}^{n}(i+1)^{-1/2}\left|\left|\delta_{i}\right|\right|_{2}\leqslant \sqrt{q_{1}^{2}+q_{2}^{2}+(1-q_{1})^{2}}C\sum_{i=N}^{n}(i+1)^{-1/2}\E\left[K_{i}^{-1}\right]=o\left(\sqrt{\log n}\right),\nonumber
\end{equation}
which shows that \eqref{series_conv_zhang_2.1_2} is true. Thus, Theorem~\ref{thm:zhang_2.1} is applicable, and the conclusion of \eqref{thm:main_6:regime_1} now follows exactly as the conclusion of \eqref{thm:main_5:regime_1}, in the proof of Theorem~\ref{thm:main_5}, did.

\textbf{The proof of \eqref{thm:main_6:regime_2}:} Once again, \eqref{thm:C_{n}_a.s.:g_C^{2}} has been assumed, so that $g\in\mathcal{C}^{(2)}(\mathcal{S})$ in this regime. Using either \eqref{Bernstein_error_C^{2}_H_{n}} or \eqref{Bernstein_error_C^{2}_E_{n}} of Lemma~\ref{lem:Bernstein_approximation_errors_C^{2}}, we deduce, similar to \eqref{delta_{i}_partial_sum_C^{2}}, that $||\sum_{i=N}^{n}\delta_{i}||_{2}=O(\sum_{i=N}^{n}\E[K_{i}^{-1}])$, which is $o(n^{1-\rho-\epsilon})$, for some $0<\epsilon<1-\rho$, due to \eqref{thm:main_6:regime_2:series_criteria}, thus validating \eqref{series_conv_zhang_2.2_2}. 

To verify that \eqref{series_conv_zhang_2.2_1} is true, we focus on \eqref{MDS_squared_entries}, with $f_{n}\equiv H_{n}$ when sampling with replacement, and $f_{n}\equiv E_{n}$ when sampling without replacement. Making use of either \eqref{H_{n}_bounds} or \eqref{E_{n}_bounds}, we see that the absolute value of each entry of $\E[\Delta M_{n+1}\Delta M_{n+1}^{t}|\mathcal{F}_{n}]$ is bounded above by $1$ (in fact, the bounds can be refined much further), so that \eqref{series_conv_zhang_2.2_1} is obviously satisfied. We may, now, apply Theorem~\ref{thm:zhang_2.2}, and the conclusion of \eqref{thm:main_6:regime_2} now follows exactly as the conclusion of \eqref{thm:main_5:regime_2}, in the proof of Theorem~\ref{thm:main_5}, did.

\textbf{The proof of \eqref{thm:main_6:regime_3}:} When \eqref{thm:C_{n}_a.s.:g_C^{1}} is assumed to hold, $g\in\mathcal{C}^{(1)}(\mathcal{S})$, so that by \eqref{Bernstein_error_C^{1}_H_{n}} of Lemma~\ref{lem:Bernstein_approximation_errors_C^{1}} when the samples are drawn with replacement, or by \eqref{Bernstein_error_C^{1}_E_{n}} of Lemma~\ref{lem:Bernstein_approximation_errors_C^{1}} when the samples are drawn without replacement, we conclude that 
\begin{equation}
\left|\left|\sum_{i=N}^{n}\delta_{i}\right|\right|_{2}\leqslant \sqrt{q_{1}^{2}+q_{2}^{2}+(1-q_{1})^{2}}C\sum_{i=N}^{n}\min\left\{\omega\left(\nabla g;\sqrt{\E\left[K_{i}^{-1}\right]}\right)\sqrt{\E\left[K_{i}^{-1}\right]},\omega\left(\nabla g;\frac{\E\left[K_{i}^{-1}\right]}{\E\left[K_{i}^{-1/2}\right]}\right)\E\left[K_{i}^{-1/2}\right]\right\},\nonumber
\end{equation}
which is $o(\sqrt{n})$ because of \eqref{thm:main_6:regime_3.1:series_criteria}. When \eqref{thm:C_{n}_a.s.:g_C^{2}} is assumed to hold, $g\in\mathcal{C}^{(2)}(\mathcal{S})$, so that we deduce, similar to \eqref{delta_{i}_partial_sum_C^{2}}, that $||\sum_{i=N}^{n}\delta_{i}||_{2}=O(\sum_{i=N}^{n}\E[K_{i}^{-1}])=o(\sqrt{n})$ because of \eqref{thm:main_6:regime_3.2:series_criteria}. Thus, either way, \eqref{series_conv_zhang_2.3} holds, and we may now apply Theorem~\ref{thm:zhang_2.3}. The conclusions of \eqref{thm:main_6:regime_3} now follow exactly as the conclusions of \eqref{thm:main_5:regime_3}, in the proof of Theorem~\ref{thm:main_5}, did.
\end{proof}

\begin{proof}[Proof of Proposition~\ref{prop:K_{n}_distributions}]
Throughout the proof of Proposition~\ref{prop:K_{n}_distributions}, we assume $g\in\mathcal{C}^{(2)}(\mathcal{S})$ and that $g$ satisfies \eqref{contraction_criterion_g}. We begin by arguing that it suffices to show that the first criterion of \eqref{thm:main_6:regime_1:series_criteria} is satisfied by $\{K_{n}:n\geqslant N\}$ considered in Proposition~\ref{prop:K_{n}_distributions}. In the regime described in \eqref{thm:main_6:regime_2}, we have $\rho<1/2$, so that choosing $0<\epsilon<(1-\rho)-1/2$ and assuming \eqref{thm:main_6:regime_1:series_criteria} is true, yield $\sum_{i=N}^{n}\E[K_{i}^{-1}]=o(\sqrt{n/\log n})=o(\sqrt{n})=o(n^{1-\rho-\epsilon})$, establishing \eqref{thm:main_6:regime_2:series_criteria}. Since $\sqrt{n/\log n}=o(\sqrt{n})$, it is even more straightforward to see that \eqref{thm:main_6:regime_1:series_criteria} implies \eqref{thm:main_6:regime_3.2:series_criteria}, in the regime of \eqref{thm:main_6:regime_3}. Finally, the first criterion of \eqref{thm:main_6:regime_1:series_criteria} implies \eqref{thm:main_2:C^{2}_series_convergence} by Lemma~\ref{lem:stronger_convergence_criteria_zhang}, thus ensuring \eqref{thm:main_2:C^{2}} is true. Thus, making sure that the first criterion of \eqref{thm:main_6:regime_1:series_criteria}, as well as the convergence criterion stated in \eqref{thm:C_{n}_a.s.:g_C^{2}}, is satisfied, is enough to ensure that the conclusions drawn in each of Theorem~\ref{thm:main_2_special}, Theorem~\ref{thm:a.s._convergence_C_{n}} and Theorem~\ref{thm:main_6} remain true. 

When $K_{n}$ is discrete uniform with support $[n]$, we have $\E[K_{n}^{-1}]=1/n\sum_{i\in[n]}1/i=O(\log n/n)\rightarrow 0$ as $n\rightarrow\infty$, establishing the convergence criterion of \eqref{thm:C_{n}_a.s.:g_C^{2}}. Moreover, this estimate yields $\sum_{i=N}^{n}\E[K_{i}^{-1}]=\sum_{i=N}^{n}O(i^{-1}\log i)=O((\log n)^{2})=o(\sqrt{n/\log n})$ as $n\rightarrow\infty$, proving that the first criterion of \eqref{thm:main_6:regime_1:series_criteria} holds. 

Suppose $K_{n}=\min\{n,G_{n}\}$, where $G_{n}$ follows Geometric$(p_{n})$ with probability of success $p_{n}=c n^{-\alpha}$ for some $\alpha>1/2$ (in other words, $G_{n}$ equals the total number of trials required to achieve the first success, wherein each trial, independent of all else, results in success with probability $p_{n}$ and in failure with probability $(1-p_{n})$. This yields, using the inequality $1-x\leqslant e^{-x}$ for all $x\geqslant 0$: 
\begin{align}
{}&\E[K_{n}^{-1}]=\frac{p_{n}}{1-p_{n}}\sum_{k=1}^{n-1}\frac{(1-p_{n})^{k}}{k}+\frac{1}{n}\frac{(1-p_{n})^{n}}{1-p_{n}}\leqslant\frac{c}{n^{\alpha}-c}\left[e^{-c n^{-\alpha}}+\sum_{k=2}^{n-1}\int_{k-1}^{k}\frac{e^{-c k n^{-\alpha}}}{k}dx\right]+\frac{e^{-c n^{1-\alpha}}}{n-c n^{1-\alpha}}\nonumber\\
\leqslant{}&\frac{c}{n^{\alpha}-c}\left[1+\int_{1}^{n-1}\frac{e^{-c x n^{-\alpha}}}{x}dx\right]+\frac{1}{n-c n^{1-\alpha}}\leqslant\frac{c}{n^{\alpha}-c}\left[1+\int_{1}^{n-1}\frac{1}{x}dx\right]+\frac{1}{n-c n^{1-\alpha}}=\Theta\left(\frac{\log n}{n^{\alpha}}\right),\nonumber
\end{align}
so that $\E[K_{n}^{-1}]\rightarrow0$ as $n\rightarrow\infty$, ensuring that \eqref{thm:C_{n}_a.s.:g_C^{2}} holds. For $N$ sufficiently large (say, $N>1+e^{1/\alpha}$), 
\begin{align}
\sum_{i=N}^{n}\E[K_{i}^{-1}]={}&\Theta\left(\sum_{i=N}^{n}\frac{\log i}{i^{\alpha}}\right)=\Theta\left(\sum_{i=N}^{n}\int_{i-1}^{i}\frac{\log i}{i^{\alpha}}dx\right)=O\left(\int_{N-1}^{n}\frac{\log x}{x^{\alpha}}dx\right),\nonumber
\end{align}
which is $O(n^{1-\alpha}\log n)$ for $\alpha\in (1/2,\infty)\setminus\{1\}$, and it is $O((\log n)^{2})$ for $\alpha=1$, and in either case, we have $\sum_{i=N}^{n}\E[K_{i}^{-1}]=o(\sqrt{n/\log n})$, verifying the first criterion of \eqref{thm:main_6:regime_1:series_criteria}.

When $K_{n}=1+T_{n-1}$, where $T_{n-1}$ follows Binomial$(n-1,p_{n})$, with $p_{n}=c n^{-\alpha}$ for some $0<\alpha<1/2$,
\begin{align}
\E[K_{n}^{-1}]={}&\sum_{k=0}^{n-1}\frac{1}{k+1}{n-1\choose k}p_{n}^{k}(1-p_{n})^{n-1-k}=\frac{1}{np_{n}}\left[1-(1-p_{n})^{n}\right]=\Theta\left(n^{-1+\alpha}\right),\nonumber
\end{align}
so that the convergence criterion of \eqref{thm:C_{n}_a.s.:g_C^{2}} is satisfied, and moreover, $\sum_{i=N}^{n}\E[K_{i}^{-1}]=\Theta(\sum_{i=N}^{n}i^{-1+\alpha})=\Theta(n^{\alpha})$, which is $o(\sqrt{n/\log n})$ since $\alpha<1/2$, verifying the first criterion of \eqref{thm:main_6:regime_1:series_criteria}.

Finally, when $K_{n}=\min\{n,1+P_{n}\}$ where $P_{n}$ follows Poisson$(\lambda_{n})$ with $\lambda_{n}=c n^{\alpha}$ for some $1/2<\alpha<1$, 
\begin{align}
\E[K_{n}^{-1}]={}&\sum_{k=1}^{n-1}\frac{1}{k}e^{-\lambda_{n}}\frac{\lambda_{n}^{k-1}}{(k-1)!}+\frac{1}{n}\sum_{k\geqslant n-1}e^{-\lambda_{n}}\frac{\lambda_{n}^{k}}{k!}=\frac{e^{-\lambda_{n}}}{\lambda_{n}}\sum_{k=1}^{n-1}\frac{\lambda_{n}^{k}}{k!}+\frac{e^{-\lambda_{n}}}{n}\sum_{k\geqslant n-1}\frac{\lambda_{n}^{k}}{k!}=\Theta\left(\frac{1}{\lambda_{n}}\right),\label{Poisson_intermediate}
\end{align}
since the second sum in the second last expression of \eqref{Poisson_intermediate} can be bounded above as follows:
\begin{align}
\frac{e^{-\lambda_{n}}\lambda_{n}^{n-1}}{n!}\sum_{k\geqslant n-1}\frac{\lambda_{n}^{k-n+1}}{n(n+1)\cdots k}\leqslant \frac{\lambda_{n}^{n-1}}{n!}=\Theta\left(\frac{c^{n-1}e^{n}}{n^{n(1-\alpha)+\alpha+1/2}}\right)=O\left(\left(\frac{c e}{n^{1-\alpha}}\right)^{n}\right)=o(1),\nonumber
\end{align}
using Stirling's approximation. From \eqref{Poisson_intermediate}, we have $\E[K_{n}^{-1}]=\Theta(n^{-\alpha})\rightarrow 0$ as $n\rightarrow\infty$, so that the convergence criterion of \eqref{thm:C_{n}_a.s.:g_C^{2}} is satisfied, and $\sum_{i=N}^{n}\E[K_{i}^{-1}]=\Theta(\sum_{i=N}^{n}i^{-\alpha})=\Theta(n^{1-\alpha})=o(\sqrt{n/\log n})$ since $\alpha>1/2$, verifying the first criterion of \eqref{thm:main_6:regime_1:series_criteria}. This completes the proof of Proposition~\ref{prop:K_{n}_distributions}.
\end{proof}

\section{Acknowledgements}
A.\ Roy acknowledges support from the IIM-K SGRP Research Grant (No. SGRP/2025-26/22) for the accomplishment of this project.

\section{Appendix}\label{sec:appendix}
We state and prove here some lemmas whose proofs involve long and intricate computations. These lemmas find applications in the proofs of many of our main results.
\begin{lemma}\label{lem:Bernstein_approximation_errors_Holder_continuous}
When the function $g$, defined in \eqref{g_defn}, is \emph{H\"{o}lder continuous} with constant $L$ and exponent $\alpha$ for some $0<\alpha\leqslant 1$, i.e.\ 
\begin{equation}
\left|g(x_{1},y_{1})-g(x_{2},y_{2})\right|\leqslant L\left|\left|(x_{1},y_{1})-(x_{2},y_{2})\right|\right|_{2}^{\alpha} \quad \text{for all } (x_{1},y_{1}), (x_{2},y_{2})\in\mathcal{S},\label{Holder_cont_cond}
\end{equation} 
we have, for $H_{n}$ as defined in \eqref{H_{n}_defn}, $E_{n}$ as defined in \eqref{E_{n}_defn} and $\mathcal{S}_{n}$ as defined in \eqref{S_{n}_defn}, for each $n\geqslant N$:
\begin{align}
{}&\sup\left\{\left|H_{n}(x,y)-g(x,y)\right|:(x,y)\in\mathcal{S}\right\}\leqslant 2^{-\alpha/2}L\E\left[K_{n}^{-\alpha/2}\right]\label{Bernstein_error_Lipschitz_H_{n}}\\
{}&\sup\left\{\left|E_{n}\left(\frac{r_{1}}{n},\frac{r_{2}}{n}\right)-g\left(\frac{r_{1}}{n},\frac{r_{2}}{n}\right)\right|:\left(\frac{r_{1}}{n},\frac{r_{2}}{n}\right)\in \mathcal{S}_{n}\right\}\leqslant 2^{-\alpha/2}L\E\left[K_{n}^{-\alpha/2}\right]\label{Bernstein_error_Lipschitz_E_{n}}
\end{align}
\end{lemma}
\begin{proof}
To prove \eqref{Bernstein_error_Lipschitz_H_{n}}, we fix $(x,y)\in\mathcal{S}$, and recall, as mentioned right after \eqref{H_{n}_defn}, that $H_{n}(x,y)=\E[g(V_{1}/K_{n},V_{2}/K_{n})]$, where $V_{1}$ and $V_{2}$ are as defined in the experiment outlined in \eqref{Exp_1}. Therefore,
\begin{align}
{}&\left|H_{n}(x,y)-g(x,y)\right|\leqslant\sum_{k\in I_{n}}\mu_{n}(k)\E\left[\left|g\left(\frac{V_{1}}{K_{n}},\frac{V_{2}}{K_{n}}\right)-g(x,y)\right|\Big|K_{n}=k\right]\nonumber\\
\leqslant{}&L\sum_{k\in I_{n}}\mu_{n}(k)\E\left[\left|\left|\left(\frac{V_{1}}{K_{n}},\frac{V_{2}}{K_{n}}\right)-(x,y)\right|\right|_{2}^{\alpha}\Big|K_{n}=k\right], \quad \text{using \eqref{Holder_cont_cond};}\nonumber\\
\leqslant{}&L\sum_{k\in I_{n}}\mu_{n}(k)\left\{\E\left[\left(\frac{V_{1}}{K_{n}}-x\right)^{2}+\left(\frac{V_{2}}{K_{n}}-y\right)^{2}\Big|K_{n}=k\right]\right\}^{\alpha/2}, \quad \text{using H\"{o}lder's inequality;}\nonumber\\
={}&L\sum_{k\in I_{n}}\mu_{n}(k)\left\{\frac{x(1-x)}{k}+\frac{y(1-y)}{k}\right\}^{\alpha/2}\leqslant 2^{-\alpha/2}L\E\left[K_{n}^{-\alpha/2}\right],\nonumber
\end{align}
where the second last step follows from \eqref{exp:1_variances}, and the last inequality follows by noting that each of $x(1-x)$ and $y(1-y)$ is bounded above by $1/4$. This completes the proof of \eqref{Bernstein_error_Lipschitz_H_{n}}.

Next, to prove \eqref{Bernstein_error_Lipschitz_E_{n}}, we recall, as mentioned right after \eqref{E_{n}_defn}, that $E_{n}(r_{1}/n,r_{2}/n)=\E[g(V_{1}/K_{n},V_{2}/K_{n})]$ where $V_{1}$ and $V_{2}$ are as defined in Experiment~\eqref{Exp_2}. Thus, using \eqref{Holder_cont_cond} and H\"{o}lder's inequality in the first two steps, we have
\begin{align}
{}&\left|E_{n}\left(\frac{r_{1}}{n},\frac{r_{2}}{n}\right)-g\left(\frac{r_{1}}{n},\frac{r_{2}}{n}\right)\right|
\leqslant L\sum_{k\in I_{n}}\mu_{n}(k)\left\{\E\left[\left(\frac{V_{1}}{K_{n}}-\frac{r_{1}}{n}\right)^{2}+\left(\frac{V_{2}}{K_{n}}-\frac{r_{2}}{n}\right)^{2}\Big|K_{n}=k\right]\right\}^{\alpha/2}\nonumber\\
={}&L\sum_{k\in I_{n}}\mu_{n}(k)\left[\frac{(n-k)}{k(n-1)n^{2}}\left\{r_{1}(n-r_{1})+r_{2}(n-r_{2})\right\}\right]^{\alpha/2}\leqslant 2^{-\alpha/2}L\E\left[K_{n}^{-\alpha/2}\right],\nonumber
\end{align}
where the second last step follows from \eqref{exp:2_expectations_variances}, and the last step follows from the inequalities that $r_{i}(n-r_{i})\leqslant n^{2}/4$ for $i\in\{1,2\}$, and $(n-k)\leqslant(n-1)$. This completes the proof of \eqref{Bernstein_error_Lipschitz_E_{n}}.
\end{proof}

\begin{lemma}\label{lem:Bernstein_approximation_errors_C^{1}}
Assume that $g$, defined in \eqref{g_defn}, is in $\mathcal{C}^{(1)}(\mathcal{S})$. Then, for $\omega(\nabla g;\cdot)$ as defined in \S\ref{subsec:notations_2}, $H_{n}$ as defined in \eqref{H_{n}_defn}, $E_{n}$ as defined in \eqref{E_{n}_defn}, $\mathcal{S}_{n}$ as defined in \eqref{S_{n}_defn} and $C=1/\sqrt{2}+1/2$, for each $n\geqslant N$: 
\begin{align}
\sup\left\{\left|H_{n}(x,y)-g(x,y)\right|:(x,y)\in\mathcal{S}\right\}\leqslant
\begin{cases}
& C\omega\left(\nabla g;\sqrt{\E\left[K_{n}^{-1}\right]}\right)\sqrt{\E\left[K_{n}^{-1}\right]},\\
& C\omega\left(\nabla g;\E\left[K_{n}^{-1}\right]\left(\E\left[K_{n}^{-1/2}\right]\right)^{-1}\right)\E\left[K_{n}^{-1/2}\right],
\end{cases}\label{Bernstein_error_C^{1}_H_{n}}\\
\sup\left\{\left|E_{n}\left(\frac{r_{1}}{n},\frac{r_{2}}{n}\right)-g\left(\frac{r_{1}}{n},\frac{r_{2}}{n}\right)\right|:\left(\frac{r_{1}}{n},\frac{r_{2}}{n}\right)\in \mathcal{S}_{n}\right\}\leqslant
\begin{cases}
& C\omega\left(\nabla g;\sqrt{\E\left[K_{n}^{-1}\right]}\right)\sqrt{\E\left[K_{n}^{-1}\right]},\\
& C\omega\left(\nabla g;\E\left[K_{n}^{-1}\right]\left(\E\left[K_{n}^{-1/2}\right]\right)^{-1}\right)\E\left[K_{n}^{-1/2}\right].
\end{cases}\label{Bernstein_error_C^{1}_E_{n}}
\end{align}
\end{lemma}
\begin{proof}
For any $(u_{1},v_{1}), (u_{2},v_{2})\in\mathcal{S}$, with $\mathcal{S}$ as defined in \eqref{domain_defn}, and $\delta>0$, let us define
\begin{equation}
\lambda\left((u_{1},v_{1}),(u_{2},v_{2});\delta\right)=\left\lfloor\sqrt{(u_{1}-u_{2})^{2}+(v_{1}-v_{2})^{2}}\delta^{-1}\right\rfloor,\label{lambda_defn}
\end{equation}
where the notation $\lfloor x \rfloor$ indicates the greatest integer less than or equal to $x$. Using the notation introduced in \eqref{nabla_defn} and an application of the mean value theorem, we obtain, for some $(\xi_{1},\xi_{2})$ lying on the line segment joining $(u_{1},v_{1})$ and $(u_{2},v_{2})$:
\begin{align}
{}&g\left(u_{2},v_{2}\right)-g\left(u_{1},v_{1}\right)=\nabla g(u_{1},v_{1})\begin{bmatrix}
u_{2}-u_{1}\\
v_{2}-v_{1}
\end{bmatrix}+\left(\nabla g(\xi_{1},\xi_{2})-\nabla g(u_{1},v_{1})\right)\begin{bmatrix}
u_{2}-u_{1}\\
v_{2}-v_{1}
\end{bmatrix}.\label{intermediate_1}
\end{align}  
By \eqref{lambda_defn}, and the fact that $(\xi_{1},\xi_{2})$ lies on the line segment \emph{between} $(u_{1},v_{1})$ and $(u_{2},v_{2})$, we have $\lambda\leqslant \lambda\left((u_{1},v_{1}),(u_{2},v_{2});\delta\right)$ where, for brevity of notation, we set $\lambda=\lambda\left((u_{1},v_{1}),(\xi_{1},\xi_{2});\delta\right)$. We choose points $(w_{1},z_{1}), \ldots, (w_{\lambda},z_{\lambda})$ on the line segment joining $(u_{1},v_{1})$ and $(\xi_{1},\xi_{2})$ such that, for each $i\in\{1,2,\ldots,\lambda\}$: 
\begin{equation}
w_{i}=u_{1}+\frac{i\delta}{\sqrt{1+\beta^{2}}} \quad \text{and} \quad z_{i}=v_{1}+\frac{i\beta\delta}{\sqrt{1+\beta^{2}}}, \quad \text{where} \quad \beta=\frac{\xi_{2}-v_{1}}{\xi_{1}-u_{1}}.\nonumber
\end{equation}
\sloppy Setting $(w_{0},z_{0})=(u_{1},v_{1})$ and $(w_{\lambda+1},z_{\lambda+1})=(\xi_{1},\xi_{2})$, we see that the Euclidean distance between $(w_{i},z_{i})$ and $(w_{i+1},z_{i+1})$ is at most $\delta$ for each $i\in\{0,1,\ldots,\lambda\}$, so that $||\nabla g(w_{i+1},z_{i+1})-\nabla g(w_{i},z_{i})||_{2}\leqslant \omega(\nabla g;\delta)$ for each $i\in\{0,1,\ldots,\lambda\}$, where $\omega(\nabla g;\delta)$ is as defined in \eqref{derivative_g_modulus_of_continuity}. This, along with the Cauchy-Schwarz inequality and the triangle inequality, yields:
\begin{align}
{}&\left|\left(\nabla g(\xi_{1},\xi_{2})-\nabla g(u_{1},v_{1})\right)\begin{bmatrix}
u_{2}-u_{1}\\
v_{2}-v_{1}
\end{bmatrix}\right|\leqslant\left|\left|\sum_{i=0}^{\lambda}\left(\nabla g(w_{i+1},z_{i+1})-\nabla g(w_{i},z_{i})\right)\right|\right|_{2}\left|\left|\begin{bmatrix}
u_{2}-u_{1}\\
v_{2}-v_{1}
\end{bmatrix}\right|\right|_{2}\nonumber\\
\leqslant{}&\left[\sum_{i=0}^{\lambda}\left|\left|\nabla g(w_{i+1},z_{i+1})-\nabla g(w_{i},z_{i})\right|\right|_{2}\right]\left|\left|\begin{bmatrix}
u_{2}-u_{1}\\
v_{2}-v_{1}
\end{bmatrix}\right|\right|_{2}\nonumber\\
\leqslant{}&(\lambda+1)\omega\left(\nabla g;\delta\right)\left|\left|\begin{bmatrix}
u_{2}\\
v_{2}
\end{bmatrix}-\begin{bmatrix}
u_{1}\\
v_{1}
\end{bmatrix}\right|\right|_{2}\leqslant \left\{\lambda\left((u_{1},v_{1}),(u_{2},v_{2});\delta\right)+1\right\}\omega\left(\nabla g;\delta\right)\left|\left|\begin{bmatrix}
u_{2}\\
v_{2}
\end{bmatrix}-\begin{bmatrix}
u_{1}\\
v_{1}
\end{bmatrix}\right|\right|_{2}.\label{intermediate_3}
\end{align}
Recalling $H_{n}$ from \eqref{H_{n}_defn}, applying \eqref{intermediate_1} and \eqref{intermediate_3}, and recalling $V_{1}$ and $V_{2}$ as defined in Experiment~\eqref{Exp_1}, we deduce, for any $(x,y)\in\mathcal{S}$:
\begin{align}
{}&\left|H_{n}(x,y)-g(x,y)\right|\leqslant\left|\sum_{k\in I_{n}}\mu_{n}(k)\sum_{i=0}^{k}\sum_{j=0}^{k-i}\nabla g(x,y)\begin{bmatrix}
i/k-x\\
j/k-y
\end{bmatrix}{k\choose i}{k-i\choose j}x^{i}y^{j}(1-x-y)^{k-i-j}\right|+\nonumber\\&\sum_{k\in I_{n}}\mu_{n}(k)\sum_{i=0}^{k}\sum_{j=0}^{k-i}\left\{\lambda\left((x,y),\left(\frac{i}{k},\frac{j}{k}\right);\delta\right)+1\right\}\omega\left(\nabla g;\delta\right)\left|\left|\begin{bmatrix}
i/k-x\\
j/k-y
\end{bmatrix}\right|\right|_{2}{k\choose i}{k-i\choose j}x^{i}y^{j}(1-x-y)^{k-i-j}\nonumber\\
\leqslant{}&\left|\sum_{k\in I_{n}}\mu_{n}(k)\left\{\frac{\partial g(x,y)}{\partial x}\E\left[\frac{V_{1}}{K_{n}}-x\Big|K_{n}=k\right]+\frac{\partial g(x,y)}{\partial y}\E\left[\frac{V_{2}}{K_{n}}-y\Big|K_{n}=k\right]\right\}\right|\nonumber\\&+\omega\left(\nabla g;\delta\right)\sum_{k\in I_{n}}\mu_{n}(k)\left(\sum_{i=0}^{k}\sum_{j=0}^{k-i}\left|\left|\begin{bmatrix}
i/k-x\\
j/k-y
\end{bmatrix}\right|\right|^{2}_{2}{k\choose i}{k-i\choose j}x^{i}y^{j}(1-x-y)^{k-i-j}\right)^{1/2}\nonumber\\&+\omega\left(\nabla g;\delta\right)\delta^{-1}\sum_{k\in I_{n}}\mu_{n}(k)\sum_{i=0}^{k}\sum_{j=0}^{k-i}\left|\left|\begin{bmatrix}
i/k-x\\
j/k-y
\end{bmatrix}\right|\right|_{2}^{2}{k\choose i}{k-i\choose j}x^{i}y^{j}(1-x-y)^{k-i-j}\label{intermediate_8_1}\\
={}&\omega\left(\nabla g;\delta\right)\sum_{k\in I_{n}}\mu_{n}(k)\left[\E\left\{\left(\frac{V_{1}}{K_{n}}-x\right)^{2}+\left(\frac{V_{2}}{K_{n}}-y\right)^{2}\Bigg|K_{n}=k\right\}\right]^{1/2}\nonumber\\&+\omega\left(\nabla g;\delta\right)\delta^{-1}\sum_{k\in I_{n}}\mu_{n}(k)\E\left\{\left(\frac{V_{1}}{K_{n}}-x\right)^{2}+\left(\frac{V_{2}}{K_{n}}-y\right)^{2}\Bigg|K_{n}=k\right\}\label{intermediate_8_2}\\
={}&\omega\left(\nabla g;\delta\right)\sum_{k\in I_{n}}\mu_{n}(k)\left\{\left(\frac{x(1-x)}{k}+\frac{y(1-y)}{k}\right)^{1/2}+\frac{1}{\delta}\left(\frac{x(1-x)}{k}+\frac{y(1-y)}{k}\right)\right\}\nonumber\\
\leqslant{}&\omega\left(\nabla g;\delta\right)\left(2^{-1/2}\E\left[K_{n}^{-1/2}\right]+(2\delta)^{-1}\E\left[K_{n}^{-1}\right]\right).\label{intermediate_8}
\end{align}
Here, in order to deduce \eqref{intermediate_8_1}, we have used the joint distribution of $V_{1}$ and $V_{2}$, deduced from the description of Experiment~\eqref{Exp_1}, applied the Cauchy-Schwarz inequality (this yields the term in the middle), and the inequality $\lambda((x,y),(i/k,j/k);\delta)\leqslant \delta^{-1}||(i/k,j/k)-(x,y)||_{2}$, which follows from \eqref{lambda_defn}. The equality in \eqref{intermediate_8_2} is deduced from \eqref{exp:1_expectations} and the conditional variance of each of $V_{1}$ and $V_{2}$ that follows from Experiment~\eqref{Exp_1}. The penultimate step follows from \eqref{exp:1_variances}, and the final step is obtained by noting that each of $x(1-x)$ and $y(1-y)$ is bounded above by $1/4$. At this point, setting $\delta=(\E[K_{n}^{-1}])^{1/2}$ in \eqref{intermediate_8}, and applying Jensen's inequality, yields the first inequality stated in \eqref{Bernstein_error_C^{1}_H_{n}}, whereas setting $\delta=\E[K_{n}^{-1}](\E[K_{n}^{-1/2}])^{-1}$ in \eqref{intermediate_8} yields the second inequality stated in \eqref{Bernstein_error_C^{1}_H_{n}}.

Next, recalling the functions $E_{n}$ from \eqref{E_{n}_defn}, applying \eqref{intermediate_1} and \eqref{intermediate_3}, and recalling $V_{1}$ and $V_{2}$ from the experiment outlined in \eqref{Exp_2}, we deduce, for any $(r_{1}/n,r_{2}/n)\in \mathcal{S}_{n}$:
\begin{align}
{}&\left|E_{n}\left(\frac{r_{1}}{n},\frac{r_{2}}{n}\right)-g\left(\frac{r_{1}}{n},\frac{r_{2}}{n}\right)\right|\leqslant\Bigg|\sum_{k\in I_{n}}\mu_{n}(k)\sum_{i=0}^{k}\sum_{j=0}^{k-i}\nabla g(r_{1}/n,r_{2}/n)\begin{bmatrix}
i/k-r_{1}/n\\
j/k-r_{2}/n
\end{bmatrix}{r_{1}\choose i}{r_{2}\choose j}{n-r_{1}-r_{2}\choose k-i-j}\nonumber\\&\left\{{n\choose k}\right\}^{-1}\Bigg|+\sum_{k\in I_{n}}\mu_{n}(k)\sum_{i=0}^{k}\sum_{j=0}^{k-i}\left\{\lambda\left(\left(\frac{r_{1}}{n},\frac{r_{2}}{n}\right),\left(\frac{i}{k},\frac{j}{k}\right);\delta\right)+1\right\}\omega\left(\nabla g;\delta\right)\left|\left|\begin{bmatrix}
i/k-r_{1}/n\\
j/k-r_{2}/n
\end{bmatrix}\right|\right|_{2}{r_{1}\choose i}{r_{2}\choose j}\nonumber\\&{n-r_{1}-r_{2}\choose k-i-j}\left\{{n\choose k}\right\}^{-1}\nonumber\\
\leqslant{}&\left|\sum_{k\in I_{n}}\mu_{n}(k)\left\{\frac{\partial g}{\partial x}\Big|_{(r_{1}/n,r_{2}/n)}\E\left[\frac{V_{1}}{K_{n}}-\frac{r_{1}}{n}\Big|K_{n}=k\right]+\frac{\partial g}{\partial y}\Big|_{(r_{1}/n,r_{2}/n)}\E\left[\frac{V_{2}}{K_{n}}-\frac{r_{2}}{n}\Big|K_{n}=k\right]\right\}\right|\nonumber\\&+\omega\left(\nabla g;\delta\right)\sum_{k\in I_{n}}\mu_{n}(k)\left\{\E\left[\left|\left|\left(\frac{V_{1}}{K_{n}},\frac{V_{2}}{K_{n}}\right)-\left(\frac{r_{1}}{n},\frac{r_{2}}{n}\right)\right|\right|_{2}^{2}\Big|K_{n}=k\right]\right\}^{1/2}\nonumber\\&+\omega\left(\nabla g;\delta\right)\delta^{-1}\sum_{k\in I_{n}}\mu_{n}(k)\E\left[\left|\left|\left(\frac{V_{1}}{K_{n}},\frac{V_{2}}{K_{n}}\right)-\left(\frac{r_{1}}{n},\frac{r_{2}}{n}\right)\right|\right|_{2}^{2}\Big|K_{n}=k\right], \label{intermediate_11_1}\\
={}&\omega\left(\nabla g;\delta\right)\sum_{k\in I_{n}}\mu_{n}(k)\left\{\left[\frac{(n-k)}{k(n-1)n^{2}}\left\{r_{1}(n-r_{1})+r_{2}(n-r_{2})\right\}\right]^{1/2}+\frac{\delta^{-1}(n-k)}{k(n-1)n^{2}}\left\{r_{1}(n-r_{1})+r_{2}(n-r_{2})\right\}\right\}\nonumber\\
\leqslant{}&\omega\left(\nabla g;\delta\right)\left\{2^{-1/2}\E\left[K_{n}^{-1/2}\right]+(2\delta)^{-1}\E\left[K_{n}^{-1}\right]\right\}.\label{intermediate_11}
\end{align}
The inequality in \eqref{intermediate_11_1} follows from the joint distribution of $V_{1}$ and $V_{2}$ as deduced from the experiment described in \eqref{Exp_2}, by an application of the Cauchy-Schwarz inequality (which yields the term in the middle), and the inequality $\lambda((r_{1}/n,r_{2}/n),(i/k,j/k);\delta)\leqslant \delta^{-1}||(r_{1}/n,r_{2}/n)-(i/k,j/k)||_{2}$ that follows from \eqref{lambda_defn}. The next step follows from \eqref{exp:2_expectations_variances}, and the last step follows from noting that 
\begin{enumerate*}
\item $r_{i}(n-r_{i})\leqslant n^{2}/4$ for each $i\in\{1,2\}$ and 
\item $(n-k)\leqslant (n-1)$.
\end{enumerate*}
Setting $\delta=(\E[K_{n}^{-1}])^{1/2}$ in \eqref{intermediate_11}, and applying Jensen's inequality, yields the first inequality stated in \eqref{Bernstein_error_C^{1}_E_{n}}, whereas setting $\delta=\E[K_{n}^{-1}](\E[K_{n}^{-1/2}])^{-1}$ in \eqref{intermediate_11} yields the second inequality stated in \eqref{Bernstein_error_C^{1}_E_{n}}. 
\end{proof}

\begin{lemma}\label{lem:Bernstein_approximation_errors_C^{2}}
When the function $g$, defined in \eqref{g_defn}, is in $\mathcal{C}^{(2)}(\mathcal{S})$, then, for $H_{n}$ as defined in \eqref{H_{n}_defn}, $E_{n}$ as defined in \eqref{E_{n}_defn} and $\mathcal{S}_{n}$ as defined in \eqref{S_{n}_defn}, there exists some constant $C>0$ such that, for each $n\geqslant N$:
\begin{align}
{}&\sup\left\{\left|H_{n}(x,y)-g(x,y)\right|:(x,y)\in\mathcal{S}\right\}\leqslant C\E\left[K_{n}^{-1}\right]\label{Bernstein_error_C^{2}_H_{n}}\\
{}&\sup\left\{\left|E_{n}\left(\frac{r_{1}}{n},\frac{r_{2}}{n}\right)-g\left(\frac{r_{1}}{n},\frac{r_{2}}{n}\right)\right|:\left(\frac{r_{1}}{n},\frac{r_{2}}{n}\right)\in\mathcal{S}_{n}\right\}\leqslant C\E\left[K_{n}^{-1}\right].\label{Bernstein_error_C^{2}_E_{n}}
\end{align}
\end{lemma}
\begin{proof}
We begin by noting that, since $g\in\mathcal{C}^{(2)}(\mathcal{S})$ and $\mathcal{S}$ is a compact set, each of 
\begin{equation}
M_{1,1}=\sup\left\{\left|\frac{\partial^{2}g}{\partial x^{2}}\right|:(x,y)\in\mathcal{S}\right\},\ M_{2,2}=\sup\left\{\left|\frac{\partial^{2}g}{\partial y^{2}}\right|:(x,y)\in\mathcal{S}\right\},\ M_{1,2}=\sup\left\{\left|\frac{\partial^{2}g}{\partial x \partial y}\right|:(x,y)\in\mathcal{S}\right\}\nonumber
\end{equation}
is finite. Fix any $(x,y)\in\mathcal{S}$. In what follows, we let $(\xi_{i,j},\zeta_{i,j})$ indicate \emph{some} point lying on the line segment joining $(x,y)$ and $(i/k,j/k)$), so that, by performing a Taylor expansion of $g$ around $(x,y)$, we obtain:
\begin{align}
{}&\left|H_{n}(x,y)-g(x,y)\right|=\Bigg|\sum_{k\in I_{n}}\mu_{n}(k)\sum_{i=0}^{k}\sum_{j=0}^{k-i}\Bigg\{\frac{\partial g}{\partial x}\left(\frac{i}{k}-x\right)+\frac{\partial g}{\partial y}\left(\frac{j}{k}-y\right)+\frac{1}{2}\frac{\partial^{2}g}{\partial x^{2}}\Big|_{(\xi_{i,j},\zeta_{i,j})}\left(\frac{i}{k}-x\right)^{2}+\nonumber\\&\frac{1}{2}\frac{\partial^{2}g}{\partial y^{2}}\Big|_{(\xi_{i,j},\zeta_{i,j})}\left(\frac{j}{k}-y\right)^{2}+\frac{\partial^{2}g}{\partial x \partial y}\Big|_{(\xi_{i,j},\zeta_{i,j})}\left(\frac{i}{k}-x\right)\left(\frac{j}{k}-y\right)\Bigg\}{k\choose i}{k-i\choose j}x^{i}y^{j}(1-x-y)^{k-i-j}\Bigg|\nonumber\\
\leqslant{}&\left|\sum_{k\in I_{n}}\mu_{n}(k)\left\{\frac{\partial g}{\partial x}\E\left[\left(\frac{V_{1}}{K_{n}}-x\right)\Big|K_{n}=k\right]+\frac{\partial g}{\partial y}\E\left[\left(\frac{V_{2}}{K_{n}}-y\right)\Big|K_{n}=k\right]\right\}\right|\nonumber\\&+\frac{M_{1,1}}{2}\sum_{k\in I_{n}}\mu_{n}(k)\E\left[\left(\frac{V_{1}}{K_{n}}-x\right)^{2}\Big|K_{n}=k\right]+\frac{M_{2,2}}{2}\sum_{k\in I_{n}}\mu_{n}(k)\E\left[\left(\frac{V_{2}}{K_{n}}-y\right)^{2}\Big|K_{n}=k\right]\nonumber\\&+M_{1,2}\sum_{k\in I_{n}}\mu_{n}(k)\left(\E\left[\left(\frac{V_{1}}{K_{n}}-x\right)^{2}\Big|K_{n}=k\right]\right)^{1/2}\left(\E\left[\left(\frac{V_{2}}{K_{n}}-y\right)^{2}\Big|K_{n}=k\right]\right)^{1/2}\label{intermediate_11_3}\\
\leqslant{}&M\sum_{k\in I_{n}}\mu_{n}(k)\left\{\frac{x(1-x)}{k}+\frac{y(1-y)}{k}+\frac{\sqrt{x(1-x)y(1-y)}}{k}\right\}\leqslant \frac{3M}{4}\E\left[K_{n}^{-1}\right],\label{intermediate_11_2}
\end{align}
where $V_{1}$ and $V_{2}$ are as defined in Experiment~\eqref{Exp_1}, and we let $M=\max\{M_{1,1},M_{1,2},M_{2,2}\}$. The step shown in \eqref{intermediate_11_3} follows from the joint distribution of $V_{1}$ and $V_{2}$ deduced from Experiment~\eqref{Exp_1}, and an application of the Cauchy-Schwarz inequality. The penultimate step is obtained from \eqref{exp:1_expectations} and \eqref{exp:1_variances}, and the final expression in \eqref{intermediate_11_2} follows since each of $x(1-x)$ and $y(1-y)$ is bounded above by $1/4$. 

We now come to the proof of \eqref{Bernstein_error_C^{2}_E_{n}}. Fix $(r_{1}/n,r_{2}/n)\in \mathcal{S}_{n}$. For \emph{some} point $(\xi_{i,j},\zeta_{i,j})$ lying on the line segment joining $(r_{1}/n,r_{2}/n)$ and $(i/k,j/k)$ for each $i,j \in\mathbb{N}_{0}$ with $i+j\leqslant k$, we have, by Taylor expansion of the function $g$ around $(r_{1}/n,r_{2}/n)$: 
\begin{align}
{}&\left|E_{n}\left(\frac{r_{1}}{n},\frac{r_{2}}{n}\right)-g\left(\frac{r_{1}}{n},\frac{r_{2}}{n}\right)\right|=\Bigg|\sum_{k\in I_{n}}\mu_{n}(k)\sum_{i=0}^{k}\sum_{j=0}^{k-i}\Bigg\{\frac{\partial g}{\partial x}\Big|_{(r_{1}/n,r_{2}/n)}\left(\frac{i}{k}-\frac{r_{1}}{n}\right)+\frac{\partial g}{\partial y}\Big|_{(r_{1}/n,r_{2}/n)}\left(\frac{j}{k}-\frac{r_{2}}{n}\right)\nonumber\\&+\frac{1}{2}\frac{\partial^{2}g}{\partial x^{2}}\Big|_{(\xi_{i,j},\zeta_{i,j})}\left(\frac{i}{k}-\frac{r_{1}}{n}\right)^{2}+\frac{1}{2}\frac{\partial^{2}g}{\partial y^{2}}\Big|_{(\xi_{i,j},\zeta_{i,j})}\left(\frac{j}{k}-\frac{r_{2}}{n}\right)^{2}+\frac{\partial^{2}g}{\partial x \partial y}\Big|_{(\xi_{i,j},\zeta_{i,j})}\left(\frac{i}{k}-\frac{r_{1}}{n}\right)\left(\frac{j}{k}-\frac{r_{2}}{n}\right)\Bigg\}\nonumber\\&{r_{1}\choose i}{r_{2}\choose j}{n-r_{1}-r_{2}\choose k-i-j}\left\{{n\choose k}\right\}^{-1}\Bigg|\nonumber\\
\leqslant{}&\left|\sum_{k\in I_{n}}\mu_{n}(k)\left\{\frac{\partial g}{\partial x}\Big|_{(r_{1}/n,r_{2}/n)}\E\left[\frac{V_{1}}{K_{n}}-\frac{r_{1}}{n}\Big|K_{n}=k\right]+\frac{\partial g}{\partial y}\Big|_{(r_{1}/n,r_{2}/n)}\E\left[\frac{V_{2}}{K_{n}}-\frac{r_{2}}{n}\Big|K_{n}=k\right]\right\}\right|\nonumber\\&+\frac{M_{1,1}}{2}\sum_{k\in I_{n}}\mu_{n}(k)\E\left[\left(\frac{V_{1}}{K_{n}}-\frac{r_{1}}{n}\right)^{2}\Big|K_{n}=k\right]+\frac{M_{2,2}}{2}\sum_{k\in I_{n}}\mu_{n}(k)\E\left[\left(\frac{V_{2}}{K_{n}}-\frac{r_{2}}{n}\right)^{2}\Big|K_{n}=k\right]\nonumber\\&+M_{1,2}\sum_{k\in I_{n}}\mu_{n}(k)\left\{\E\left[\left(\frac{V_{1}}{K_{n}}-\frac{r_{1}}{n}\right)^{2}\Big|K_{n}=k\right]\right\}^{1/2}\left\{\E\left[\left(\frac{V_{2}}{K_{n}}-\frac{r_{2}}{n}\right)^{2}\Big|K_{n}=k\right]\right\}^{1/2}\label{intermediate_11_4}\\
\leqslant{}&M\sum_{k\in I_{n}}\mu_{n}(k)\left\{\frac{r_{1}(n-r_{1})(n-k)}{k n^{2}(n-1)}+\frac{r_{2}(n-r_{2})(n-k)}{k n^{2}(n-1)}+\left(\frac{r_{1}(n-r_{1})(n-k)}{k n^{2}(n-1)}\right)^{1/2}\left(\frac{r_{2}(n-r_{2})(n-k)}{k n^{2}(n-1)}\right)^{1/2}\right\}\nonumber\\
\leqslant{}&3M/4\E\left[K_{n}^{-1}\right], \quad \text{since }r_{i}(n-r_{i})\leqslant n^{2}/4 \text{ for }i\in\{1,2\}, \text{ and } (n-k)\leqslant (n-1).\nonumber
\end{align}
Here, $V_{1}$ and $V_{2}$ are as defined in Experiment~\eqref{Exp_2}, and the step indicated by \eqref{intermediate_11_4} has been obtained by deducing the joint distribution of $V_{1}$ and $V_{2}$ from the description provided in \eqref{Exp_2}, as well as an application of the Cauchy-Schwarz inequality. The next step has been obtained by \eqref{exp:2_expectations_variances}.
\end{proof}

\begin{lemma}\label{lem:hypergeometric_approx_sample_size_iid}
\sloppy For $F_{n}$ as defined in \eqref{F_{n}_defn}, and $H_{0}$ as defined in \eqref{H_{0}_defn}, we have $|F_{n}(r_{1}/n,r_{2}/n)-H_{0}(r_{1}/n,r_{2}/n)|=O(n^{-1})$ as $n\rightarrow\infty$, for any $(r_{1}/n,r_{2}/n)\in\mathcal{S}_{n}$, where $\mathcal{S}_{n}$ is as defined in \eqref{S_{n}_defn}.
\end{lemma}
\begin{proof}
For finite sequences $\{a_{i}:i\in[m]\}$ and $\{b_{i}:i\in[m]\}$ of reals, with $\max\{|a_{i}|,|b_{i}|\}\leqslant c$ for all $i\in[m]$:
\begin{align}
{}&\left|\prod_{i=1}^{m}a_{i}-\prod_{i=1}^{m}b_{i}\right|\leqslant\left|\prod_{i=1}^{m}a_{i}-b_{1}\prod_{i=2}^{m}a_{i}\right|+\left|b_{1}\prod_{i=2}^{m}a_{i}-\prod_{i=1}^{m}b_{i}\right|\leqslant |a_{1}-b_{1}|\left|\prod_{i=2}^{m}a_{i}\right|+|b_{1}|\left|\prod_{i=2}^{m}a_{i}-\prod_{i=2}^{m}b_{i}\right|\nonumber\\
\leqslant{}& c^{m-1}|a_{1}-b_{1}|+c\left|\prod_{i=2}^{m}a_{i}-\prod_{i=2}^{m}b_{i}\right|\leqslant \cdots \leqslant c^{m-1}\sum_{i=1}^{m}\left|a_{i}-b_{i}\right|.\label{general_product_difference_bound}
\end{align}
Utilizing \eqref{general_product_difference_bound} and the fact that each factor in the products below has absolute value bounded above by $1$, we have, for all $(r_{1}/n,r_{2}/n)\in\mathcal{S}_{n}$ and all $i,j\in\mathbb{N}_{0}$ with $i+j\leqslant k$ and $k\in[M]$:
\begin{align}
{}&\left|{r_{1}\choose i}{r_{2}\choose j}{n-r_{1}-r_{2}\choose k-i-j}{n\choose k}^{-1}-{k\choose i}{k-i\choose j}\left(\frac{r_{1}}{n}\right)^{i}\left(\frac{r_{2}}{n}\right)^{j}\left(1-\frac{r_{1}}{n}-\frac{r_{2}}{n}\right)^{k-i-j}\right|\nonumber\\
={}&{k\choose i}{k-i\choose j}\frac{(n-k)!n^{k}}{n!}\Bigg|\prod_{s=0}^{i-1}\left(\frac{r_{1}}{n}-\frac{s}{n}\right)\prod_{t=0}^{j-1}\left(\frac{r_{2}}{n}-\frac{t}{n}\right)\prod_{u=0}^{k-i-j-1}\left(1-\frac{r_{1}}{n}-\frac{r_{2}}{n}-\frac{u}{n}\right)-\left\{\prod_{s=0}^{i-1}\frac{r_{1}}{n}\left(1-\frac{s}{n}\right)\right\}\nonumber\\&\left\{\prod_{t=0}^{j-1}\frac{r_{2}}{n}\left(1-\frac{t+i}{n}\right)\right\}\left\{\prod_{u=0}^{k-i-j-1}\left(1-\frac{r_{1}}{n}-\frac{r_{2}}{n}\right)\left(1-\frac{u+i+j}{n}\right)\right\}\Bigg|\nonumber\\
\leqslant{}&\Bigg[\left|\prod_{s=0}^{i-1}\left(\frac{r_{1}}{n}-\frac{s}{n}\right)-\prod_{s=0}^{i-1}\left\{\frac{r_{1}}{n}\left(1-\frac{s}{n}\right)\right\}\right|\left|\prod_{t=0}^{j-1}\left(\frac{r_{2}}{n}-\frac{t}{n}\right)\prod_{u=0}^{k-i-j-1}\left(1-\frac{r_{1}}{n}-\frac{r_{2}}{n}-\frac{u}{n}\right)\right|+\left|\prod_{s=0}^{i-1}\left\{\frac{r_{1}}{n}\left(1-\frac{s}{n}\right)\right\}\right|\nonumber\\&\left|\prod_{t=0}^{j-1}\left(\frac{r_{2}}{n}-\frac{t}{n}\right)-\prod_{t=0}^{j-1}\left\{\frac{r_{2}}{n}\left(1-\frac{t+i}{n}\right)\right\}\right|\left|\prod_{u=0}^{k-i-j-1}\left(1-\frac{r_{1}}{n}-\frac{r_{2}}{n}-\frac{u}{n}\right)\right|+\Bigg|\prod_{s=0}^{i-1}\left\{\frac{r_{1}}{n}\left(1-\frac{s}{n}\right)\right\}\nonumber\\&\prod_{t=0}^{j-1}\left\{\frac{r_{2}}{n}\left(1-\frac{t+i}{n}\right)\right\}\Bigg|\left|\prod_{u=0}^{k-i-j-1}\left(1-\frac{r_{1}}{n}-\frac{r_{2}}{n}-\frac{u}{n}\right)-\prod_{u=0}^{k-i-j-1}\left(1-\frac{r_{1}}{n}-\frac{r_{2}}{n}\right)\left(1-\frac{u+i+j}{n}\right)\right|\Bigg]\nonumber\\&{k\choose i}{k-i\choose j}\prod_{w=0}^{k-1}\left(1-\frac{w}{n}\right)^{-1}\nonumber\\
\leqslant{}&\Bigg[\sum_{s=0}^{i-1}\frac{s}{n}\left(1-\frac{r_{1}}{n}\right)+\sum_{t=0}^{j-1}\left\{\frac{t}{n}\left(1-\frac{r_{2}}{n}\right)+\frac{r_{2}i}{n^{2}}\right\}+\sum_{u=0}^{k-i-j-1}\left\{\frac{u}{n}\left(\frac{r_{1}}{n}+\frac{r_{2}}{n}\right)+\frac{i+j}{n}\left(1-\frac{r_{1}}{n}-\frac{r_{2}}{n}\right)\right\}\Bigg]\nonumber\\&{k\choose i}{k-i\choose j}\prod_{w=0}^{M-1}\left(1-\frac{w}{n}\right)^{-1}\leqslant c_{M}{k\choose i}{k-i\choose j}\frac{k^{2}}{n}, \quad \text{where }c_{M}=\frac{7}{2}\prod_{w=0}^{M-1}\left(1-\frac{w}{n}\right)^{-1}.\label{hypergeometric_approximation_by_binomial}
\end{align}
From \eqref{F_{n}_defn} and \eqref{H_{0}_defn}, the fact that $0\leqslant g(x,y)\leqslant 1$ for each $(x,y)\in\mathcal{S}$ (mentioned right before \eqref{g_defn}), by triangle inequality and \eqref{hypergeometric_approximation_by_binomial}, we obtain 
\begin{equation}
\left|F_{n}\left(\frac{r_{1}}{n},\frac{r_{2}}{n}\right)-H_{0}\left(\frac{r_{1}}{n},\frac{r_{2}}{n}\right)\right|\leqslant\sum_{k=1}^{M}\mu(k)\sum_{i=0}^{k}\sum_{j=0}^{k-i}g\left(\frac{i}{k},\frac{j}{k}\right)c_{M}{k\choose i}{k-i\choose j}\frac{k^{2}}{n}\leqslant \frac{c_{M}M^{2}3^{M}}{n}=O(n^{-1}).\nonumber\qedhere
\end{equation} 
\end{proof}

\begin{lemma}\label{lem:a.s.convergence}
Let $\mathcal{A}$, and $\mathcal{A}_{n}\subset \mathcal{A}$ for each $n\in\mathbb{N}$, be compact subsets of $\mathbb{R}^{d}$. Let $Q$, taking values in $\mathcal{A}$, and $Q_{n}$, taking values in $\mathcal{A}_{n}$, for each $n\in\mathbb{N}$, be random variables, such that $\{Q_{n}\}$ converges almost surely to $Q$. Let $f_{n}:\mathcal{A}_{n}\rightarrow\mathbb{R}$ for each $n$, and $f:\mathcal{A}\rightarrow\mathbb{R}$ with $f\in\mathcal{C}^{(0)}(\mathcal{A})$. Let $\beta_{n}>0$ for each $n$ be such that 
\begin{equation}
\sup\left\{\left|f_{n}(x)-f(x)\right|:x\in \mathcal{A}_{n}\right\}\leqslant \beta_{n} \text{ for each } n\in\mathbb{N} \quad \text{and} \quad \beta_{n}\rightarrow 0 \text{ as } n\rightarrow\infty.\label{uniform_convergence_criterion} 
\end{equation}
Then the Cesaro average
\begin{equation}
\frac{1}{n}\sum_{i=1}^{n}f_{i}(Q_{i}) \convas f(Q) \text{ as }n\rightarrow\infty.\label{Cesaro_convergence}
\end{equation} 

Alternatively, let $Q$ as well as $Q_{n}$, for each $n\in\mathbb{N}$, take values in $\mathcal{A}$, and let $Q_{n}\convas Q$. Let $f_{n}:\mathcal{A}\rightarrow\mathbb{R}$ and $f_{n}\in\mathcal{C}^{(0)}(\mathcal{A})$ for each $n\in\mathbb{N}$, and $f:\mathcal{A}\rightarrow\mathbb{R}$, such that $\{f_{n}\}$ converges uniformly to $f$ on $\mathcal{A}$. Then, once again, the conclusion of \eqref{Cesaro_convergence} holds. A special case of this scenario is where $f_{n}\equiv f$ for each $n\in\mathbb{N}$.
\end{lemma}
\begin{proof}
Let $\Omega$ be the sample space on which $Q$, and $Q_{n}$ for each $n$, have been defined, and let $\Omega'\subset\Omega$ consist of all sample points $\omega$ such that $Q_{n}(\omega)\rightarrow Q(\omega)$ as $n\rightarrow\infty$. Evidently, $\Prob[\Omega']=1$. Given $\delta>0$, for each $\omega\in\Omega'$ there exists $N_{\delta}(\omega)\in\mathbb{N}$ such that $\left|\left|Q_{n}(\omega)-Q(\omega)\right|\right|_{2}<\delta$ for all $n\geqslant N_{\delta}(\omega)$. 

In the first part of the statement of Lemma~\ref{lem:a.s.convergence}, since $f\in\mathcal{C}^{(0)}(\mathcal{A})$ and $\mathcal{A}$ is compact, $f$ must be uniformly continuous throughout $\mathcal{A}$. Thus, given $\epsilon>0$, there exists $\delta_{f}(\epsilon)>0$ such that $\left|f(x)-f(y)\right|<\epsilon/3$ whenever $x,y\in \mathcal{A}$ and $||x-y||_{2}<\delta_{f}(\epsilon)$. Next, since $\beta_{n}\rightarrow 0$ as $n\rightarrow \infty$, given any $\epsilon>0$, there exists $R_{\epsilon}\in\mathbb{N}$ such that $\beta_{n}<\epsilon/3$ for all $n\geqslant R_{\epsilon}$. Fixing $\omega\in\Omega'$ and $\epsilon>0$, setting $N_{0}=\max\{N_{\delta_{f}(\epsilon)}(\omega),R_{\epsilon}\}$, defining $C=\sum_{i\in[N_{0}-1]}\left|f_{i}(Q_{i}(\omega))-f(Q(\omega))\right|$ (which remains constant as $n$ grows), we have, for $n\geqslant \max\{N_{0},3C/\epsilon\}$:
\begin{align}
\left|\frac{1}{n}\sum_{i=1}^{n}f_{i}(Q_{i}(\omega))-f(Q(\omega))\right|\leqslant{}&\frac{1}{n}\sum_{i=1}^{N_{0}-1}\left|f_{i}(Q_{i}(\omega))-f(Q(\omega))\right|+\frac{1}{n}\sum_{i=N_{0}}^{n}\left|f_{i}(Q_{i}(\omega))-f(Q(\omega))\right|\nonumber\\
\leqslant{}&\frac{C}{n}+\frac{1}{n}\sum_{i=N_{0}}^{n}\left|f_{i}(Q_{i}(\omega))-f(Q_{i}(\omega))\right|+\frac{1}{n}\sum_{i=N_{0}}^{n}\left|f(Q_{i}(\omega))-f(Q(\omega))\right|\nonumber\\
\leqslant{}&\frac{C}{n}+\frac{1}{n}\sum_{i=N_{0}}^{n}\beta_{i}+\frac{(n-N_{0}+1)\epsilon}{3n}<\frac{\epsilon}{3}+\frac{2(n-N_{0}+1)\epsilon}{3n}<\epsilon.\nonumber
\end{align}
Since this is true for each $\omega\in\Omega'$, we conclude that $n^{-1}\sum_{i=1}^{n}f_{i}(Q_{i})$ converges to $f(Q)$ almost surely.

The second part of Lemma~\ref{lem:a.s.convergence} is a special case of the first: we set $\mathcal{A}_{n}=\mathcal{A}$ for each $n$, and $\{f_{n}\}$ converging uniformly to $f$ on $\mathcal{A}$ ensures that \eqref{uniform_convergence_criterion} holds. Since $f_{n}\in\mathcal{C}^{(0)}(\mathcal{A})$ and $\{f_{n}\}$ converges uniformly to $f$ on $\mathcal{A}$, we have $f\in\mathcal{C}^{(0)}(\mathcal{A})$ as well. Thus, all criteria mentioned in the first part of the statement of Lemma~\ref{lem:a.s.convergence} are satisfied, so that \eqref{Cesaro_convergence} is also true here.
\end{proof}

\begin{lemma}\label{lem:stronger_convergence_criteria_zhang}
When any one of \eqref{thm:main_6:regime_1:series_criteria}, \eqref{thm:main_6:regime_2:series_criteria} and \eqref{thm:main_6:regime_3.2:series_criteria} holds, \eqref{thm:main_2:C^{2}_series_convergence} holds. When \eqref{thm:main_6:regime_3.1:series_criteria} holds, \eqref{thm:main_2:C^{1}_series_convergence} holds.
\end{lemma}
\begin{proof}
Let $a_{n}=\E[K_{n}^{-1}]$ for $n\geqslant N$. We set $b_{N-1}=0$ and $b_{n}=\sum_{i=N}^{n}a_{i}$ for $n\geqslant N$, so that the partial sums of the series in \eqref{thm:main_2:C^{2}_series_convergence} can be expressed as:
\begin{align}
\sum_{i=N}^{n}(i+1)^{-1}a_{i}=\sum_{i=N}^{n}(i+1)^{-1}(b_{i}-b_{i-1})=(n+1)^{-1}b_{n}+\sum_{i=N}^{n-1}(i+1)^{-1}(i+2)^{-1}b_{i}.\label{thm:main_2:C^{2}_series_partial_sum_1}
\end{align}
When the first criterion of \eqref{thm:main_6:regime_1:series_criteria} holds, $b_{n}\leqslant C\sqrt{n/\log n}$ for $n\geqslant N$, for some $C>0$, so that from \eqref{thm:main_2:C^{2}_series_partial_sum_1}:
\begin{align}
\sum_{i=N}^{n}(i+1)^{-1}a_{i}\leqslant C(n\log n)^{-1/2}+C\sum_{i=N}^{n-1}(i+1)^{-1}(i+2)^{-1}\sqrt{i/\log i}\leqslant C(n\log n)^{-1/2}+C\sum_{i=N}^{n-1}i^{-3/2}(\log i)^{-1/2},\nonumber
\end{align}
which converges as $n\rightarrow\infty$. When \eqref{thm:main_6:regime_2:series_criteria} holds, $b_{n}\leqslant C n^{1-\rho-\epsilon}$ for $n\geqslant N$, for some $C>0$. From \eqref{thm:main_2:C^{2}_series_partial_sum_1}:
\begin{align}
\sum_{i=N}^{n}(i+1)^{-1}a_{i}\leqslant Cn^{-\rho-\epsilon}+C\sum_{i=N}^{n-1}(i+1)^{-1}(i+2)^{-1}i^{1-\rho-\epsilon}\leqslant Cn^{-\rho-\epsilon}+C\sum_{i=N}^{n-1}i^{-1-\rho-\epsilon},\nonumber
\end{align}
which converges as $n\rightarrow\infty$. When \eqref{thm:main_6:regime_3.2:series_criteria} holds, $b_{n}\leqslant C\sqrt{n}$ for $n\geqslant N$, for some $C>0$. From \eqref{thm:main_2:C^{2}_series_partial_sum_1}:
\begin{align}
\sum_{i=N}^{n}(i+1)^{-1}a_{i}\leqslant Cn^{-1/2}+C\sum_{i=N}^{n-1}(i+1)^{-1}(i+2)^{-1}\sqrt{i}\leqslant Cn^{-1/2}+C\sum_{i=N}^{n-1}i^{-3/2},\label{thm:main_2:C^{2}_series_partial_sum_1_bound_C^{2}}
\end{align}
which converges as $n\rightarrow\infty$. 

Next, let $a_{n}=(n+1)^{-1/2}\E[K_{n}^{-1}]$ for $n\geqslant N$, and we set $b_{N-1}=0$, while $b_{n}=\sum_{i=N}^{n}a_{i}$ for each $n\geqslant N$. Now the partial sums of the series in \eqref{thm:main_2:C^{2}_series_convergence} can be expressed as:
\begin{align}
\sum_{i=N}^{n}(i+1)^{-1/2}a_{i}=\sum_{i=N}^{n}(i+1)^{-1/2}(b_{i}-b_{i-1})=\frac{b_{n}}{\sqrt{n+1}}+\sum_{i=N}^{n-1}\frac{b_{i}}{\sqrt{(i+1)(i+2)}(\sqrt{i+1}+\sqrt{i+2})}.\label{thm:main_2:C^{2}_series_partial_sum_2}
\end{align}
When the second criterion of \eqref{thm:main_6:regime_1:series_criteria} holds, $b_{n}\leqslant C\sqrt{\log n}$ for all $n\geqslant N$, for some $C>0$, so that from \eqref{thm:main_2:C^{2}_series_partial_sum_2}:
\begin{align}
\sum_{i=N}^{n}(i+1)^{-1/2}a_{i}\leqslant C\frac{\sqrt{\log n}}{\sqrt{n+1}}+C\sum_{i=N}^{n-1}\frac{\sqrt{\log i}}{\sqrt{(i+1)(i+2)}(\sqrt{i+1}+\sqrt{i+2})}\leqslant C\frac{\sqrt{\log n}}{\sqrt{n+1}}+\frac{C}{2}\sum_{i=N}^{n-1}\frac{\sqrt{\log i}}{i^{3/2}},\nonumber
\end{align}
which converges as $n\rightarrow\infty$. This completes the proof that any one of \eqref{thm:main_6:regime_1:series_criteria}, \eqref{thm:main_6:regime_2:series_criteria} and \eqref{thm:main_6:regime_3.2:series_criteria} implies \eqref{thm:main_2:C^{2}_series_convergence}.

Finally, let $a_{n}=\min\{\omega(\nabla g;(\E[K_{n}^{-1}])^{1/2})(\E[K_{n}^{-1}])^{1/2},\omega(\nabla g;\E[K_{n}^{-1}]/\E[K_{n}^{-1/2}])\E[K_{n}^{-1/2}]\}$, and $b_{n}=\sum_{i=N}^{n}a_{i}$ for $n\geqslant N$, with $b_{N-1}=0$. Then, the partial sums of the series in \eqref{thm:main_2:C^{1}_series_convergence} take the same form as shown in \eqref{thm:main_2:C^{2}_series_partial_sum_1}. When \eqref{thm:main_6:regime_3.1:series_criteria} holds, $b_{n}\leqslant C\sqrt{n}$ for $n\geqslant N$, for some $C>0$, so that the inequality in \eqref{thm:main_2:C^{2}_series_partial_sum_1_bound_C^{2}} is also true here, implying \eqref{thm:main_2:C^{1}_series_convergence}.
\end{proof}

\bibliography{ERW_bib}
\end{document}